\documentclass[11pt]{amsart}
\pdfoutput=1

\usepackage{amssymb,mathtools}
\usepackage{microtype}
\usepackage{mathtools}
\usepackage{tikz-cd}

\usepackage{booktabs}

\usepackage[margin=1in, marginparwidth=.75in]{geometry}

\usepackage[numbers]{natbib}

\usepackage{hyperref}
\hypersetup{
  pdftitle={Hybridization and postprocessing in finite element exterior calculu
s},
  pdfauthor={Gerard Awanou, Maurice Fabien, Johnny Guzm\'an, and Ari Stern},
  pdfsubject={MSC 2010: 65N30 (58A14)},
  bookmarksopen=false
}

\usepackage[capitalize,nameinlink]{cleveref}

\newcommand\normal{{\widehat{n}}}

\DeclareMathSymbol{\star}{\mathord}{letters}{"3F}

\newcommand{\lVERT}{\lvert\mkern-2mu\lvert\mkern-2mu\lvert}
\newcommand{\rVERT}{\rvert\mkern-2mu\rvert\mkern-2mu\rvert}

\makeatletter
\g@addto@macro\bfseries{\boldmath}
\makeatother

\theoremstyle{plain} 
\newtheorem{theorem}{Theorem}[section]
\newtheorem{lemma}[theorem]{Lemma}
\newtheorem{corollary}[theorem]{Corollary}
\newtheorem{proposition}[theorem]{Proposition}

\theoremstyle{definition}
\newtheorem{definition}[theorem]{Definition}
\newtheorem{example}[theorem]{Example}
\newtheorem{assumption}{Assumption}

\theoremstyle{remark}
\newtheorem{remark}[theorem]{Remark}

\begin{document}

\title[Hybridization and postprocessing in finite element exterior calculus]{%
  Hybridization and postprocessing\\in finite element exterior calculus}

\author{Gerard Awanou}
\address{Department of Mathematics, Statistics, and Computer Science\\
  University of Illinois, Chicago}
\email{awanou@uic.edu}

\author{Maurice Fabien}
\author{Johnny Guzm\'an}
\address{Division of Applied Mathematics, Brown University}
\email{fabien@brown.edu}
\email{johnny\_guzman@brown.edu}

\author{Ari Stern}
\address{Department of Mathematics and Statistics, Washington University in St.~Louis}
\email{stern@wustl.edu}

\begin{abstract}
  We hybridize the methods of finite element exterior calculus for the
  Hodge--Laplace problem on differential $k$-forms in
  $\mathbb{R}^n$. In the cases $ k = 0 $ and $ k = n $, we recover
  well-known primal and mixed hybrid methods for the scalar Poisson
  equation, while for $ 0 < k < n $, we obtain new hybrid finite
  element methods, including methods for the vector Poisson equation
  in $ n = 2 $ and $ n = 3 $ dimensions. We also generalize
  \citeauthor{Stenberg1991} postprocessing from $ k = n $ to arbitrary
  $k$, proving new superconvergence estimates. Finally, we discuss how
  this hybridization framework may be extended to include nonconforming
  and hybridizable discontinuous Galerkin methods.
\end{abstract}

\vspace*{-\baselineskip}
\maketitle
\vspace*{-\baselineskip}

\section{Introduction}

Finite element exterior calculus (FEEC) is a powerful framework that
unifies the analysis of several families of conforming finite element
methods for problems involving Laplace-type operators
(\citet*{ArFaWi2006,ArFaWi2010,Arnold2018}). These include the classic
``continuous Galerkin'' Lagrange finite element method and the
Raviart--Thomas (RT) \citep{RaTh1977} and Brezzi--Douglas--Marini
(BDM) \citep{BrDoMa1985} mixed methods for the scalar Poisson
equation, as well as mixed methods based on N\'ed\'elec
elements~\citep{Nedelec1980,Nedelec1986} for the $2$- and
$3$-dimensional vector Poisson equation. In FEEC, these are all seen
as finite element methods for the Hodge--Laplace operator on
differential $k$-forms in $ \mathbb{R}^n $, where scalar fields are
identified with $0$- and $n$-forms and vector fields with $1$- and
$ (n-1) $-forms.

In this paper, we hybridize FEEC for arbitrary dimension $n$ and form
degree $k$. That is, we construct hybrid finite element methods using
discontinuous spaces of differential forms, enforcing continuity and
boundary conditions using Lagrange multipliers on the element
boundaries. The solutions agree with those of the original, non-hybrid
FEEC methods, and the Lagrange multipliers are seen to correspond to
weak tangential and normal traces. This hybrid formulation enables
static condensation: since only the Lagrange multipliers are globally
coupled, the remaining internal degrees of freedom can be eliminated
using an efficient local procedure, and the resulting Schur complement
system can be substantially smaller than the original one. We also
present a generalization of Stenberg postprocessing
\citep{Stenberg1991}, which for $ 0 < k < n $ is shown to give new
improved estimates.

The special cases $ k = 0 $ and $ k = n $ are shown to recover known
results on hybridization and postprocessing for the scalar Poisson
equation. In particular, the case $ k = n $ corresponds to the
hybridized RT \citep{ArBr1985} and BDM \citep{BrDoMa1985} methods, and
the postprocessing procedure is precisely that of
\citet{Stenberg1991}. The case $ k = 0 $ corresponds to the more
recent hybridization of the continuous Galerkin method by
\citet*{CoGoWa2007}.

The hybrid and postprocessing schemes in the remaining cases
$ 0 < k < n $ are new and, to the best of our knowledge, have not
appeared in the literature even for the vector Poisson equation when
$ n = 2 $ or $ n = 3$.  In particular, the hybridization of
N\'ed\'elec edge elements is different from that in \citet{CoGo2005b}:
here, the Lagrange multipliers are simply traces of standard elements,
rather than living in a space of ``jumps.'' We expect these new
methods to be especially useful in computational electromagnetics,
where N\'ed\'elec elements are ubiquitous and the differential forms
point of view has provided significant insight
(cf.~\citet{Hiptmair2002}).

While we restrict our attention primarily to hybrid methods for
conforming simplicial meshes, we remark that the framework developed
here has the potential to be applied to other types of domain
decomposition methods, including methods on cubical meshes,
nonconforming meshes, mortar methods, etc. We also discuss briefly how
the unified hybridization framework of \citet*{CoGoLa2009}, which
includes hybridizable discontinuous Galerkin (HDG) methods, may also
be generalized to the Hodge--Laplace problem for $ 0 < k < n $.

\subsection{Why hybridize?} There are several theoretical and
practical benefits of hybridization:
\begin{itemize}
\item \emph{additional information about solutions}: The Lagrange
  multiplier functions often correspond to weak boundary traces of
  solution components, even though the numerical solution may not be
  regular enough for a trace to exist in the usual sense (e.g., the
  trace of an $ L ^2 $ function or normal derivative of an $ H ^1 $
  function).

\item \emph{static condensation}: Degrees of freedom for discontinuous
  function spaces can be locally eliminated. The resulting Schur
  complement only involves boundary degrees of freedom for the
  Lagrange multipliers, so it can be substantially smaller than the
  original global problem.

\item \emph{local postprocessing and superconvergence}: The numerical
  solution may be efficiently ``postprocessed'' by using the boundary
  traces to solve a local problem on each element, resulting in an
  improved approximation compared to the original solution.
  
\end{itemize}

Seminal work on hybridization of mixed finite element methods was done
by \citet{FraeijsDeVeubeke1965}. For the scalar Poisson equation, the
RT method was hybridized in this manner by \citet{ArBr1985}, who
introduced the notion of postprocessing.  Hybridization and
postprocessing were also discussed in the original paper introducing
the BDM method \citep{BrDoMa1985}, and an interesting characterization
of the Lagrange multipliers for the hybridized RT and BDM methods
appears in \citet{CoGo2004}. A refined local postprocessing procedure
for mixed methods, which can be applied with or without hybridization,
was given by \citet{Stenberg1991}; see also \citet{GaNo1989}, who
discovered this independently (cf.~\citep[eqs.~4.14--4.15]{GaNo1989}),
as well as \citet{BrXu1989}.

More recently, \citet*{CoGoWa2007} hybridized the continuous Galerkin
method, using an approach similar to the ``three-field domain
decomposition method'' of \citet{BrMa1994}, and showed that static
condensation yields the same condensed system as that obtained by the
original, non-hybrid static condensation procedure of
\citet{Guyan1965}. Even more recently, \citet*{CoGoLa2009} introduced
an important unified hybridization framework that includes the above
methods, as well as nonconforming and HDG methods, for the scalar
Poisson equation. A survey of historical and recent developments
appears in \citet{Cockburn2016}.

\subsection{Organization of the paper}

The paper is organized as follows:
\begin{itemize}
\item \cref{sec:background} recalls the basic machinery and
  terminology of differential forms, the Hodge--Laplace problem, and
  FEEC. This includes a discussion of tangential and normal traces,
  which play an important role throughout the paper.

\item \cref{sec:dd} presents a domain decomposition of the
  Hodge--Laplace problem. The variational form of this problem
  involves broken spaces of differential forms, along with boundary
  traces that act as Lagrange multipliers enforcing interelement
  continuity and boundary conditions.

\item \cref{sec:hybrid} develops hybrid finite element methods for
  the Hodge--Laplace problem, based on the domain-decomposed
  variational principle from the previous section. We prove that these
  are hybridized versions of the FEEC methods, show how static
  condensation can be used to reduce the size of the global system,
  and develop error estimates for the hybrid variables.

\item \cref{sec:postprocessing} generalizes the postprocessing
  procedure of \citet{Stenberg1991} from $ k = n $ to arbitrary
  $k$. This procedure only uses the statically condensed variables, so
  it can be applied immediately after solving the condensed system, or
  it can be applied to solutions obtained by ordinary finite element
  methods without hybridization. In addition to known superconvergence
  results for $ k = n $, we give new improved error estimates for
  $ k < n $.

\item \cref{sec:examples} gives concrete illustrations of the hybrid
  and postprocessing methods when $ n = 3 $, using the language of
  vector calculus and classic families of finite elements.

\item \cref{sec:numerical} presents numerical experiments, confirming
  the error estimates of \cref{sec:hybrid,sec:postprocessing}.

\item Finally, \cref{sec:hdg} presents an extension of the framework
  of \citet*{CoGoLa2009}, whereas the previous sections only address
  conforming methods. This lays the groundwork for hybridization of
  nonconforming and discontinuous Galerkin methods for FEEC, although
  we postpone the analysis of such methods for future work.
\end{itemize} 

\section{Background: differential forms and finite element exterior calculus}
\label{sec:background}

In this section, we quickly recall the exterior calculus of
differential forms, the Hodge--Laplace problem, and FEEC, in order to
lay the foundation and fix the notation for the subsequent
sections. We refer to \citet*{ArFaWi2006,ArFaWi2010,Arnold2018}, and
references therein for a comprehensive treatment. We also discuss
tangential and normal traces of differential forms, which will play an
important role in domain decomposition and hybridization. Our
treatment of these traces follows that in \citet{Weck2004} (see also
\citet{KuAu2012}), which extended work of
\citet*{BuCi2001a,BuCi2001b,BuCoSh2002} for vector fields in
$\mathbb{R}^3$.

\subsection{Exterior calculus of differential forms}

Let $ \Omega \subset \mathbb{R}^n $ be a bounded Lipschitz domain, and
denote by $ \Lambda ^k (\Omega) $ the space of smooth differential
$k$-forms on $\Omega$, where $ k = 0, \ldots, n $. We assume that the
reader is familiar with the following basic operations of exterior
calculus:
\begin{itemize}
\item the wedge product
  $ \wedge \colon \Lambda ^k (\Omega) \times \Lambda ^\ell (\Omega)
  \rightarrow \Lambda ^{ k + \ell } (\Omega) $,

\item the (Euclidean) Hodge star isomorphism
  $ \star \colon \Lambda ^k (\Omega) \rightarrow \Lambda ^{ n - k }
  (\Omega) $,
\item the exterior derivative $ \mathrm{d} \colon \Lambda ^k (\Omega) \rightarrow \Lambda ^{ k + 1 } (\Omega) $,
\item the codifferential $ \delta \coloneqq ( - 1 ) ^k \star ^{-1} \mathrm{d} \star \colon \Lambda ^k (\Omega) \rightarrow \Lambda ^{ k -1 } (\Omega) $,
\item the Hodge--Laplace operator
  $ L \coloneqq \mathrm{d} \delta + \delta \mathrm{d} \colon \Lambda
  ^k (\Omega) \rightarrow \Lambda ^k (\Omega) $.
\end{itemize}
These are graded operators, but we suppress the form degree for
notational simplicity, e.g., writing $\mathrm{d}$ rather than
$ \mathrm{d} ^k $. From the Leibniz rule for $ \mathrm{d} $ and
definition of $ \delta $, we have the important identity
\begin{equation}
  \label{eqn:d_delta_leibniz}
  \mathrm{d} ( \tau  \wedge \star v ) =  \mathrm{d} \tau \wedge \star v - \tau  \wedge \star \delta v ,
\end{equation}
where $ \tau \in \Lambda ^{ k -1 } (\Omega) $ and
$ v \in \Lambda ^k (\Omega) $.

The Hilbert space $ L ^2 \Lambda ^k (\Omega) $ is the completion of
$ \Lambda ^k (\Omega) $ with respect to the $ L ^2 $ inner product
$ ( v, w ) _\Omega \coloneqq \int _\Omega v \wedge \star w $, whose
associated norm is denoted $ \lVert \cdot \rVert _\Omega $. Taking
$ \mathrm{d} $ in the sense of distributions allows it to be extended
to a closed, densely defined operator with domain
\begin{equation*}
  H \Lambda ^k (\Omega) \coloneqq \bigl\{ v \in L ^2 \Lambda ^k (\Omega) : \mathrm{d} v \in L ^2 \Lambda ^{ k + 1 } (\Omega) \bigr\} ,
\end{equation*}
which is itself a Hilbert space with the graph inner product
$ ( v, w ) _{ H \Lambda ^k (\Omega) } \coloneqq ( v , w ) _\Omega + (
\mathrm{d} v , \mathrm{d} w ) _\Omega $. The subspace
$ \mathring{ H } \Lambda ^k (\Omega) \subset H \Lambda ^k (\Omega) $
is defined to be the closure of $ C _0 ^\infty \Lambda ^k (\Omega) $,
the space of smooth $k$-forms with compact support in $\Omega$.
Likewise, $\delta$ may be extended to a closed, densely defined
operator with domain
\begin{equation*}
  H ^\ast \Lambda ^k (\Omega) \coloneqq \bigl\{ v \in L ^2 \Lambda ^k (\Omega) : \delta v \in L ^2 \Lambda ^{ k -1 } (\Omega) \bigr\} = \star H \Lambda ^{ n - k } (\Omega) ,
\end{equation*} 
which is a Hilbert space with the graph inner product
$ ( v, w ) _{ H ^\ast \Lambda ^k (\Omega) } \coloneqq ( v, w ) _\Omega
+ ( \delta v , \delta w ) _\Omega $, and the subspace
$ \mathring{ H } ^\ast \Lambda ^k (\Omega) = \star \mathring{ H }
\Lambda ^{ n - k } (\Omega) \subset H ^\ast \Lambda ^k (\Omega) $ is
the closure of $ C _0 ^\infty \Lambda ^k (\Omega) $.

\subsection{Tangential and normal traces}
\label{sec:traces}
The restriction of a differential form to the boundary
$ \partial \Omega $ is encoded in a pair of differential forms on
$ \partial \Omega $, called the tangential trace and normal
trace. This is analogous to decomposing a vector field into its
tangential and normal components at the boundary.

We begin with the case of smooth differential forms, where the
boundary $ \partial \Omega $ is also smooth. The \emph{trace map}
$ \operatorname{tr} \colon \Lambda ^k ( \Omega ) \rightarrow \Lambda
^k ( \partial \Omega ) $ is defined to be the pullback of $k$-forms by
the inclusion $ \partial \Omega \hookrightarrow \Omega $, i.e.,
$ \operatorname{tr} v \in \Lambda ^k ( \partial \Omega ) $ is just the
restriction of $ v \in \Lambda ^k (\Omega) $ to vectors tangent to the
boundary. Denote the Hodge star on $ \partial \Omega $ by
$ \widehat{ \star } $ and the associated $ L ^2 $ inner product by
$ \langle \cdot , \cdot \rangle _{ \partial \Omega } $.

\begin{definition}[tangential and normal traces]
  Given $ v \in \Lambda ^k (\Omega) $,
  \begin{equation*}
    v ^{\mathrm{tan}} \coloneqq \operatorname{tr} v \in \Lambda ^k ( \partial \Omega ) , \qquad v ^{\mathrm{nor}} \coloneqq \widehat{ \star } ^{-1} \operatorname{tr} \star v \in \Lambda ^{k-1} ( \partial \Omega ) .
  \end{equation*} 
\end{definition}

These definitions allow a particularly elegant expression of the
integration by parts formula for differential forms. The following
result is standard, but the proof is short and illuminates the
definition of the normal trace.

\begin{proposition}
  If $ \tau \in \Lambda ^{ k -1 } (\Omega) $ and
  $ v \in \Lambda ^k (\Omega) $, then we have the integration by parts
  formula
  \begin{equation}
    \label{eqn:ibp}
    \langle \tau ^{\mathrm{tan}} , v ^{\mathrm{nor}} \rangle _{ \partial \Omega } = ( \mathrm{d} \tau , v ) _\Omega - ( \tau, \delta v ) _\Omega .
  \end{equation} 
\end{proposition}

\begin{proof}
  Using the definitions of $ \tau ^{\mathrm{tan}} $ and
  $ v ^{\mathrm{nor}} $, we calculate
  \begin{equation*}
    \langle \tau ^{\mathrm{tan}} , v ^{\mathrm{nor}} \rangle _{ \partial \Omega } = \int _{ \partial \Omega  } \tau ^{\mathrm{tan}} \wedge \widehat{ \star } v ^{\mathrm{nor}} = \int _{ \partial \Omega } \operatorname{tr} \tau \wedge \operatorname{tr} \star v = \int _{ \partial \Omega } \operatorname{tr} ( \tau  \wedge \star v ) = \int _\Omega \mathrm{d} ( \tau  \wedge \star v ) ,
  \end{equation*}
  where the last step uses Stokes' theorem. Applying
  \eqref{eqn:d_delta_leibniz} completes the proof.
\end{proof}

\begin{table}
  \centering
  \renewcommand\arraystretch{1.25}
  \begin{tabular}{clll}
    \toprule
    $k$ & proxy field & tangential trace & normal trace\\
    \midrule $0$ & $ \varphi \in C ^\infty (\Omega) $ & $ \varphi \rvert _{ \partial \Omega } $ & 0 \\
    $1$ & $ v \in C ^\infty ( \Omega , \mathbb{R}^3  ) $ & $ v \rvert _{ \partial \Omega } - ( v \cdot \normal ) \normal  $ & $ v \cdot \normal $\\
    $2$ & $ w \in C ^\infty ( \Omega , \mathbb{R}^3  ) $ & $ ( w \cdot \normal ) \normal $ & $ w \times \normal $ \\
    $3$ & $ \psi \in C ^\infty (\Omega) $ & $0$ & $ \psi \normal $\\
    \bottomrule
  \end{tabular}
  \bigskip 
  \caption{Tangential and normal traces of differential forms on
    $ \Omega \subset \mathbb{R}^3 $, in terms of scalar and vector
    proxy fields.\label{tab:traces}}
\end{table}

An equivalent description of tangential and normal traces uses the
outer unit normal vector field $\normal$ and its associated
$1$-form $ \normal ^\flat = \normal _i \,\mathrm{d}x ^i
$. Letting $ \iota _\normal $ denote the interior product (or
contraction) with $\normal$, the Leibniz rule for this operator
gives the identity
\begin{equation*}
  v \rvert _{ \partial \Omega } = \iota _\normal ( \normal ^\flat \wedge v ) + \normal ^\flat \wedge ( \iota _\normal v ) .
\end{equation*}
We may then identify $ v ^{\mathrm{tan}} $ with the $k$-form
$ \iota _\normal ( \normal ^\flat \wedge v )$ and
$ v ^{\mathrm{nor}} $ with the $ ( k -1 )$-form
$ \iota _\normal v $.  When $ \Omega \subset \mathbb{R} ^3 $, the
correspondence of these traces to scalar and vector proxy fields is
given in \cref{tab:traces}, using the proxy operations for
$ \iota _{ \normal } $ and $ \normal ^\flat \wedge {} $, and
\eqref{eqn:ibp} recovers the familiar integration by parts formulas of
vector calculus.

\citet{Weck2004} showed that it is possible to extend the tangential
and normal traces so that a weak version of \eqref{eqn:ibp} holds for
$ \tau \in H \Lambda ^{ k -1 } (\Omega) $ and
$ v \in H ^\ast \Lambda ^k (\Omega) $, where $ \partial \Omega $ is
only assumed to be Lipschitz. We denote the trace spaces in which
$ \tau ^{\mathrm{tan}} $ and $ v ^{\mathrm{nor}} $ live by
$ \widehat{ H } \Lambda ^{k-1, \mathrm{tan}} (\partial \Omega )$ and
$ \widehat{ H } ^\ast \Lambda ^{k-1, \mathrm{nor}} (\partial \Omega
)$, respectively. These are generally subspaces of
$ H ^{ - 1/2 } \Lambda ^{ k -1 } ( \partial \Omega )$, but not
necessarily of $ L ^2 \Lambda ^{ k -1 } ( \partial \Omega )$, so
$ \langle \cdot , \cdot \rangle _{ \partial \Omega } $ should be
interpreted as a duality pairing extending the $ L ^2 $ inner product
on $ \partial \Omega $ \citep[Theorem 8]{Weck2004}. See
\citet{KuAu2012} for an excellent account of \citeauthor{Weck2004}'s
results and some concrete applications to
electromagnetics. \citet*{MiMiSh2008} obtain comparable results by
extending the alternative approach using $ \iota _{ \normal } $ and
$ \normal ^\flat \wedge {} $ described above.

The definitions of
$ \widehat{ H } \Lambda ^{k-1, \mathrm{tan}} (\partial \Omega )$ and
$ \widehat{ H } ^\ast \Lambda ^{k-1, \mathrm{nor}} (\partial \Omega )$
are somewhat technical, but thankfully, we may make use of
\citep[Theorems 5 and 7]{Weck2004}, which give isomorphisms
\begin{equation}
  \label{eqn:trace_quotient}
  \widehat{ H } \Lambda ^{k-1, \mathrm{tan}} (\partial \Omega ) \cong H \Lambda ^{ k -1 } (\Omega) / \mathring{ H } \Lambda ^{ k -1 } (\Omega) , \qquad \widehat{ H } ^\ast \Lambda ^{k-1, \mathrm{nor}} (\partial \Omega
) \cong H ^\ast \Lambda ^k (\Omega) / \mathring{ H } ^\ast \Lambda ^k (\Omega) .
\end{equation}
Therefore, we may treat the trace spaces as quotient spaces, equipped
with the quotient norms
\begin{equation*}
  \lVert \widehat{ \tau } ^{\mathrm{tan}} \rVert _{\mathrm{tan}, \partial \Omega } \coloneqq \inf \bigl\{ \lVert \tau \rVert _{ H \Lambda ^{ k -1 } (\Omega) } : \tau  ^{\mathrm{tan}} = \widehat{ \tau } ^{\mathrm{tan}} \bigr\} , \qquad   \lVert \widehat{ v } ^{\mathrm{nor}} \rVert _{\mathrm{nor}, \partial \Omega } \coloneqq \inf \bigl\{ \lVert v \rVert _{ H ^\ast \Lambda ^k (\Omega) } : v ^{\mathrm{nor}} = \widehat{ v } ^{\mathrm{nor}} \bigr\} .
\end{equation*}
These generalize the ``minimum energy extension'' quotient norms
discussed in \citet*[Section 2]{CaDeGo2016} for $ H ^1 $,
$ H ( \operatorname{curl} ) $, and $ H(\operatorname{div})$ traces in
$\mathbb{R}^3$. The next result, relating these norms to the duality
pairing, is a straightforward generalization of \citep[Lemma
2.2]{CaDeGo2016}.

\begin{lemma}
  \label{lem:duality_isometry}
  For all
  $ \widehat{ \tau } ^{\mathrm{tan}} \in \widehat{ H } \Lambda ^{k-1,
    \mathrm{tan}} (\partial \Omega )$ and
  $ \widehat{ v } ^{\mathrm{nor}} \in \widehat{ H } ^\ast \Lambda
  ^{k-1, \mathrm{nor}} (\partial \Omega ) $, we have the equalities
  \begin{equation*}
    \lVert \widehat{ \tau } ^{\mathrm{tan}} \rVert _{\mathrm{tan}, \partial \Omega } = \sup _{ \widehat{ v } ^{\mathrm{nor}} \neq 0 } \frac{ \langle \widehat{ \tau } ^{\mathrm{tan}} , \widehat{ v } ^{\mathrm{nor}} \rangle _{ \partial \Omega } }{ \lVert \widehat{ v } \rVert _{\mathrm{nor}, \partial \Omega } } , \qquad
    \lVert \widehat{ v } ^{\mathrm{nor}} \rVert _{\mathrm{nor}, \partial \Omega } = \sup _{ \widehat{ \tau } ^{\mathrm{tan}} \neq 0 } \frac{ \langle \widehat{ \tau } ^{\mathrm{tan}} , \widehat{ v } ^{\mathrm{nor}} \rangle _{ \partial \Omega } }{ \lVert \widehat{ \tau } \rVert _{\mathrm{tan}, \partial \Omega } }.
  \end{equation*}
  That is, the duality isomorphisms
  $ \widehat{ \tau } ^{\mathrm{tan}} \mapsto \langle \widehat{ \tau }
  ^{\mathrm{tan}} , \cdot \rangle _{ \partial \Omega } $ and
  $ \widehat{ v } ^{\mathrm{nor}} \mapsto \langle \cdot , \widehat{ v
  } ^{\mathrm{nor}} \rangle _{ \partial \Omega } $ are isometries.
\end{lemma}

\begin{proof}
  Given $ \widehat{ \tau } ^{\mathrm{tan}} $, the Riesz representation
  theorem gives a unique $ w \in H ^\ast \Lambda ^k (\Omega) $ such
  that
  \begin{equation*}
    ( w, v ) _\Omega + ( \delta w , \delta v ) _\Omega = \langle  \widehat{ \tau } ^{\mathrm{tan}} , v ^{\mathrm{nor}} \rangle _{ \partial \Omega } , \quad \forall v \in H ^\ast \Lambda ^k (\Omega) ,
  \end{equation*}
  so $ w + \mathrm{d} \delta w = 0 $ with
  $ ( -\delta w ) ^{\mathrm{tan}} = \widehat{ \tau } ^{\mathrm{tan}}
  $. Taking $ \tau = -\delta w \in H \Lambda ^{k-1} (\Omega) $, we
  have $ \tau + \delta \mathrm{d} \tau = 0 $ with
  $ \tau ^{\mathrm{tan}} = \widehat{ \tau } ^{\mathrm{tan}} $, so
  $ ( \tau, \phi ) _\Omega + ( \mathrm{d} \tau , \mathrm{d} \phi )
  _\Omega = 0 $ for all
  $ \phi \in \mathring{ H } \Lambda ^{ k -1 } (\Omega) $. This is
  precisely the variational problem satisfied uniquely by the
  minimum-$ H \Lambda $-norm extension of
  $ \widehat{ \tau } ^{\mathrm{tan}} $, so $\tau$ is this extension
  and
  $ \lVert \widehat{ \tau } ^{\mathrm{tan}} \rVert _{\mathrm{tan},
    \partial \Omega } = \lVert \tau \rVert _{ H \Lambda ^{k-1}
    (\Omega) } $. Since $ \tau = -\delta w $ and
  $ \mathrm{d} \tau = w $, we have
  $ \lVert \tau \rVert _{ H \Lambda ^{ k -1 } (\Omega) } = \lVert w
  \rVert _{ H ^\ast \Lambda ^k (\Omega) } $, and
  \begin{equation*}
    \lVert \widehat{ \tau } ^{\mathrm{tan}} \rVert _{\mathrm{tan}, \partial \Omega } = \lVert w \rVert _{ H ^\ast \Lambda ^k (\Omega) } = \sup _{ \substack{v \in H ^\ast \Lambda ^k (\Omega), \\ v \neq 0 } } \frac{ ( w, v ) _\Omega + ( \delta w , \delta v ) _\Omega }{ \lVert v \rVert _{ H ^\ast \Lambda ^k (\Omega) } } = \sup _{ \substack{v \in H ^\ast \Lambda ^k (\Omega), \\ v \neq 0 } } \frac{ \langle \widehat{ \tau } ^{\mathrm{tan}} , v ^{\mathrm{nor}} \rangle _{ \partial \Omega }  }{ \lVert v \rVert _{ H ^\ast \Lambda ^k (\Omega) } } .
  \end{equation*}
  For any $ v ^{\mathrm{nor}} = \widehat{ v } ^{\mathrm{nor}} $, the
  denominator is minimized when
  $ \lVert v \rVert _{ H ^\ast \Lambda ^k (\Omega) } = \lVert
  \widehat{ v } ^{\mathrm{nor}} \rVert _{\mathrm{nor}, \partial \Omega
  } $, so the first equality follows. The second equality is proved
  similarly.
\end{proof}

\begin{remark}
  \label{rmk:ideal-bc}

  As an immediate consequence of the isomorphisms
  \eqref{eqn:trace_quotient}, we have
  \begin{equation*}
    \mathring{ H } \Lambda ^k (\Omega) = \bigl\{ v \in H \Lambda ^k (\Omega) : v ^{\mathrm{tan}} = 0 \bigr\} ,\qquad 
    \mathring{ H } ^\ast \Lambda ^k (\Omega) = \bigl\{ v \in H ^\ast \Lambda ^k (\Omega) : v ^{\mathrm{nor}} = 0 \bigr\} .
  \end{equation*}
  More generally, any closed extension of
  $ \mathrm{d} \colon C _0 ^\infty \Lambda ^k (\Omega) \rightarrow C
  _0 ^\infty \Lambda ^{ k + 1 } (\Omega) $ resulting in a Hilbert
  complex
  $ \mathring{ H } \Lambda ^k (\Omega) \subset V ^k \subset H \Lambda
  ^k (\Omega) $ is called a choice of \emph{ideal boundary
    conditions}, cf.~\citet{BrLe1992}. For example, one may take a
  suitably nice decomposition of $ \partial \Omega $ into two pieces,
  $ \Gamma ^{\mathrm{tan}} $ and $ \Gamma ^{\mathrm{nor}} $, and let
  $ V ^k \coloneqq \bigl\{ v \in H \Lambda ^k (\Omega) : v
  ^{\mathrm{tan}} \rvert _{ \Gamma ^{\mathrm{tan}} } = 0 \bigr\}
  $. For an analysis of these mixed boundary conditions (including
  what qualifies as a ``suitably nice decomposition''), see
  \citet*{JaMiMi2009,GoMiMi2011}.
\end{remark}

\subsection{The Hodge decomposition and Poincar\'e inequality}
\label{sec:hodge-decomp}

Although much of the following analysis applies to more general
Hilbert complexes, we focus our attention on
\begin{equation*}
  0 \rightarrow  H \Lambda ^0 (\Omega) \xrightarrow{ \mathrm{d} }  H \Lambda ^1 (\Omega) \xrightarrow{ \mathrm{d} } \cdots \xrightarrow{ \mathrm{d} } H \Lambda ^n (\Omega) \rightarrow 0 .
\end{equation*}
The operators $ \mathrm{d} $ satisfy a compactness property, as shown
by \citet{Picard1984}, and in particular they are Fredholm and thus
have closed range. Define
\begin{gather*}
  \mathfrak{B}  ^k \coloneqq \bigl\{ \mathrm{d} \tau : \tau \in H \Lambda ^{ k -1 } (\Omega) \bigr\} , \qquad \mathfrak{Z} ^k \coloneqq \bigl\{ v \in H \Lambda ^k (\Omega) : \mathrm{d} v = 0 \bigr\} , \qquad 
  \mathfrak{H} ^k \coloneqq \mathfrak{Z} ^k \cap \mathfrak{B} ^{k \perp },
\end{gather*}
which are the subspaces of \emph{exact}, \emph{closed}, and
\emph{harmonic} $k$-forms in $ L ^2 \Lambda ^k (\Omega) $. It follows
that
\begin{equation*}
  L ^2 \Lambda ^k (\Omega) = \mathfrak{B} ^k \oplus \mathfrak{H} ^k \oplus \mathfrak{Z} ^{ k \perp } ,
\end{equation*}
which is an $ L ^2 $-orthogonal decomposition called the \emph{Hodge
  decomposition}.  By Banach's closed range theorem and the
adjointness of $ \mathrm{d} $ and $ \delta $, we may also write
\begin{equation*}
  \mathfrak{B} ^{ k \perp } = \bigl\{ v \in \mathring{ H } ^\ast \Lambda ^k (\Omega) : \delta v = 0 \bigr\} \eqqcolon \mathring{ \mathfrak{Z}  } _k ^\ast , \qquad \mathfrak{Z} ^{ k \perp } = \bigl\{ \delta \eta : \eta \in \mathring{ H } ^\ast \Lambda ^{ k + 1 } (\Omega) \bigr\} \eqqcolon \mathring{ \mathfrak{B}  } ^\ast _k ,
\end{equation*}
called \emph{coclosed} and \emph{coexact} $k$-forms. This implies
\begin{equation*}
  \mathfrak{H} ^k = \mathfrak{Z}  ^k \cap \mathring{ \mathfrak{Z}  } _k ^\ast = \bigl\{ v \in H \Lambda ^k (\Omega) \cap \mathring{ H } ^\ast \Lambda ^k (\Omega) : \mathrm{d} v = 0 ,\, \delta v = 0 \bigr\} ,
\end{equation*}
which is an equivalent characterization of harmonic forms.

Finally, since $ \mathrm{d} $ is an $ H \Lambda $-bounded isomorphism
between $ H \Lambda ^k (\Omega) \cap \mathfrak{Z} ^{ k \perp } $ and
$ \mathfrak{B} ^{ k + 1 } $, Banach's bounded inverse theorem implies
that there exists a constant $ c _P (\Omega) $ such that
\begin{equation*}
  \lVert v \rVert _\Omega \leq c _P (\Omega) \lVert \mathrm{d} v \rVert _\Omega , \quad \forall v \in H \Lambda ^k (\Omega) \cap \mathfrak{Z} ^{ k \perp },
\end{equation*}
which is called the \emph{Poincar\'e inequality}. Note that
\citet*{ArFaWi2010,Arnold2018} write the Poincar\'e inequality
differently, using the
$ \lVert \cdot \rVert _{ H \Lambda ^k (\Omega) } $ norm, so that the
constant is $ \sqrt{ 1 + c _P (\Omega) ^2 } $. However, the form we
have chosen is more convenient for scaling arguments that we will
apply later.

\subsection{The Hodge--Laplace problem}

Recall the Hodge--Laplace operator
$ L \coloneqq \mathrm{d} \delta + \delta \mathrm{d} $ on $k$-forms,
which we can now interpret in a weak sense. Given
$ f \in L ^2 \Lambda ^k (\Omega) $, we wish to solve the following
problem: Find $ u \in \mathfrak{H} ^{ k \perp } $,
$ p \in \mathfrak{H} ^k $, such that
\begin{align*}
  L u + p = f \quad & \text{in $\Omega$},\\[0.5ex]
  u ^{\mathrm{nor}} = 0,\  ( \mathrm{d} u ) ^{\mathrm{nor}} = 0, \quad &
                                                                         \text{on $ \partial \Omega $}.
\end{align*}
The solution gives the Hodge decomposition
$ f = \mathrm{d} \sigma + p + \delta \rho $, where
$ \sigma = \delta u $ and $ \rho = \mathrm{d} u $.

FEEC is based on the following mixed formulation of the Hodge--Laplace
problem: Find $ \sigma \in H \Lambda ^{ k -1 } (\Omega) $,
$ u \in H \Lambda ^k (\Omega) $, $ p \in \mathfrak{H} ^k $ such that
\begin{subequations}
  \label{eqn:hodge-laplace}
  \begin{alignat}{2}
    ( \sigma, \tau ) _\Omega - ( u , \mathrm{d} \tau ) _\Omega &= 0 ,\quad &\forall \tau &\in H \Lambda ^{ k -1 } (\Omega) , \label{eqn:hodge-laplace_tau} \\
    ( \mathrm{d} \sigma , v ) _\Omega + ( \mathrm{d} u , \mathrm{d} v
    ) _\Omega + ( p , v ) _\Omega &= (f, v ) _\Omega ,\quad &\forall v
    &\in H \Lambda ^k (\Omega) , \label{eqn:hodge-laplace_v}\\
    (u, q ) _\Omega &= 0,\quad &\forall q &\in \mathfrak{H} ^k
    , \label{eqn:hodge-laplace_q}
  \end{alignat}
\end{subequations}
where both boundary conditions are natural. More generally,
nonvanishing natural boundary conditions may be imposed by adding
$ \langle \cdot , \cdot \rangle _{ \partial \Omega } $ terms on the
right-hand side. The well-posedness of this mixed formulation is
proved in \citet*[Theorem~7.2]{ArFaWi2006} and generalized to abstract
Hilbert complexes in \citet*[Theorem~3.2]{ArFaWi2010}.

\begin{remark}
  \label{rmk:nonzero-bc}

  Instead of natural boundary conditions, one may impose essential
  boundary conditions $ \sigma ^{\mathrm{tan}} = 0 $ and
  $ u ^{\mathrm{tan}} = 0 $ by taking the test and trial functions
  from $ \mathring{ H } \Lambda ^{ k -1 } (\Omega) $,
  $ \mathring{ H } \Lambda ^k (\Omega) $,
  $ \mathring{ \mathfrak{H} } ^k $, cf.~\citep[Section
  6.2]{ArFaWi2010}. This may be generalized to nonvanishing
  $ \sigma ^{\mathrm{tan}} $ and $ u ^{\mathrm{tan}} $ via a standard
  extension argument. We may also impose other ideal boundary
  conditions
  $ \mathring{ H } \Lambda (\Omega) \subset V \subset H \Lambda
  (\Omega) $, as discussed in \cref{rmk:ideal-bc}. For example, mixed
  boundary conditions are essential for $ \sigma ^{\mathrm{tan}} $,
  $ u ^{\mathrm{tan}} $ on $ \Gamma ^{\mathrm{tan}} $ and natural for
  $ u ^{\mathrm{nor}} $, $ ( \mathrm{d} u ) ^{\mathrm{nor}} $ on
  $ \Gamma ^{\mathrm{nor}} $.
\end{remark}

\subsection{Finite element exterior calculus}
\label{sec:feec}

Just as the Galerkin method approximates problems on
infinite-dimensional Hilbert spaces by restricting to
finite-dimensional subspaces, FEEC approximates problems on
infinite-dimensional Hilbert \emph{complexes} by restricting to
finite-dimensional \emph{subcomplexes}.

A subcomplex $ V _h \subset H \Lambda (\Omega) $ is a sequence of
(here, finite-dimensional) subspaces
$ V _h ^k \subset H \Lambda ^k (\Omega) $ that is closed with respect
to $ \mathrm{d} $, i.e.,
$ \mathrm{d} V _h ^k \subset V _h ^{ k + 1 } $. Just as in
\cref{sec:hodge-decomp}, we have subspaces
\begin{equation*}
  \mathfrak{B} _h ^k \coloneqq \{ \mathrm{d} \tau _h : \tau _h \in V _h ^{ k -1 } \} , \qquad \mathfrak{Z} _h ^k \coloneqq \{ v _h \in V _h ^k : \mathrm{d} v _h = 0 \} , \qquad \mathfrak{H}  _h ^k \coloneqq \mathfrak{Z}  _h ^k \cap \mathfrak{B} _h ^{ k \perp } ,
\end{equation*}
along with a discrete Hodge decomposition
$ V _h ^k = \mathfrak{B} _h ^k \oplus \mathfrak{H} _h ^k \oplus
\mathfrak{Z} _h ^{ k \perp } $ and discrete Poincar\'e
inequality. Note that the subcomplex assumption implies
$ \mathfrak{B} _h ^k \subset \mathfrak{B} ^k $ and
$ \mathfrak{Z} _h ^k \subset \mathfrak{Z} ^k $, although in general
$ \mathfrak{H} _h ^k \not\subset \mathfrak{H} ^k $ and
$ \mathfrak{Z} _h ^{ k \perp } \not\subset \mathfrak{Z} ^{ k \perp }
$. An additional key assumption in the analysis (but not
implementation) of FEEC is the existence of \emph{bounded commuting
  projections}
$ \pi _h ^k \colon H \Lambda ^k (\Omega) \rightarrow V _h ^k $, which
among other uses gives control of the discrete Poincar\'e constant in
terms of $ c _P (\Omega) $.

In FEEC, one then approximates the Hodge--Laplace problem
\eqref{eqn:hodge-laplace} by the following finite-dimensional
variational problem: Find $ \sigma _h \in V _h ^{ k -1 } $,
$ u _h \in V _h ^k $, $ p _h \in \mathfrak{H} _h ^k $ such that
\begin{subequations}
  \label{eqn:feec}
  \begin{alignat}{2}
    ( \sigma _h , \tau _h ) _\Omega - ( u _h , \mathrm{d} \tau _h ) _\Omega &= 0 ,\quad &\forall \tau _h &\in V _h ^{ k -1 } , \label{eqn:feec_tau} \\
    ( \mathrm{d} \sigma _h , v _h ) _\Omega + ( \mathrm{d} u _h ,
    \mathrm{d} v _h ) _\Omega + ( p _h , v _h ) _\Omega &= (f, v _h )
    _\Omega ,\quad &\forall v _h
    &\in V _h ^k , \label{eqn:feec_v}\\
    (u _h , q _h ) _\Omega &= 0,\quad &\forall q _h &\in \mathfrak{H}
    _h ^k . \label{eqn:feec_q}
  \end{alignat}
\end{subequations}
\citet*{ArFaWi2006,ArFaWi2010} establish stability and convergence for
this problem, proving quasi-optimal error estimates in the
$ H \Lambda $-norm and improved $ L ^2 $-error estimates under
additional regularity assumptions using the aforementioned compactness
property. (In \citep{ArFaWi2010}, much of this analysis takes place in
the setting of abstract Hilbert complexes.) As in
\cref{rmk:nonzero-bc}, we may instead take essential boundary
conditions for $ \sigma _h ^{\mathrm{tan}} $ and
$ u _h ^{\mathrm{tan}} $. \citet{Licht2019} has recently extended the
analysis of FEEC to mixed boundary conditions, including the
construction of bounded commuting projections.

One more essential ingredient of FEEC is the construction of finite
elements for the spaces $ V _h ^k $. Suppose that
$ \Omega \subset \mathbb{R}^n $ is polyhedral, and let
$ \mathcal{T} _h $ be a triangulation of $\Omega$ by $n$-simplices
$ K \in \mathcal{T} _h $. \citet*{ArFaWi2006,ArFaWi2010} construct two
families of piecewise-polynomial differential forms, called
$ \mathcal{P} _r \Lambda $ and $ \mathcal{P} _r ^- \Lambda $, which we
will sometimes refer to collectively as
$ \mathcal{P} _r ^\pm \Lambda $.  \citet*{ArFaWi2006,ArFaWi2010} show that any of the pairs of spaces
\begin{equation}
  \label{eqn:stable_pairs}
  V _h ^{ k -1 } =
  \mathcal{P} _{ r + 1 } ^\pm \Lambda ^{ k -1 } ( \mathcal{T} _h ) , \qquad V _h ^k =
  \begin{Bmatrix}
    \renewcommand\arraystretch{2} \mathcal{P} _r \Lambda ^k (
    \mathcal{T} _h ) \text{ (if
      $ r \geq 1 $)} \\[0.5ex]
    \text{or}\\[0.5ex]
    \mathcal{P} _{ r + 1 } ^- \Lambda ^k ( \mathcal{T} _h )
  \end{Bmatrix},
\end{equation}
results in a subcomplex for the problem \eqref{eqn:feec} satisfying
the needed analytical assumptions.

\section{Domain decomposition of the Hodge--Laplace problem}
\label{sec:dd}

This section presents a domain decomposition of the Hodge--Laplace
problem, where $ \Omega \subset \mathbb{R}^n $ is partitioned into
non-overlapping Lipschitz subdomains $ K \in \mathcal{T} _h $. This
will be the foundation for the hybrid methods in \cref{sec:hybrid},
where $\Omega$ is polyhedral and $ K \in \mathcal{T} _h $ are elements
of a conforming mesh. However, the results of this section also apply
to more general types of domain decomposition.

\subsection{Decomposition of Hilbert complexes of differential forms}

Define the broken spaces
\begin{equation*}
  H \Lambda ^k ( \mathcal{T} _h ) \coloneqq \prod _{ K \in \mathcal{T} _h } H \Lambda ^k (K) , \qquad H ^\ast \Lambda ^k ( \mathcal{T} _h ) \coloneqq \prod _{ K \in \mathcal{T} _h } H ^\ast \Lambda ^k (K) .
\end{equation*}
As product spaces, these naturally inherit the inner products
\begin{equation*}
  ( \cdot , \cdot ) _{ \mathcal{T} _h } \coloneqq \sum _{ K \in \mathcal{T} _h } ( \cdot , \cdot  ) _K , \quad ( \cdot , \cdot  ) _{ H \Lambda ^k ( \mathcal{T} _h ) } \coloneqq \sum _{ K \in \mathcal{T} _h } ( \cdot , \cdot ) _{ H \Lambda ^k (K) } , \quad ( \cdot , \cdot  ) _{ H ^\ast \Lambda ^k ( \mathcal{T} _h ) } \coloneqq \sum _{ K \in \mathcal{T} _h } ( \cdot , \cdot ) _{ H ^\ast \Lambda ^k (K) }.
\end{equation*} 
We can then define
$ \mathrm{d} \colon H \Lambda ^k ( \mathcal{T} _h ) \rightarrow H
\Lambda ^{ k + 1 } ( \mathcal{T} _h ) $ to be
$ \mathrm{d} \rvert _{ H \Lambda ^k (K) } $ on each
$ K \in \mathcal{T} _h $, and likewise for
$ \delta \colon H ^\ast \Lambda ^k ( \mathcal{T} _h ) \rightarrow H
^\ast \Lambda ^{ k -1 } ( \mathcal{T} _h ) $. These broken Hilbert
complexes are simply the $ H \Lambda $ and $ H ^\ast \Lambda $
complexes for the disjoint union
$ \bigsqcup _{ K \in \mathcal{T} _h } K $.

For these broken spaces, we can define tangential and normal traces on
$ \partial \mathcal{T} _h \coloneqq \bigsqcup _{ K \in \mathcal{T} _h
} \partial K $ by taking the trace on $ \partial K $ for each
$ K \in \mathcal{T} _h $. Defining the pairing
$ \langle \cdot , \cdot \rangle _{ \partial \mathcal{T} _h } \coloneqq
\sum _{ K \in \mathcal{T} _h } \langle \cdot , \cdot \rangle _{
  \partial K } $, we immediately get the integration by parts formula
\begin{equation*}
  \langle \tau ^{\mathrm{tan}} , v ^{\mathrm{nor}} \rangle _{ \partial \mathcal{T} _h } = ( \mathrm{d} \tau , v ) _{ \mathcal{T} _h } - ( \tau , \delta v ) _{ \mathcal{T} _h } , \quad \forall \tau \in H \Lambda ^{ k -1 } (\mathcal{T} _h) ,\ v \in H ^\ast \Lambda ^k (\mathcal{T} _h ),
\end{equation*}
simply by summing the integration by parts formulas for each
$ K \in \mathcal{T} _h $.  Note that, if
$ e = \partial K ^+ \cap \partial K^- $ is the interface between
$ K ^\pm \in \mathcal{T} _h $, then $e$ appears twice in the disjoint
union $ \partial \mathcal{T} _h $: once as part of $ \partial K^+ $,
and a second time as part of $ \partial K ^- $. The traces of broken
differential forms can therefore be seen as ``double valued,'' since
there is no continuity imposed at interfaces between subdomains.

There are natural inclusions
$ H \Lambda ^k (\Omega) \hookrightarrow H \Lambda ^k ( \mathcal{T} _h
) $ and
$ H ^\ast \Lambda ^k (\Omega) \hookrightarrow H ^\ast \Lambda ^k (
\mathcal{T} _h ) $, which are defined by restriction to each
$ K \in \mathcal{T} _h $. The next result characterizes these
subspaces of unbroken differential forms, generalizing some classic
results on domain decomposition of $ H ^1 $,
$ H ( \operatorname{curl} ) $, and $ H ( \operatorname{div} ) $ spaces
(cf.~Propositions 2.1.1--2.1.3 of \citet*{BoBrFo2013}). In a weak
sense, it says that unbroken differential forms are precisely those
with ``single valued'' tangential or normal traces.

\begin{proposition}
  \label{prop:dd_spaces}
  If $ \mathcal{T} _h $ is a decomposition of $ \Omega $ into
  Lipschitz subdomains, then
  \begin{align*}
    H \Lambda ^k (\Omega)
    &= \bigl\{ v \in H \Lambda ^k ( \mathcal{T} _h ) : \langle v ^{\mathrm{tan}} , \eta ^{\mathrm{nor}} \rangle _{ \partial \mathcal{T} _h } = 0 , \ \forall \eta \in \mathring{ H } ^\ast \Lambda ^{ k + 1 } (\Omega) \bigr\}, \\
    \mathring{ H } \Lambda ^k (\Omega)
    &= \bigl\{ v \in H \Lambda ^k ( \mathcal{T} _h ) : \langle v ^{\mathrm{tan}} , \eta ^{\mathrm{nor}} \rangle _{ \partial \mathcal{T} _h } = 0 , \ \forall \eta \in H ^\ast \Lambda ^{ k + 1 } (\Omega) \bigr\}, \\
    H ^\ast \Lambda ^k (\Omega)
    &= \bigl\{ v \in H ^\ast \Lambda ^k ( \mathcal{T} _h ) : \langle \tau ^{\mathrm{tan}} , v ^{\mathrm{nor}} \rangle _{ \partial \mathcal{T} _h } = 0 , \ \forall \tau  \in \mathring{ H } \Lambda ^{ k - 1 } (\Omega) \bigr\} ,\\
    \mathring{ H } ^\ast \Lambda ^k (\Omega)
    &= \bigl\{ v \in H ^\ast \Lambda ^k ( \mathcal{T} _h ) : \langle \tau ^{\mathrm{tan}} , v ^{\mathrm{nor}} \rangle _{ \partial \mathcal{T} _h } = 0 , \ \forall \tau  \in H \Lambda ^{ k - 1 } (\Omega) \bigr\}.    
  \end{align*} 
\end{proposition}

\begin{proof}
  These four identities are proved using essentially the same
  argument, so we give only a proof of the first. If
  $ v \in H \Lambda ^k (\Omega) $, then for all
  $ \eta \in \mathring{ H } ^\ast \Lambda ^{ k + 1 } (\Omega) $,
  \begin{equation*} 
    \langle v ^{\mathrm{tan}} , \eta ^{\mathrm{nor}} \rangle _{ \partial \mathcal{T} _h } = ( \mathrm{d} v , \eta ) _{ \mathcal{T} _h } - ( v , \delta \eta ) _{ \mathcal{T} _h } = ( \mathrm{d} v, \eta ) _\Omega - ( v, \delta \eta ) _\Omega = \langle v ^{\mathrm{tan}} , \eta ^{\mathrm{nor}} \rangle _{ \partial \Omega } = 0 .
  \end{equation*}
  Conversely, suppose that
  $ v \in H \Lambda ^k ( \mathcal{T} _h ) \subset L ^2 \Lambda ^k (
  \mathcal{T} _h ) \cong L ^2 \Lambda ^k (\Omega) $ satisfies
  $ \langle v ^{\mathrm{tan}} , \eta ^{\mathrm{nor}} \rangle _{
    \partial \mathcal{T} _h } = 0 $ for all
  $ \eta \in \mathring{ H } ^\ast \Lambda ^{ k + 1 } (\Omega) $. Then,
  using integration by parts and Cauchy--Schwarz,
  \begin{equation*}
    ( v , \delta \eta ) _\Omega = ( v , \delta \eta ) _{ \mathcal{T} _h } = ( \mathrm{d} v, \eta ) _{ \mathcal{T} _h } \leq \lVert \mathrm{d} v \rVert _{ \mathcal{T} _h } \lVert \eta \rVert _{ \mathcal{T} _h } = \lVert \mathrm{d} v \rVert _{ \mathcal{T} _h } \lVert \eta \rVert _\Omega .
  \end{equation*}
  In particular, this holds for
  $ \eta \in C _0 ^\infty \Lambda ^{ k + 1 } (\Omega) $, implying
  $ \mathrm{d} v \in L ^2 \Lambda ^{ k + 1 } (\Omega) $ and hence
  $ v \in H \Lambda ^k (\Omega) $.
\end{proof}

\subsection{Decomposition of the Hodge--Laplace problem}
\label{sec:dd_hodge-laplace}

For each $ K \in \mathcal{T} _h $, observe that $\sigma$ and $u$ solve
the local problem
\begin{alignat*}{2}
  ( \sigma, \tau ) _K - ( u , \mathrm{d} \tau ) _K &= 0, \quad & \forall \tau &\in \mathring{ H } \Lambda ^{ k -1 } (K) ,\\
  (\mathrm{d} \sigma, v ) _K + ( \mathrm{d} u , \mathrm{d} v ) _K &= (
  f - p , v ) _K ,\quad &\forall v &\in \mathring{ H } \Lambda ^k (K)
  ,
\end{alignat*}
with essential boundary conditions $ \sigma ^{\mathrm{tan}} $ and
$ u ^{\mathrm{tan}} $. However, if the space of \emph{local harmonic
  forms} $ \mathring{ \mathfrak{H} } ^k (K) $ is nontrivial, then this
local problem is not well-posed.\footnote{When $K \in \mathcal{T} _h $
  are contractible (e.g., simplices in a triangulation), this is only
  an issue for $ k = n $, where
  $ \mathring{ \mathfrak{H} } ^n (K) \cong \mathbb{R} $.} Therefore,
we include an additional local variable
$ \overline{ p } \in \mathring{ \mathfrak{H} } ^k (K) $ and solve
\begin{subequations}
  \label{eqn:local_solver}
\begin{alignat}{2}
  ( \sigma, \tau ) _K - ( u , \mathrm{d} \tau ) _K &= 0, \quad & \forall \tau &\in \mathring{ H } \Lambda ^{ k -1 } (K) ,\\
  (\mathrm{d} \sigma, v ) _K + ( \mathrm{d} u , \mathrm{d} v ) _K + ( \overline{p} , v ) _K &= ( f - p , v ) _K ,\quad &\forall v &\in \mathring{ H } \Lambda ^k (K) ,\\
  ( u , \overline{ q } ) _K &= ( \overline{ u } , \overline{ q } ) _K
  , \quad & \forall \overline{q} &\in \mathring{ \mathfrak{H} } ^k (K)
  ,
\end{alignat}
\end{subequations}
where $ \overline{u} $ is the projection of $u$ onto
$ \mathring{ \mathfrak{H} } ^k (K) $. Following \cref{rmk:nonzero-bc},
these local solvers are well-posed for any right-hand side and
tangential traces $ \sigma ^{\mathrm{tan}} $, $ u ^{\mathrm{tan}} $.

We now allow the tangential traces
$ \widehat{ \sigma } ^{\mathrm{tan}} $,
$ \widehat{ u } ^{\mathrm{tan}} $ to be independent variables and
impose the constraints
$ \sigma ^{\mathrm{tan}} = \widehat{ \sigma } ^{\mathrm{tan}} $,
$ u ^{\mathrm{tan}} = \widehat{ u } ^{\mathrm{tan}} $ using Lagrange
multipliers $ \widehat{ u } ^{\mathrm{nor}} $,
$ \widehat{ \rho } ^{\mathrm{nor}} $, which will turn out to be the
normal traces of $u$ and $ \rho = \mathrm{d} u $. Define the spaces
\begin{align*}
  W ^k &\coloneqq H \Lambda ^k ( \mathcal{T} _h ) , & \overline{ \mathfrak{H} } ^k  &\coloneqq  \prod _{ K \in \mathcal{T} _h } \mathring{ \mathfrak{H}  } ^k (K), \\
  \widehat{ W } ^{k, \mathrm{nor}} &\coloneqq \bigl\{ \eta ^{\mathrm{nor}} : \eta \in H ^\ast \Lambda ^{ k + 1 } ( \mathcal{T} _h ) \bigr\} , & \widehat{ V } ^{k, \mathrm{tan}} &\coloneqq \bigl\{ v ^{\mathrm{tan}} : v \in H \Lambda ^k (\Omega) \bigr\} .
\end{align*}
Note that $ \widehat{ V } ^{k, \mathrm{tan}} $ consists of ``single
valued'' traces from the unbroken space
$ H \Lambda ^k (\Omega) $, whereas the other three spaces
contain broken $k$-forms. Consider the variational problem: Find
\begin{align*} 
  \tag{local variables} \sigma &\in W ^{ k -1 } , & u &\in W ^k , & \overline{p} &\in \overline{ \mathfrak{H}  } ^k , & \widehat{ u } ^{\mathrm{nor}} &\in \widehat{ W } ^{k - 1, \mathrm{nor}} , & \widehat{ \rho } ^{\mathrm{nor}} &\in \widehat{ W } ^{k, \mathrm{nor}} , \\
  \tag{global variables} && p &\in \mathfrak{H} ^k ,& \overline{ u } &\in \overline{ \mathfrak{H}  } ^k , & \widehat{ \sigma } ^{\mathrm{tan}} &\in \widehat{ V } ^{k-1, \mathrm{tan}} , & \widehat{ u } ^{\mathrm{tan}} &\in \widehat{ V } ^{k, \mathrm{tan}} ,
\end{align*} 
satisfying
\begin{subequations}
  \label{eqn:dd}
  \begin{alignat}{2}
    ( \sigma, \tau ) _{ \mathcal{T} _h } - ( u , \mathrm{d} \tau ) _{ \mathcal{T} _h } + \langle \widehat{ u } ^{\mathrm{nor}} , \tau ^{\mathrm{tan}} \rangle _{ \partial \mathcal{T} _h } &= 0, \quad &\forall \tau &\in W ^{ k -1 } , \label{eqn:dd_tau}\\
    (\mathrm{d} \sigma, v ) _{ \mathcal{T} _h } + ( \mathrm{d} u , \mathrm{d} v ) _{ \mathcal{T} _h } + ( \overline{ p } + p, v ) _{ \mathcal{T} _h } - \langle \widehat{ \rho } ^{\mathrm{nor}} , v ^{\mathrm{tan}} \rangle _{ \partial \mathcal{T} _h } &= ( f, v ) _{ \mathcal{T} _h } , \quad &\forall v &\in W ^k , \label{eqn:dd_v}\\
   ( \overline{ u } - u , \overline{ q } ) _{ \mathcal{T} _h } &= 0 ,\quad &\forall \overline{ q } &\in \overline{ \mathfrak{H}  } ^k , \label{eqn:dd_qbar}\\
  \langle \widehat{ \sigma } ^{\mathrm{tan}} - \sigma ^{\mathrm{tan}} , \widehat{ v } ^{\mathrm{nor}} \rangle _{ \partial \mathcal{T} _h } &= 0, \quad &\forall \widehat{ v } ^{\mathrm{nor}} &\in \widehat{ W } ^{k-1, \mathrm{nor}}, \label{eqn:dd_vnor}\\
  \langle \widehat{ u } ^{\mathrm{tan}} - u ^{\mathrm{tan}} , \widehat{ \eta } ^{\mathrm{nor}} \rangle _{ \partial \mathcal{T} _h } &= 0, \quad &\forall \widehat{ \eta } ^{\mathrm{nor}} &\in \widehat{ W } ^{k, \mathrm{nor}} , \label{eqn:dd_etanor}\\
  (u, q ) _{ \mathcal{T} _h } &= 0, \quad &\forall q &\in \mathfrak{H} ^k \label{eqn:dd_q},\\
   ( \overline{ p } , \overline{ v } ) _{ \mathcal{T} _h } &= 0 ,\quad &\forall \overline{ v } &\in \overline{ \mathfrak{H}  } ^k \label{eqn:dd_vbar} ,\\
  \langle \widehat{ u } ^{\mathrm{nor}} , \widehat{ \tau } ^{\mathrm{tan}} \rangle _{ \partial \mathcal{T} _h } &= 0 , \quad&\forall \widehat{ \tau } ^{\mathrm{tan}} &\in \widehat{ V } ^{k-1, \mathrm{tan}} , \label{eqn:dd_tautan}\\
  \langle \widehat{ \rho } ^{\mathrm{nor}} , \widehat{ v } ^{\mathrm{tan}} \rangle _{ \partial \mathcal{T} _h } &= 0 , \quad&\forall \widehat{ v } ^{\mathrm{tan}} &\in \widehat{ V } ^{k, \mathrm{tan}} . \label{eqn:dd_vtan}
\end{alignat}
\end{subequations}
Given values for the global variables, notice that
\eqref{eqn:dd_tau}--\eqref{eqn:dd_etanor} simply amounts to solving
the local problem \eqref{eqn:local_solver} on each
$ K \in \mathcal{T} _h $.

We now prove that this is indeed a domain decomposition of the
Hodge--Laplace problem \eqref{eqn:hodge-laplace}, which in particular
implies well-posedness of \eqref{eqn:dd}. A more general proof of
well-posedness, where the right-hand side of \eqref{eqn:dd} is allowed
to be arbitrary, will be given in \cref{sec:saddle}.

\begin{theorem}
  \label{thm:dd}
  The following are equivalent:
  \begin{itemize}
  \item
    $ ( \sigma, u, \overline{ p } , \widehat{ u } ^{\mathrm{nor}} ,
    \widehat{ \rho } ^{\mathrm{nor}} , p, \overline{ u } , \widehat{
      \sigma } ^{\mathrm{tan}} , \widehat{ u } ^{\mathrm{tan}} ) $ is
    a solution to \eqref{eqn:dd}.
  \item $ ( \sigma, u, p ) $ is a solution to
    \eqref{eqn:hodge-laplace}, and furthermore,
    $ \overline{ p } = 0 $,
    $ \widehat{ u } ^{\mathrm{nor}} = u ^{\mathrm{nor}}$,
    $ \widehat{ \rho } ^{\mathrm{nor}} = (\mathrm{d} u)
    ^{\mathrm{nor}} $, $ \overline{ u } $ is the projection of $u$
    onto $ \overline{ \mathfrak{H} } ^k $,
    $ \widehat{ \sigma } ^{\mathrm{tan}} = \sigma ^{\mathrm{tan}} $,
    and $ \widehat{ u } ^{\mathrm{tan}} = u ^{\mathrm{tan}} $.
  \end{itemize} 
\end{theorem}

\begin{proof}
  Suppose we have a solution to \eqref{eqn:dd}. The claimed equalities
  are immediate from the variational problem, so it remains only to
  show that $ ( \sigma, u , p ) $ solves
  \eqref{eqn:hodge-laplace}. Since
  $ \sigma ^{\mathrm{tan}} = \widehat{ \sigma } ^{\mathrm{tan}} $ and
  $ u ^{\mathrm{tan}} = \widehat{ u } ^{\mathrm{tan}} $,
  \cref{prop:dd_spaces} implies that
  $ \sigma \in H \Lambda ^{ k -1 } (\Omega) $ and
  $ u \in H \Lambda ^k (\Omega) $. Therefore, taking test functions
  $ \tau \in H \Lambda ^{ k -1 } (\Omega) $ and
  $ v \in H \Lambda ^k (\Omega) $ in
  \eqref{eqn:dd_tau}--\eqref{eqn:dd_v}, the normal trace terms vanish
  by \eqref{eqn:dd_tautan}--\eqref{eqn:dd_vtan}, and we obtain
  \eqref{eqn:hodge-laplace_tau}--\eqref{eqn:hodge-laplace_v}. Finally,
  \eqref{eqn:dd_q} is the same as \eqref{eqn:hodge-laplace_q}, which
  proves the forward direction.

  Conversely, given a solution $ ( \sigma, u, p ) $ to
  \eqref{eqn:hodge-laplace}, it is immediate that
  \eqref{eqn:dd_tau}--\eqref{eqn:dd_vbar} hold. For the remaining two
  equations, first observe that combining
  \eqref{eqn:hodge-laplace_tau} and \eqref{eqn:dd_tau} gives
  $ \langle \widehat{ u } ^{\mathrm{nor}} , \tau ^{\mathrm{tan}}
  \rangle _{ \partial \mathcal{T} _h } = 0 $ for
  $ \tau \in H \Lambda ^{ k -1 } (\Omega) $, which implies
  \eqref{eqn:dd_tautan}. Similarly, combining
  \eqref{eqn:hodge-laplace_v} and \eqref{eqn:dd_v} gives
  $ \langle \widehat{ \rho } ^{\mathrm{nor}} , v ^{\mathrm{tan}}
  \rangle _{ \partial \mathcal{T} _h } = 0 $ for
  $ v \in H \Lambda ^k (\Omega) $, which implies \eqref{eqn:dd_vtan}.
\end{proof}

For the last step of the proof, we could instead have used that
\eqref{eqn:hodge-laplace_tau} gives
$ u \in \mathring{ H } ^\ast \Lambda ^k (\Omega) $ and
\eqref{eqn:hodge-laplace_v} gives
$ \mathrm{d} u \in \mathring{ H } ^\ast \Lambda ^{ k + 1 } (\Omega) $,
applying \cref{prop:dd_spaces} to conclude that their normal traces
satisfy \eqref{eqn:dd_tautan}--\eqref{eqn:dd_vtan}. However, as we
will see, the variational argument above generalizes more readily to
the hybridization of FEEC in \cref{sec:hybrid}.

\begin{remark}
  \label{rmk:dd_bc}
  Although the domain decomposition is presented above for
  $ H \Lambda (\Omega) $ with natural boundary conditions on
  $ \partial \Omega $, it is easily generalized to
  $ \mathring{ H } \Lambda (\Omega) $ or other ideal boundary
  conditions
  $ \mathring{ H } \Lambda (\Omega) \subset V \subset H \Lambda
  (\Omega) $, as in \cref{rmk:nonzero-bc}. In this case, the broken
  spaces are unchanged, and we take the unbroken tangential traces and
  harmonic forms to be those from the complex $V$.
\end{remark}

We note two special cases that recover known methods for the scalar
Poisson equation:
\begin{itemize}
\item When $ k = 0 $, the only nontrivial fields are $u$,
  $ \widehat{ \rho } ^{\mathrm{nor}} $, $p$, and
  $ \widehat{ u } ^{\mathrm{nor}} $, and the Neumann problem on
  $\Omega$ is decomposed into local Dirichlet problems on
  $ K \in \mathcal{T} _h $. If
  $ V = \mathring{ H } \Lambda (\Omega) $, so that $ \partial \Omega $
  also has Dirichlet conditions, then $p$ is trivial, and we recover
  the ``three-field domain decomposition method'' of
  \citet{BrMa1994}. This decomposition is the foundation for the
  hybridized continuous Galerkin method of \citet*{CoGoWa2007}.

\item When $ k = n $, the mixed formulation of the Dirichlet problem
  on $\Omega$ is decomposed into local Neumann problems on
  $ K \in \mathcal{T} _h $. Assuming the subdomains are connected, the
  local harmonic variables $ \overline{ u } $ and $ \overline{ p } $
  are piecewise constant, and we recover the domain decomposition
  appearing in \citet[Section 5.1]{Cockburn2016}, used for
  hybridization with local Neumann solvers.
\end{itemize} 

\subsection{Saddle point formulation and well-posedness}
\label{sec:saddle}

Define the bilinear forms
\begin{align*}
  a \bigl( ( \sigma , u, \overline{p}, \widehat{ u } ^{\mathrm{nor}} , \widehat{ \rho } ^{\mathrm{nor}} ), ( \tau, v, \overline{q}, \widehat{ v } ^{\mathrm{nor}} , \widehat{ \eta } ^{\mathrm{nor}} ) \bigr) 
  &\coloneqq - ( \sigma, \tau ) _{ \mathcal{T} _h } + ( u , \mathrm{d} \tau ) _{ \mathcal{T} _h } - \langle \widehat{ u } ^{\mathrm{nor}} , \tau ^{\mathrm{tan}} \rangle _{ \partial \mathcal{T} _h } \\
  &\hphantom{{}\coloneqq{}} {+} ( \mathrm{d} \sigma , v ) _{ \mathcal{T} _h } + ( \mathrm{d} u, \mathrm{d} v ) _{ \mathcal{T} _h } + ( \overline{p} , v ) _{ \mathcal{T} _h } - \langle \widehat{ \rho } ^{\mathrm{nor}} , v ^{\mathrm{tan}} \rangle _{ \partial \mathcal{T} _h } \\
  &\hphantom{{}\coloneqq{}} {+} ( u, \overline{q} ) _{ \mathcal{T} _h } - \langle \sigma ^{\mathrm{tan}} , \widehat{ v } ^{\mathrm{nor}} \rangle _{ \partial \mathcal{T} _h } - \langle u ^{\mathrm{tan}} , \widehat{ \eta } ^{\mathrm{nor}} \rangle _{ \partial \mathcal{T} _h } ,\\
  b \bigl( ( \tau , v, \overline{q}, \widehat{ v } ^{\mathrm{nor}} , \widehat{ \eta } ^{\mathrm{nor}} ), (q, \overline{ v } , \widehat{ \tau } ^{\mathrm{tan}} , \widehat{ v } ^{\mathrm{tan}} )\bigr) &\coloneqq ( v, q ) _{ \mathcal{T} _h } - ( \overline{q}, \overline{ v } ) _{ \mathcal{T} _h } + \langle \widehat{ v } ^{\mathrm{nor}} , \widehat{ \tau } ^{\mathrm{tan}} \rangle _{ \partial \mathcal{T} _h } + \langle \widehat{ \eta } ^{\mathrm{nor}} , \widehat{ v } ^{\mathrm{tan}} \rangle _{ \partial \mathcal{T} _h } ,
\end{align*}
where we have chosen the signs so that $a ( \cdot , \cdot ) $ is
symmetric. Then the domain-decomposed Hodge--Laplace problem
\eqref{eqn:dd} becomes a particular instance of the saddle-point
problem
\begin{subequations}
\label{eqn:saddle}
\begin{alignat}{2}
  a ( x, x ^\prime ) + b ( x ^\prime , y ) &= F (x ^\prime), \quad &\forall x ^\prime &\in X ,\label{eqn:saddle_x}\\
  b ( x, y ^\prime ) &= G ( y ^\prime ) , \quad &\forall y ^\prime
  &\in Y \label{eqn:saddle_y}.
\end{alignat}
\end{subequations}
Here, $X$ is the space of local variables and $Y$ is the space of
global variables, so $ a ( \cdot , \cdot ) $ corresponds to the local
solvers and $ b ( \cdot , \cdot ) $ to the coupling between local and
global variables. This saddle point formulation will also be useful
for describing the procedure of static condensation in
\cref{sec:condensation}.

\begin{theorem}
  The problem \eqref{eqn:saddle} is well-posed.
\end{theorem}

\begin{proof}
  It suffices to show that $ b ( \cdot , \cdot ) $ satisfies a single
  inf-sup condition, meaning that the map $ x \mapsto b ( x, \cdot ) $
  is surjective, and that $ a ( \cdot , \cdot ) $ satisfies a double
  inf-sup condition on the kernel of this map, cf.~\citet*[Theorem
  4.2.3]{BoBrFo2013}.

  Let $ q $, $ \overline{ v } $, $ \widehat{ \tau } ^{\mathrm{tan}} $,
  and $ \widehat{ v } ^{\mathrm{tan}} $ be arbitrary. For the first
  two terms appearing in $b ( \cdot , \cdot ) $, we have
  \begin{equation*}
    \lVert q \rVert _{ \mathcal{T} _h } = \sup _{ v \neq 0 } \frac{ ( v, q ) _{ \mathcal{T} _h } }{ \lVert v \rVert _{ \mathcal{T} _h } } , \qquad \lVert \overline{ v }  \rVert _{ \mathcal{T} _h } = \sup_{ \overline{ q } \neq 0 } \frac{ -( \overline{ q } , \overline{ v }  ) _{ \mathcal{T} _h } }{ \hphantom{-}\lVert \overline{ q } \rVert _{ \mathcal{T} _h } } ,
  \end{equation*}
  attained at $ v = q $ and $ \overline{ q } = - \overline{ v } $ when
  these are nonzero. Applying \cref{lem:duality_isometry} to each
  $ K \in \mathcal{T} _h $ gives
  \begin{equation*}
    \lVert \widehat{ \tau } ^{\mathrm{tan}} \rVert _{\mathrm{tan}, \partial \mathcal{T} _h } = \sup _{ \widehat{ v } ^{\mathrm{nor}} \neq 0 } \frac{ \langle \widehat{ \tau } ^{\mathrm{tan}} , \widehat{ v } ^{\mathrm{nor}} \rangle _{ \partial \mathcal{T} _h } }{ \lVert \widehat{ v } ^{\mathrm{nor}} \rVert _{\mathrm{nor}, \partial \mathcal{T} _h } },\qquad 
    \lVert \widehat{ v } ^{\mathrm{tan}} \rVert _{\mathrm{tan}, \partial \mathcal{T} _h } = \sup _{ \widehat{ \eta } ^{\mathrm{nor}} \neq 0 } \frac{ \langle \widehat{ v } ^{\mathrm{tan}}, \widehat{ \eta } ^{\mathrm{nor}} \rangle _{ \partial \mathcal{T} _h } }{ \lVert \widehat{ \eta } ^{\mathrm{nor}} \rVert _{\mathrm{nor}, \partial \mathcal{T} _h } } ,
  \end{equation*}
  which proves the inf-sup condition for $ b ( \cdot , \cdot )$.  It
  remains to show that $a (\cdot , \cdot )$ satisfies an inf-sup
  condition on the kernel of $ x \mapsto b ( x, \cdot ) $. On this
  kernel, we have
  \begin{equation*}
    u, v \perp \mathfrak{H}  ^k , \qquad \overline{q}, \overline{p} = 0 , \qquad \widehat{ u } ^{\mathrm{nor}} , \widehat{ v } ^{\mathrm{nor}} \perp \widehat{ V } ^{ k -1, \mathrm{tan} } , \qquad \widehat{ \rho } ^{\mathrm{nor}} , \widehat{ \eta } ^{\mathrm{nor}} \perp \widehat{ V } ^{ k , \mathrm{tan} } ,
  \end{equation*}
  and we may further separate $a ( \cdot , \cdot )$ into a pair of bilinear forms
  \begin{align*}
    \alpha \bigl( ( \sigma, u ) , ( \tau, v ) \bigr) &= - ( \sigma , \tau ) _{ \mathcal{T} _h } + ( u , \mathrm{d} \tau ) _{ \mathcal{T} _h } + ( \mathrm{d} \sigma , v ) _{ \mathcal{T} _h } + ( \mathrm{d} u , \mathrm{d} v ) _{ \mathcal{T} _h } ,\\
    \beta \bigl( ( \tau, v ) , ( \widehat{ v } ^{\mathrm{nor}} , \widehat{ \eta } ^{\mathrm{nor}} ) \bigr) &= - \langle \widehat{ v } ^{\mathrm{nor}} , \tau ^{\mathrm{tan}} \rangle _{ \partial \mathcal{T} _h } - \langle \widehat{ \eta } ^{\mathrm{nor}} , v ^{\mathrm{tan}} \rangle _{ \partial \mathcal{T} _h } .
  \end{align*}
  The inf-sup condition for $\beta ( \cdot , \cdot ) $ holds by
  another application of \cref{lem:duality_isometry} on each
  $ K \in \mathcal{T} _h $. Finally, using \cref{prop:dd_spaces}, the
  kernel of $ \xi \mapsto \beta ( \xi , \cdot ) $ is precisely
  $ H \Lambda ^{ k -1 } (\Omega) \times \mathfrak{H} ^{ k \perp } $,
  so the inf-sup condition for $\alpha ( \cdot , \cdot ) $ on the
  kernel is just that for the non-domain-decomposed Hodge--Laplace
  problem, cf.~\citet*[Theorem 3.2]{ArFaWi2010}.
\end{proof}

\section{Hybrid methods and static condensation}
\label{sec:hybrid}

In this section, we present a hybridization of the FEEC methods of
\cref{sec:feec} for the Hodge--Laplace problem, based on the
domain-decomposed variational principle \eqref{eqn:dd}. We then
perform static condensation of these methods, using the local solvers
to efficiently reduce the system to a smaller one involving only the
global variables. This condensed system is shown to be as small or
smaller than that for standard FEEC without hybridization, and we
prove an explicit formula for the number of reduced degrees of
freedom. Finally, we prove error estimates for the hybrid variables,
which approximate tangential and normal traces.

\subsection{Hybridized FEEC methods}
\label{sec:feec-h}

For each $ K \in \mathcal{T} _h $, let
$ W _h (K) \subset H \Lambda (K) $ be a finite-dimensional subcomplex,
so that
\begin{equation*}
  W _h \coloneqq \prod _{ K \in \mathcal{T} _h } W _h (K) , \qquad V _h \coloneqq V \cap W _h ,
\end{equation*}
are respectively subcomplexes of $W = H \Lambda ( \mathcal{T} _h ) $
and $ V = H \Lambda (\Omega) $.\footnote{As in \cref{rmk:dd_bc}, the
  arguments readily generalize to
  $ V = \mathring{ H } \Lambda (\Omega) $ or other choices of ideal
  boundary conditions.}  Let
$ \overline{ \mathfrak{H} } _h ^k \coloneqq \prod _{ K \in \mathcal{T}
  _h } \mathring{ \mathfrak{H} } _h ^k (K) $, where
$ \mathring{ \mathfrak{H} } _h ^k (K) $ is the space of local harmonic
$k$-forms in $ \mathring{ W } _h ^k (K) $, and let
$ \mathfrak{H} _h ^k $ be the space of global harmonic $k$-forms in
$ V _h ^k $. Next, we define broken and unbroken tangential traces,
\begin{equation*}
  \widehat{ W } _h ^{k, \mathrm{tan}} \coloneqq \{ v _h ^{\mathrm{tan}} : v _h \in W _h ^k \} ,   \qquad \widehat{ V } _h ^{k, \mathrm{tan}} \coloneqq \{ v _h ^{\mathrm{tan}} : v _h \in V _h ^k \} = \widehat{ V } ^{k, \mathrm{tan}} \cap \widehat{ W } _h ^{ k , \mathrm{tan} } ,
\end{equation*}
and take
$ \widehat{ W } _h ^{ k, \mathrm{nor} } \coloneqq ( \widehat{ W } _h
^{ k , \mathrm{tan}} ) ^\ast $. Since
$ \langle \cdot , \cdot \rangle _{ \partial \mathcal{T} _h } $ is a
duality pairing, we use this same notation for the pairing of
$ \widehat{ W } _h ^{ k , \mathrm{tan}} $ with its dual space
$ \widehat{ W } _h ^{ k, \mathrm{nor} } $.

\begin{example}[decomposition of \unboldmath $ \mathcal{P} _r ^\pm \Lambda $ elements]
  If $ \mathcal{T} _h $ is a conforming simplicial mesh and
  $ W _h ^k (K) = \mathcal{P} _r ^\pm \Lambda ^k (K) $ for each
  $ K \in \mathcal{T} _h $, then
  $ V _h ^k = \mathcal{P} _r ^\pm \Lambda ^k ( \mathcal{T} _h )
  $. Since simplices are contractible, the local harmonic forms are
  trivial for $ k < n $ and piecewise constants for $ k = n $, and the
  global harmonic forms $ \mathfrak{H} _h ^k $ are as in
  \cref{sec:feec}.

  For each $ K \in \mathcal{T} _h $, the broken trace space
  $ \widehat{ W } _h ^{k, \mathrm{tan}} $ contains tangential traces
  of $ \mathcal{P} _r ^\pm \Lambda ^k (K) $, so the degrees of freedom
  are just those living on $ \partial K $. Since this is a broken
  space, the degrees of freedom need not match on interior facets
  $ e = \partial K ^+ \cap \partial K ^- $. By contrast,
  $ \widehat{ V } _h ^{ k , \mathrm{tan} } $ contains tangential
  traces from the unbroken space
  $ \mathcal{P} _r ^\pm \Lambda ^k ( \mathcal{T} _h ) $, so the
  degrees of freedom are single-valued. Finally, we can use duality to
  identify $ \widehat{ W } _h ^{k, \mathrm{nor}} $ with the degrees of
  freedom for $ \widehat{ W } _h ^{k, \mathrm{tan}} $. Since these
  tangential traces are piecewise polynomial and thus in
  $ L ^2 ( \partial \mathcal{T} _h ) $, for implementation we may
  simply take
  $ \widehat{ W } _h ^{k, \mathrm{nor}} = \widehat{ W } _h ^{ k,
    \mathrm{tan} } $ where
  $ \langle \cdot , \cdot \rangle _{ \partial \mathcal{T} _h } $ is
  the $ L ^2 $ inner product.
\end{example}

Now that we have defined these finite-dimensional subspaces, we may
consider the following finite-dimensional version of the
domain-decomposed variational problem \eqref{eqn:dd}: Find
\begin{align*} 
  \tag{local variables} \sigma _h &\in W _h ^{ k -1 } , & u _h &\in W _h  ^k , & \overline{p} _h &\in \overline{ \mathfrak{H}  } _h ^k , & \widehat{ u } _h ^{\mathrm{nor}} &\in \widehat{ W } _h ^{k - 1, \mathrm{nor}} , & \widehat{ \rho } _h ^{\mathrm{nor}} &\in \widehat{ W } _h ^{k, \mathrm{nor}} , \\
  \tag{global variables} && p _h &\in \mathfrak{H} _h ^k ,& \overline{ u } _h &\in \overline{ \mathfrak{H}  } _h ^k , & \widehat{ \sigma } _h ^{\mathrm{tan}} &\in \widehat{ V } _h ^{k-1, \mathrm{tan}} , & \widehat{ u } _h ^{\mathrm{tan}} &\in \widehat{ V } _h ^{k, \mathrm{tan}} ,
\end{align*} 
satisfying
\begin{subequations}
  \label{eqn:feec-h}
  \begin{alignat}{2}
    ( \sigma _h , \tau _h ) _{ \mathcal{T} _h } - ( u _h , \mathrm{d} \tau _h ) _{ \mathcal{T} _h } + \langle \widehat{ u } _h ^{\mathrm{nor}} , \tau _h ^{\mathrm{tan}} \rangle _{ \partial \mathcal{T} _h } &= 0, \quad &\forall \tau _h &\in W _h ^{ k -1 } , \label{eqn:feec-h_tau}\\
    (\mathrm{d} \sigma _h , v _h ) _{ \mathcal{T} _h } + ( \mathrm{d} u _h , \mathrm{d} v _h ) _{ \mathcal{T} _h } + ( \overline{ p } _h + p _h , v _h ) _{ \mathcal{T} _h } - \langle \widehat{ \rho } _h ^{\mathrm{nor}} , v _h ^{\mathrm{tan}} \rangle _{ \partial \mathcal{T} _h } &= ( f, v _h ) _{ \mathcal{T} _h } , &\forall v _h &\in W _h ^k , \label{eqn:feec-h_v}\\
   ( \overline{ u } _h - u _h , \overline{ q } _h ) _{ \mathcal{T} _h } &= 0 ,\quad &\forall \overline{ q } _h &\in \overline{ \mathfrak{H}  } _h ^k , \label{eqn:feec-h_qbar}\\
  \langle \widehat{ \sigma } _h ^{\mathrm{tan}} - \sigma _h ^{\mathrm{tan}} , \widehat{ v } _h ^{\mathrm{nor}} \rangle _{ \partial \mathcal{T} _h } &= 0, \quad &\forall \widehat{ v } _h ^{\mathrm{nor}} &\in \widehat{ W } _h ^{k-1, \mathrm{nor}}, \label{eqn:feec-h_vnor}\\
  \langle \widehat{ u } _h ^{\mathrm{tan}} - u _h ^{\mathrm{tan}} , \widehat{ \eta } _h ^{\mathrm{nor}} \rangle _{ \partial \mathcal{T} _h } &= 0, \quad &\forall \widehat{ \eta } _h ^{\mathrm{nor}} &\in \widehat{ W } _h ^{k, \mathrm{nor}} , \label{eqn:feec-h_etanor}\\
  (u _h , q _h ) _{ \mathcal{T} _h } &= 0, \quad &\forall q _h &\in \mathfrak{H} _h ^k , \label{eqn:feec-h_q}\\
   ( \overline{ p } _h , \overline{ v } _h ) _{ \mathcal{T} _h } &= 0 ,\quad &\forall \overline{ v } _h &\in \overline{ \mathfrak{H}  } _h ^k , \label{eqn:feec-h_vbar}\\
  \langle \widehat{ u } _h ^{\mathrm{nor}} , \widehat{ \tau } _h ^{\mathrm{tan}} \rangle _{ \partial \mathcal{T} _h } &= 0 , \quad&\forall \widehat{ \tau } _h ^{\mathrm{tan}} &\in \widehat{ V } _h ^{k-1, \mathrm{tan}} , \label{eqn:feec-h_tautan}\\
  \langle \widehat{ \rho } _h ^{\mathrm{nor}} , \widehat{ v } _h ^{\mathrm{tan}} \rangle _{ \partial \mathcal{T} _h } &= 0 , \quad&\forall \widehat{ v } _h ^{\mathrm{tan}} &\in \widehat{ V } _h ^{k, \mathrm{tan}} . \label{eqn:feec-h_vtan}
\end{alignat}
\end{subequations}
Given values for the global variables,
\eqref{eqn:feec-h_tau}--\eqref{eqn:feec-h_etanor} amounts to solving
the local FEEC problems
\begin{subequations}
  \label{eqn:feec_local}
  \begin{alignat}{2}
    ( \sigma _h , \tau _h ) _K - ( u _h , \mathrm{d} \tau _h ) _K &= 0 ,\quad &\forall \tau _h &\in \mathring{ W } _h ^{ k -1 } (K) , \label{eqn:feec_local_tau} \\
    ( \mathrm{d} \sigma _h , v _h ) _K + ( \mathrm{d} u _h ,
    \mathrm{d} v _h ) _K + ( \overline{p} _h , v _h ) _K &= (f - p _h , v _h ) _K ,\quad &\forall v _h
    &\in \mathring{ W } _h ^k (K) , \label{eqn:feec_local_v}\\
    (u _h , \overline{q} _h ) _K &= ( \overline{u} _h , \overline{q} _h ) _K ,\quad &\forall \overline{q} _h &\in \mathring{
      \mathfrak{H} } _h ^k (K) , \label{eqn:feec_local_q}
  \end{alignat}
\end{subequations}
with essential tangential boundary conditions
$ \sigma _h ^{\mathrm{tan}} = \widehat{ \sigma } _h ^{\mathrm{tan}} $
and $ u _h ^{\mathrm{tan}} = \widehat{ u } _h ^{\mathrm{tan}} $.

The following result shows that this is indeed a hybridization of the
global FEEC problem \eqref{eqn:feec}, which in particular implies
well-posedness of \eqref{eqn:feec-h}. The proof is quite similar to
\cref{thm:dd}, but there are two important distinctions. First,
$ \widehat{ u } _h ^{\mathrm{nor}} $ and
$ \widehat{ \rho } _h ^{\mathrm{nor}} $ generally \emph{do not} equal
the normal traces of $ u _h $ and $ \rho _h = \mathrm{d} u _h $,
except weakly, in a Galerkin sense. Furthermore, a crucial role is
played by the specific choice of broken tangential and normal trace
spaces above, particularly the fact that they are in duality with
respect to
$ \langle \cdot , \cdot \rangle _{ \partial \mathcal{T} _h } $.

\begin{theorem}
  \label{thm:feec-h}
  The following are equivalent:
  \begin{itemize}
  \item
    $ ( \sigma _h , u _h , \overline{p} _h , \widehat{ u } _h
    ^{\mathrm{nor}} , \widehat{ \rho } _h ^{\mathrm{nor}} , p _h ,
    \overline{ u } _h , \widehat{ \sigma } _h ^{\mathrm{tan}} ,
    \widehat{ u } _h ^{\mathrm{tan}} ) $ is a solution to
    \eqref{eqn:feec-h}.

  \item $ ( \sigma _h , u _h , p _h ) $ is a solution to
    \eqref{eqn:feec}, and furthermore, $ \overline{p} _h = 0 $,
    $ \widehat{ u } _h ^{\mathrm{nor}} $ and
    $ \widehat{ \rho } _h ^{\mathrm{nor}} $ are uniquely determined by
    \eqref{eqn:feec-h_tau}--\eqref{eqn:feec-h_v}, $ \overline{u} _h $
    is the projection of $ u _h $ onto
    $ \overline{ \mathfrak{H} } _h ^k $,
    $ \widehat{ \sigma } _h ^{\mathrm{tan}} = \sigma _h
    ^{\mathrm{tan}} $, and
    $ \widehat{ u } _h ^{\mathrm{tan}} = u _h ^{\mathrm{tan}} $.
  \end{itemize}
\end{theorem}

\begin{proof}
  Suppose we have a solution to \eqref{eqn:feec-h}. The claimed
  equalities are immediate from the variational problem, with
  uniqueness of the broken tangential and normal traces following from
  the fact that these spaces are in duality with respect to
  $ \langle \cdot , \cdot \rangle _{ \partial \mathcal{T} _h } $, so
  it remains only to show that $ ( \sigma _h , u _h , p _h ) $ solves
  \eqref{eqn:feec}. Since
  $ \sigma _h ^{\mathrm{tan}} = \widehat{ \sigma } _h ^{\mathrm{tan}}
  $ and $ u _h ^{\mathrm{tan}} = \widehat{ u } _h ^{\mathrm{tan}} $,
  \cref{prop:dd_spaces} implies that
  $ \sigma _h \in V _h ^{ k -1 } $ and
  $ u _h \in V _h ^k $. Taking
  $ \tau _h \in V _h ^{ k -1 } $ and
  $ v _h \in V _h ^k $ in
  \eqref{eqn:feec-h_tau}--\eqref{eqn:feec-h_v}, the normal trace terms
  vanish by \eqref{eqn:feec-h_tautan}--\eqref{eqn:feec-h_vtan}, and we
  obtain \eqref{eqn:feec_tau}--\eqref{eqn:feec_v}. Finally,
  \eqref{eqn:feec-h_q} is the same as \eqref{eqn:feec_q}, which proves
  the forward direction.

  Conversely, given a solution $ ( \sigma _h , u _h , p _h ) $ to
  \eqref{eqn:feec}, it is immediate that
  \eqref{eqn:feec-h_tau}--\eqref{eqn:feec-h_vbar} hold, again using
  the fact that
  $ \langle \cdot , \cdot \rangle _{ \partial \mathcal{T} _h } $ is a
  dual pairing to get uniqueness of the broken tangential and normal
  traces. For the remaining two equations, first observe that
  combining \eqref{eqn:feec_tau} and \eqref{eqn:feec-h_tau} gives
  $ \langle \widehat{ u } _h ^{\mathrm{nor}} , \tau _h ^{\mathrm{tan}}
  \rangle _{ \partial \mathcal{T} _h } = 0 $ for
  $ \tau _h \in V _h ^{ k -1 } $, which implies
  \eqref{eqn:feec-h_tautan}. Similarly, combining \eqref{eqn:feec_v}
  and \eqref{eqn:feec-h_v} gives
  $ \langle \widehat{ \rho } _h ^{\mathrm{nor}} , v _h ^{\mathrm{tan}}
  \rangle _{ \partial \mathcal{T} _h } = 0 $ for
  $ v _h \in V _h ^k $, which implies
  \eqref{eqn:feec-h_vtan}.
\end{proof}

\subsection{Static condensation}
\label{sec:condensation}

We next perform static condensation of the hybridized FEEC method
\eqref{eqn:feec-h}, eliminating the local variables using the local
solvers \eqref{eqn:feec_local} and thereby obtaining a condensed
system involving only the global variables. We present the condensed
system both in a matrix-free variational form and as a matrix Schur
complement, and we prove that this system is as small or smaller than
the standard FEEC method \eqref{eqn:feec} without hybridization.

As we did in \cref{sec:saddle} for the infinite-dimensional problem,
we may write the hybridized FEEC method \eqref{eqn:feec-h} as a saddle
point problem,
\begin{subequations}
  \label{eqn:discrete_saddle}
\begin{alignat}{2}
  a( x _h , x _h ^\prime ) + b ( x _h ^\prime , y _h ) &= F ( x _h ^\prime ) , \quad &\forall x _h ^\prime &\in X _h , \label{eqn:discrete_saddle_x}\\
  b ( x _h , y _h ^\prime ) &= G ( y _h ^\prime ) , \quad &\forall y _h ^\prime &\in Y _h \label{eqn:discrete_saddle_y}.
\end{alignat}
\end{subequations}
Since the local FEEC solvers \eqref{eqn:feec_local} corresponding to
$ a ( \cdot , \cdot ) $ are well-posed, for any given $F$ and $ y _h $
we can write the solution to \eqref{eqn:discrete_saddle_x} as
$ x _h = \mathsf{X} _F + \mathsf{X} _{y_h} $, where
\begin{equation*}
  a ( \mathsf{X} _F , x _h ^\prime ) = F ( x _h ^\prime ) , \qquad a ( \mathsf{X} _{y_h} , x _h ^\prime ) = - b ( x _h ^\prime , y _h ) , \qquad \forall x _h ^\prime \in X _h .
\end{equation*}
This is an efficient local computation that may be done
element-by-element in parallel. Substituting this into
\eqref{eqn:discrete_saddle_y} gives a reduced problem involving only the global
variables: Find $ y _h \in Y _h $ satisfying
\begin{equation}
  \label{eqn:sc}
  b ( \mathsf{X} _{y _h} , y _h  ^\prime ) = G( y _h ^\prime ) - b ( \mathsf{X} _F , y _h ^\prime ) , \quad \forall y _h ^\prime \in Y _h .
\end{equation}
This procedure of eliminating variables using local solvers is known
as \emph{static condensation}. Once the condensed system has been
solved for the global variables, the local variables may be recovered
element-by-element, if desired, using the local solvers. Furthermore,
we may use linearity to separate the influence of the individual
components, computing $ \mathsf{X} _F = \mathsf{X} _f $ and
$ \mathsf{X} _{y _h} = \mathsf{X} _{p _h} + \mathsf{X} _{\overline{u}
  _h} + \mathsf{X} _{\widehat{ \sigma } ^{\mathrm{tan}} _h} +
\mathsf{X} _{\widehat{ u } ^{\mathrm{tan}} _h} $.

Given a finite element basis, \eqref{eqn:discrete_saddle} may also be written
in the block-matrix form
\begin{equation*}
  \begin{bmatrix}
    A & B ^T \\
    B 
  \end{bmatrix}
  \begin{bmatrix}
    x _h \\
    y _h 
  \end{bmatrix} =
  \begin{bmatrix}
    F _h \\
    G _h 
  \end{bmatrix}.
\end{equation*} 
Since the matrix $A$ corresponds to the local solvers
\eqref{eqn:feec_local}, it has a block-diagonal structure, with blocks
corresponding to each $ K \in \mathcal{T} _h $, and can therefore be
inverted efficiently block-by-block. Given $F$ and $ y _h $, we can
locally solve
\begin{equation*}
  A \mathsf{X} _F = F _h , \qquad A \mathsf{X} _{ y _h } = - B ^T y _h \quad \Longrightarrow \quad x _h = \mathsf{X} _F + \mathsf{X} _{ y _h } = A ^{-1} F _h - A ^{-1} B ^T y _h .
\end{equation*}
Substituting this expression into $ B x _h = G _h $ gives the condensed system
\begin{equation*}
  - B A ^{-1} B ^T y _h = G _h - B A ^{-1} F _h ,
\end{equation*}
which is the matrix representation of the condensed variational
problem \eqref{eqn:sc}. Here, the condensed stiffness
matrix $ - B A ^{-1} B ^T $ is precisely the \emph{Schur complement}
of the original stiffness matrix $ \bigl[ \begin{smallmatrix}
  A & B ^T \\
  B
\end{smallmatrix} \bigr] $.

\begin{remark}
  The classical static condensation technique of \citet{Guyan1965} did
  not use hybridization, but simply partitioned the matrix system into
  blocks corresponding to internal and facet degrees of freedom, then
  applied the Schur complement approach above to eliminate the
  interior degrees of freedom. A similar approach has been applied to
  edge elements for Maxwell's equations, as discussed in the survey by
  \citet[Section 4.5]{LeMo2005}. The discovery of the relationship
  between \citeauthor{Guyan1965}'s static condensation and
  hybridization is more recent, cf.~\citet{Cockburn2016}.
\end{remark}

The next result proves that in full generality---without assumptions
on the topology of $ K \in \mathcal{T} _h $ or the elements used---the
condensed system \eqref{eqn:sc} on
$ Y _h = \mathfrak{H} _h ^k \times \overline{ \mathfrak{H} } _h ^k
\times \widehat{ V } _h ^{ k -1 , \mathrm{tan}} \times \widehat{ V }
_h ^{ k , \mathrm{tan}} $ is as small or smaller than the standard
FEEC system \eqref{eqn:feec} on
$ V _h ^{ k -1 } \times V _h ^k \times \mathfrak{H} _h ^k $ without
hybridization. Since the space $ \mathfrak{H} _h ^k $ appears in both
systems, it suffices to compare
$ \dim \overline{ \mathfrak{H} } _h ^k + \dim \widehat{ V } _h ^{ k -1
  , \mathrm{tan} } + \dim \widehat{ V } _h ^{ k , \mathrm{tan} } $
(condensed) with $ \dim V _h ^{k-1} + \dim V _h ^k $ (standard FEEC).

\begin{theorem}
  \label{thm:dof}
  We have the equality
  \begin{equation}
    \label{eqn:dof}
    \begin{multlined}
      (\dim V _h ^{ k -1 } + \dim V _h ^k) - 
      (\dim \overline{ \mathfrak{H}  } _h ^k + \dim \widehat{ V } _h ^{k-1,\mathrm{tan}}  + \dim \widehat{ V } _h ^{k,\mathrm{tan}} )\\
      = \sum _{ K \in \mathcal{T} _h } \bigl( \dim \mathring{ W } _h ^{ k -1 } (K) + \dim \mathring{ \mathfrak{B}  } _h ^k (K) + \dim \mathring{ \mathfrak{Z}  } _h ^{k\perp} (K) \bigr) .
    \end{multlined}
  \end{equation}
  Consequently, the size of the hybridized and condensed FEEC system
  \eqref{eqn:sc} is always less than or equal to that of the standard
  FEEC system \eqref{eqn:feec}, with equality if and only if
  $ \mathring{ W } _h ^{ k -1 } (K) $ is trivial and
  $ \mathring{ W } _h ^k (K) = \mathring{ \mathfrak{H} } _h ^k (K) $
  for all $ K \in \mathcal{T} _h $.
\end{theorem}

\begin{proof}
  By definition, $ \widehat{ V } _h ^{ k, \mathrm{tan}} $ is the image
  of $ V _h ^k $ under the tangential trace
  map. Therefore, the rank-nullity theorem implies that their
  dimensions differ by the dimension of the kernel, i.e.,
  \begin{equation*}
    \dim V _h ^k - \dim \widehat{ V } _h ^{ k, \mathrm{tan} }
    = \dim \{ v _h \in V _h ^k : v _h ^{\mathrm{tan}} = 0 \} 
    = \dim \prod _{ K \in \mathcal{T} _h } \mathring{ W } _h ^k (K) 
    = \sum _{ K \in \mathcal{T} _h } \dim \mathring{ W } _h ^k (K) .
  \end{equation*}
  Applying the discrete Hodge decomposition to each
  $ \mathring{ W } _h ^k (K) $ and using
  $ \overline{ \mathfrak{H} } _h ^k \coloneqq \prod _{ K \in \mathcal{T} _h }
  \mathring{ \mathfrak{H} } _h ^k (K) $ gives
  \begin{equation*}
    \sum _{ K \in \mathcal{T} _h } \dim \mathring{ W } _h ^k (K) =  \dim \overline{ \mathfrak{H}  } _h ^k + \sum _{ K \in \mathcal{T} _h } \bigl( \dim \mathring{ \mathfrak{B}  } _h ^k (K) + \dim \mathring{ \mathfrak{Z}  } _h ^{k\perp} (K) \bigr) .
  \end{equation*}
  Combining this with the previous expression and the corresponding
  one for
  $ \dim V _h ^{k-1} - \dim \widehat{ V } _h ^{ k - 1, \mathrm{tan} }
  $ implies \eqref{eqn:dof}, which completes the proof.
\end{proof}

We now give an explicit count of the reduced degrees of freedom when
$ \mathcal{T} _h $ is a simplicial mesh and
$ \mathcal{P} _r ^\pm \Lambda $ elements are used. \citet*[Sections
4.5--4.6]{ArFaWi2006} show that for $ r \geq 1 $,
\begin{equation*}
  \dim \mathring{ \mathcal{P} } _r \Lambda ^k (K) = \binom{r-1}{n-k}\binom{r+k}{k}, \qquad \dim \mathring{ \mathcal{P} } _r ^- \Lambda ^k (K) = \binom{n}{k}\binom{r+k-1}{n},
\end{equation*}
with the convention that $ \binom{a}{b} = 0 $ when $ b < 0 $ or
$ b > a $. Applying these formulas to the stable pairs of spaces for
FEEC given in \eqref{eqn:stable_pairs}, we get
\begin{align*}
  \dim \mathring{ \mathcal{P} } _{ r + 1 } \Lambda ^{k-1} (K)
  &= \binom{r}{n-k+1}\binom{r+k}{k-1} ,
  & \dim \mathring{ \mathcal{P} } _{ r } \Lambda ^{k} (K)
  &= \binom{r-1}{n-k}\binom{r+k}{k} \quad (\text{if $ r \geq 1 $}),\\
  \dim \mathring{ \mathcal{P} } _{ r + 1 } ^- \Lambda ^{k-1} (K)
  &= \binom{n}{k-1}\binom{r+k-1}{n},
  & \dim \mathring{ \mathcal{P} } _{ r + 1 } ^- \Lambda ^{k} (K)
  &= \binom{n}{k}\binom{r+k}{n}.
\end{align*}
For each $ K \in \mathcal{T} _h $, these formulas count the number of
internal degrees of freedom, which are precisely the ones eliminated
by static condensation.

Since simplices are contractible, the local harmonic spaces are
trivial, except for
$ \mathring{ \mathfrak{H} } _h ^n (K) \cong \mathbb{R} $. When
$ k = n $, static condensation \emph{introduces} one global degree of
freedom per simplex, so in this case, the number of degrees of freedom
is reduced if and only if $ r \geq 1 $. When $ r = 0 $ (i.e., the
lowest-order RT and BDM methods), the degrees of freedom for $ u _h $
are simply replaced by those for $ \overline{ u } _h $.

By checking when the spaces above have dimension greater than zero, we
immediately obtain the following corollary to \cref{thm:dof}.

\begin{corollary}
  Let $ \mathcal{T} _h $ be a simplicial mesh and
  $ V _h ^{ k -1 } $, $ V _h ^k $ be one of
  the stable pairs in \eqref{eqn:stable_pairs}. The hybridized and
  condensed FEEC system \eqref{eqn:sc} is strictly smaller than the
  standard FEEC system \eqref{eqn:feec} if and only if $ r \geq 1 $
  and either
  \begin{itemize}
  \item $ V _h ^k = \mathcal{P} _r \Lambda ^k ( \mathcal{T} _h ) $
    with $ r \geq n - k + 1 $, or
  \item
    $ V _h ^k = \mathcal{P} _{r+1} ^- \Lambda ^k ( \mathcal{T} _h ) $
    with $ r \geq n - k $.
  \end{itemize} 
\end{corollary}

\subsection{Error estimates for the hybrid variables}
\label{sec:error_estimates}

Let $ \{ \mathcal{T} _h \} $ be a shape-regular (but not necessarily
quasi-uniform) family of simplicial meshes of $\Omega$, where $ h _K $
denotes the diameter of $ K \in \mathcal{T} _h $ and
$ h \coloneqq \max _{ K \in \mathcal{T} _h } h _K $. We assume again
that $ V _h ^{ k -1 } $, $ V _h ^k $ is one of the stable pairs
\eqref{eqn:stable_pairs}. Error estimates are already known for
$ \sigma $, $u$, $p$ (\citet*{ArFaWi2006,ArFaWi2010}), and for
$ \overline{ u } $ when $ k = n $ (\citet*{DoRo1985,BrDoMa1985}), so
it only remains to prove estimates for the tangential and normal
traces.

The tangential traces are straightforward, since
$ \widehat{ \sigma } _h ^{\mathrm{tan}} = \sigma _h ^{\mathrm{tan}} $
and $ \widehat{ u } _h ^{\mathrm{tan}} = u _h ^{\mathrm{tan}} $. We
introduce a scaled version of the tangential trace norm from
\cref{sec:traces},
\begin{equation*}
  \lVERT \widehat{ \tau } ^{\mathrm{tan}} \rVERT _{\mathrm{tan}, \partial K } ^2 \coloneqq \inf \bigl\{ \lVert \tau \rVert _K ^2 + h _K ^2 \lVert \mathrm{d} \tau \rVert _K ^2 : \tau ^{\mathrm{tan}} = \widehat{ \tau } ^{\mathrm{tan}} \bigr\} ,
\end{equation*} 
and denote
$ \lVERT \cdot \rVERT _{ \mathrm{tan}, \partial \mathcal{T} _h } ^2
\coloneqq \sum _{ K \in \mathcal{T} _h } \lVERT \cdot \rVERT
_{\mathrm{tan}, \partial K } ^2 $. It is an easy consequence that the
errors for $ \sigma _h ^{\mathrm{tan}} $ and $ u _h ^{\mathrm{tan}} $
are controlled by those for $\sigma _h $ and $u _h $, which we now
state as a proposition.

\begin{proposition}
  For each $ K \in \mathcal{T} _h $, we have
  \begin{align*}
    \lVERT \sigma ^{\mathrm{tan}} - \sigma _h ^{\mathrm{tan}} \rVERT _{\mathrm{tan}, \partial K } ^2 &\leq  \lVert \sigma - \sigma _h \rVert _K ^2 + h _K ^2 \bigl\lVert \mathrm{d} ( \sigma - \sigma _h ) \bigr\rVert _K ^2 ,\\
    \lVERT u ^{\mathrm{tan}} - u _h ^{\mathrm{tan}} \rVERT _{\mathrm{tan}, \partial K } ^2 &\leq  \lVert u - u _h \rVert _K ^2 + h _K ^2 \bigl\lVert \mathrm{d} ( u - u _h ) \bigr\rVert _K ^2 .
  \end{align*}
  Consequently,
  \begin{align*}
    \lVERT \sigma ^{\mathrm{tan}} - \sigma _h ^{\mathrm{tan}} \rVERT _{\mathrm{tan}, \partial \mathcal{T} _h } ^2 &\leq \lVert \sigma - \sigma _h \rVert ^2 _\Omega + h ^2 \bigl\lVert \mathrm{d} ( \sigma - \sigma _h ) \bigr\rVert _\Omega ^2 ,\\
    \lVERT u ^{\mathrm{tan}} - u _h ^{\mathrm{tan}} \rVERT ^2 _{\mathrm{tan}, \partial \mathcal{T} _h } &\leq \lVert u - u _h \rVert ^2 _\Omega + h ^2 \bigl\lVert \mathrm{d} ( u - u _h ) \bigr\rVert ^2 _\Omega .
  \end{align*}
\end{proposition}

\begin{proof}
  The first pair of inequalities follows immediately from the fact
  that the scaled tangential trace norm is an infimum, and the second
  pair follows by summing over $ K \in \mathcal{T} _h $.
\end{proof}

Given sufficient elliptic regularity, the estimates of
\citet*{ArFaWi2010} now imply
\begin{align*}
  \lVERT \sigma ^{\mathrm{tan}} - \sigma _h ^{\mathrm{tan}} \rVERT _{\mathrm{tan}, \partial \mathcal{T} _h } 
  &\lesssim
  \begin{cases}
    h ^{ r + 2 } \lVert f \rVert _{ r + 1, \Omega } , & \text{if } V _h ^{ k -1 } = \mathcal{P} _{r+1} \Lambda ^{ k -1 } (\mathcal{T} _h ), \\
    h ^{ r + 1 } \lVert f \rVert _{ r, \Omega } , & \text{if } V _h ^{ k -1 } =
    \mathcal{P} _{r+1} ^-\Lambda ^{ k -1 } (\mathcal{T} _h ),
  \end{cases}\\
  \lVERT u ^{\mathrm{tan}} - u _h ^{\mathrm{tan}} \rVERT _{ \mathrm{tan}, \partial \mathcal{T} _h }
  &\lesssim
    \begin{cases}
      h \lVert f \rVert _\Omega , & \text{if } V _h ^k = \mathcal{P} _1 ^- \Lambda ^k (\mathcal{T} _h ) ,\\
      h ^{ r + 1 } \lVert f \rVert _{ r -1 , \Omega }, & \text{otherwise,}
    \end{cases}
\end{align*}
which is the optimal order allowed by the polynomial degree of the
tangential traces.

We next give estimates for the normal traces, generalizing an argument
of \citet{ArBr1985} for the hybridized RT method. Recall that
$ \widehat{ u } _h ^{\mathrm{nor}} \in ( \widehat{ W } _h ^{ k - 1 ,
  \mathrm{tan} } ) ^\ast $ and
$ \widehat{ \rho } _h ^{\mathrm{nor}} \in ( \widehat{ W } _h ^{ k ,
  \mathrm{tan} } ) ^\ast $, so we compare them to the natural
projections
$ \widehat{ P } _h u ^{\mathrm{nor}} \in ( \widehat{ W } _h ^{ k - 1 ,
  \mathrm{tan} } ) ^\ast $ and
$ \widehat{ P } _h \rho ^{\mathrm{nor}} \in ( \widehat{ W } _h ^{ k ,
  \mathrm{tan} } ) ^\ast $ defined by
\begin{alignat*}{2}
  \langle \widehat{ P } _h u ^{\mathrm{nor}} , \widehat{ \tau } _h
  ^{\mathrm{tan}} \rangle _{ \partial \mathcal{T} _h } &= \langle u
  ^{\mathrm{nor}} , \widehat{ \tau } _h ^{\mathrm{tan}} \rangle _{
    \partial \mathcal{T} _h } ,\quad &\forall \widehat{ \tau } _h &\in
  \widehat{ W } _h ^{
    k -1 , \mathrm{tan} },\\
  \langle \widehat{ P } _h \rho ^{\mathrm{nor}} , \widehat{ v } _h
  ^{\mathrm{tan}} \rangle _{ \partial \mathcal{T} _h } &= \langle \rho
  ^{\mathrm{nor}} , \widehat{ v } _h ^{\mathrm{tan}} \rangle _{
    \partial \mathcal{T} _h } ,\quad &\forall \widehat{ v } _h &\in
  \widehat{ W } _h ^{ k, \mathrm{tan} }.
\end{alignat*}
If we simply identify $ \widehat{ u } _h ^{\mathrm{nor}} $ with the
corresponding element of
$ \widehat{ W } _h ^{ k - 1 , \mathrm{tan} } \subset L ^2 \Lambda ^{ k
  -1 } ( \partial \mathcal{T} _h ) $, we generally \emph{do not}
observe convergence to the unprojected $ u ^{\mathrm{nor}} $, and
likewise for $ \widehat{ \rho } _h ^{\mathrm{nor}} $ and
$ \rho ^{\mathrm{nor}} $. The reason is that the identification of
$ \widehat{ u } _h ^{\mathrm{nor}} $ with an element of
$ L ^2 \Lambda ^{ k -1 } ( \partial \mathcal{T} _h ) $ is only unique
up to the annihilator
$ (\widehat{ W } _h ^{ k -1 , \mathrm{tan} }) ^\perp $. Therefore, we
should really measure the $ L ^2 $ error after quotienting by the
annihilator, which is equivalent to taking the projections above.  We
define the scaled $ L ^2 $ norm
$ \lVERT \cdot \rVERT _{ \partial K } \coloneqq h _K ^{ 1/2 } \lVert
\cdot \rVert _{ \partial K } $ and denote
$ \lVERT \cdot \rVERT _{ \partial \mathcal{T} _h } ^2 \coloneqq \sum
_{ K \in \mathcal{T} _h } \lVERT \cdot \rVERT _{ \partial K } ^2 $.

\begin{theorem}
  \label{thm:nor_estimates}
  For each $ K \in \mathcal{T} _h $, we have
  \begin{align*}
    \lVERT \widehat{ P } _h u ^{\mathrm{nor}} - \widehat{ u } _h ^{\mathrm{nor}} \rVERT _{ \partial  K } &\lesssim \lVert  P _h u - u _h \rVert _K + h _K \lVert \sigma - \sigma _h \rVert _K,\\
    \lVERT \widehat{ P } _h \rho ^{\mathrm{nor}} - \widehat{ \rho } _h ^{\mathrm{nor}}  \rVERT _{ \partial K } &\lesssim \bigl\lVert P _h \mathrm{d} ( u - u _h ) \bigr\rVert _K + h _K \Bigl( \bigl\lVert \mathrm{d} ( \sigma - \sigma _h ) \bigr\rVert _K + \lVert p - p _h \rVert _K \Bigr) ,
  \end{align*}
  where $ P _h $ denotes $ L ^2 $ projection onto $ W _h $. Consequently,
  \begin{align*} 
    \lVERT \widehat{ P } _h u ^{\mathrm{nor}} - \widehat{ u } _h ^{\mathrm{nor}} \rVERT _{ \partial  \mathcal{T} _h } &\lesssim \lVert P _h u - u _h \rVert _{\mathcal{T} _h} + h \lVert \sigma - \sigma _h \rVert _\Omega ,\\
 \lVERT \widehat{ P } _h \rho  ^{\mathrm{nor}} - \widehat{ \rho } _h ^{\mathrm{nor}} \rVERT _{ \partial  \mathcal{T} _h  } &\lesssim \bigl\lVert P _h \mathrm{d} ( u - u _h ) \bigr\rVert _{ \mathcal{T} _h } + h \Bigl( \bigl\lVert \mathrm{d} ( \sigma - \sigma _h ) \bigr\rVert _\Omega + \lVert p - p _h \rVert _\Omega  \Bigr) .
  \end{align*} 
\end{theorem}

\begin{proof}
  A scaling argument shows that each
  $ \widehat{ \tau } _h ^{\mathrm{tan}} \in \widehat{ W } _h ^{ k -1 ,
    \mathrm{tan} } ( \partial K ) $ has an extension
  $ \tau _h \in W _h ^{ k -1 } (K) $ with
  $ \tau _h ^{\mathrm{tan}} = \widehat{ \tau } _h ^{\mathrm{tan}} $
  such that
  \begin{equation*}
    \lVert \tau _h \rVert _K + h _K \lVert \mathrm{d} \tau _h \rVert _K \lesssim \lVERT \widehat{ \tau } _h ^{\mathrm{tan}} \rVERT _{ \partial K } .
  \end{equation*}
  Therefore, subtracting \eqref{eqn:feec-h_tau} from
  \eqref{eqn:dd_tau}, we get
  \begin{align*}
    h _K \langle \widehat{ P } _h u ^{\mathrm{nor}} - \widehat{ u } _h ^{\mathrm{nor}} , \widehat{ \tau } _h ^{\mathrm{tan}} \rangle _{ \partial K }
    &= h _K \langle u ^{\mathrm{nor}} - \widehat{ u } _h ^{\mathrm{nor}} , \tau _h ^{\mathrm{tan}} \rangle _{ \partial K }\\
    &= h _K \Bigl[ - ( \sigma - \sigma _h , \tau _h ) _K  + ( u - u _h , \mathrm{d} \tau _h ) _K \Bigr] \\
    &= h _K \Bigl[ - ( \sigma - \sigma _h , \tau _h ) _K  + (  P _h u - u _h , \mathrm{d} \tau _h )  _K \Bigr] \\
    &\leq \Bigl( h _K \lVert \sigma - \sigma _h \rVert _K + \lVert P _h  u - u _h \rVert _K \Bigr) \Bigl( \lVert \tau _h \rVert _K + h _K \lVert \mathrm{d} \tau _h \rVert _K \Bigr) \\
    &\lesssim \Bigl( h _K \lVert \sigma - \sigma _h \rVert _K + \lVert P  _h u - u _h \rVert _K \Bigr)  \lVERT \widehat{ \tau } _h ^{\mathrm{tan}} \rVERT _{ \partial K } .
  \end{align*}
  Since $ \langle \cdot , \cdot \rangle _{ \partial K } $ agrees with
  the $ L ^2 $ inner product,
  \begin{equation*}
    \lVERT \widehat{ P } _h u ^{\mathrm{nor}} - \widehat{ u } _h ^{\mathrm{nor}} \rVERT _{ \partial  K } = h _K ^{ 1/2 } \sup _{ \widehat{ \tau } ^{\mathrm{tan}} _h \neq 0 } \frac{ \langle \widehat{ P } _h u ^{\mathrm{nor}} - \widehat{ u } _h ^{\mathrm{nor}}, \widehat{ \tau } ^{\mathrm{tan}} _h \rangle _{ \partial K } }{ \lVert \widehat{ \tau } _h ^{\mathrm{tan}} \rVert _{ \partial K } } = \sup _{ \widehat{ \tau } _h ^{\mathrm{tan}} \neq 0 } \frac{ h _K \langle \widehat{ P } _h u ^{\mathrm{nor}} - \widehat{ u } _h ^{\mathrm{nor}}, \widehat{ \tau } ^{\mathrm{tan}} _h \rangle _{ \partial K } }{ \lVERT \widehat{ \tau } _h ^{\mathrm{tan}} \rVERT _{ \partial K } },
  \end{equation*}
  which completes the proof of the first estimate. The estimate for
  $ \lVERT \widehat{ P } _h \rho ^{\mathrm{nor}} - \widehat{ \rho }
  _h ^{\mathrm{nor}} \rVERT _{ \partial K } $ is obtained similarly,
  and the $ \lVERT \cdot \rVERT _{ \partial \mathcal{T} _h } $
  estimates again follow immediately from the
  $ \lVERT \cdot \rVERT _{ \partial K } $ estimates.
\end{proof}

For $ k < n $, we generally cannot improve on
$ \lVert P _h u - u _h \rVert _{\mathcal{T} _h} \leq \lVert u - u _h
\rVert _\Omega $, so assuming sufficient elliptic regularity and
applying the estimates from \citet*{ArFaWi2010} gives
\begin{equation*}
  \lVERT \widehat{ P } _h u ^{\mathrm{nor}} - \widehat{ u } _h ^{\mathrm{nor}} \rVERT _{ \partial  \mathcal{T} _h } \lesssim \begin{cases}
    h \lVert f \rVert _\Omega , & \text{if } V _h ^k = \mathcal{P} _1 ^- \Lambda ^k (\mathcal{T} _h ) ,\\
    h ^{ r + 1 } \lVert f \rVert _{ r -1 , \Omega }, & \text{otherwise,}
  \end{cases}
\end{equation*}
i.e., the convergence rate is the same as that for
$ u _h \rightarrow u $. When $ k = n $, however,
$ \lVert P _h u - u _h \rVert _{\mathcal{T} _h} $ famously
superconverges for the RT and BDM methods
(\citet*{DoRo1985,ArBr1985,BrDoMa1985}). In this case, we recover the
superconvergence results of \citep{ArBr1985,BrDoMa1985} for the
Lagrange multipliers:
\begin{equation*}
  \lVERT \widehat{ P } _h u ^{\mathrm{nor}} - \widehat{ u } _h ^{\mathrm{nor}} \rVERT _{ \partial  \mathcal{T} _h } \lesssim
  \begin{cases}
    h ^2 \lVert f \rVert _{ 1, \Omega } , & \text{if }  r = 0 ,\\
    h ^{r + 3} \lVert f \rVert _{ r + 1 , \Omega } ,& \text{if } r \geq 1,\ V _h ^{ n -1 } = \mathcal{P} _{ r + 1 } \Lambda ^{ n -1 } ( \mathcal{T} _h )  ,\\
    h ^{ r + 2 } \lVert f \rVert _{ r, \Omega } ,& \text{if } r \geq 1,\ V _h ^{ n -1 } = \mathcal{P} _{ r + 1 } ^- \Lambda ^{ n -1 } ( \mathcal{T} _h )  .
  \end{cases}
\end{equation*}
From the perspective of FEEC, this occurs since
$ W _h ^n = V _h ^n = \mathfrak{B} _h ^n $, so
$ \lVert P _h u - u _h \rVert _{ \mathcal{T} _h } = \bigl\lVert P _{
  \mathfrak{B} _h } ( u - u _h ) \bigr\rVert _\Omega $, which
superconverges according to \citep[Lemma 3.13]{ArFaWi2010}. On the
other hand, when $ k < n $, the error is dominated by the
nonvanishing $ \mathfrak{Z} _h ^{ k \perp } $ component \citep[Lemma
3.16]{ArFaWi2010}, so there is no improvement.

Similarly, when $ k < n - 1 $, we generally cannot do better than
$ \bigl\lVert P _h \mathrm{d} ( u - u _h ) \bigr\rVert _{ \mathcal{T}
  _h } \leq \bigl\lVert \mathrm{d} ( u - u _h ) \bigr\rVert _\Omega $,
so assuming sufficient elliptic regularity,
\begin{equation*}
  \lVERT \widehat{ P } _h \rho  ^{\mathrm{nor}} - \widehat{ \rho } _h ^{\mathrm{nor}} \rVERT _{ \partial  \mathcal{T} _h  } \lesssim \begin{cases}
    h ^{r+1} \lVert f \rVert _{r, \Omega} , & \text{if } V _h ^k = \mathcal{P} _{r+1} ^- \Lambda ^k (\mathcal{T} _h ) ,\\
    h ^r \lVert f \rVert _{ r -1 , \Omega }, & \text{if } V _h ^k = \mathcal{P} _r \Lambda ^k (\mathcal{T} _h ) ,
  \end{cases}
\end{equation*}
and the convergence rate is the same as that for
$ \mathrm{d} u _h \rightarrow \mathrm{d} u $.  However, when
$ k = n - 1 $, we obtain superconvergence as a consequence of the
following lemma (which holds for all $k$, not just $ k = n -1 $).

\begin{lemma}
  \label{lem:Bh}
  The FEEC solution \eqref{eqn:feec} satisfies
  $ \bigl\lVert P _{ \mathfrak{B} _h } \mathrm{d} (u - u _h )
  \bigr\rVert _\Omega \lesssim h \Bigl( \bigl\lVert \mathrm{d} (
  \sigma - \sigma _h ) \bigr\rVert _\Omega + \lVert p - p _h \rVert
  _\Omega \Bigr) $.
\end{lemma}

\begin{proof}
  The argument is similar to \citep[Lemma 3.15]{ArFaWi2010}.  Let
  $ v _h \in \mathfrak{Z} _h ^{k \perp} $ be such that
  $ \mathrm{d} v _h = P _{ \mathfrak{B} _h } \mathrm{d} ( u - u _h )
  $, and take $v = P _{ \mathfrak{Z} ^\perp } v _h $. Since $v$ is
  orthogonal to $ \mathrm{d} ( \sigma - \sigma _h ) $ and
  $ p - p _h $, subtracting \eqref{eqn:feec_v} from
  \eqref{eqn:hodge-laplace_v} gives
  \begin{align*}
    \bigl\lVert P _{ \mathfrak{B} _h } \mathrm{d} (u - u _h )
    \bigr\rVert _\Omega ^2
    &= \bigl( \mathrm{d} ( u - u _h ) , \mathrm{d} v _h \bigr) _\Omega \\
    &= \bigl( \mathrm{d} ( \sigma - \sigma _h ) + ( p - p _h ) , v - v _h \bigr) _\Omega \\
    &\leq \Bigl( \bigl\lVert \mathrm{d} ( \sigma - \sigma _h ) \bigr\rVert _\Omega + \lVert p - p _h \rVert _\Omega \Bigr) \lVert v - v _h \rVert _\Omega \\
    &\lesssim h \Bigl( \bigl\lVert \mathrm{d} ( \sigma - \sigma _h ) \bigr\rVert _\Omega + \lVert p - p _h \rVert _\Omega \Bigr) \bigl\lVert P _{ \mathfrak{B} _h } \mathrm{d} (u - u _h ) \bigr\rVert _\Omega .
  \end{align*}
  The last step uses \citep[Lemma 3.12]{ArFaWi2010}, which says that
  $ \lVert v - v _h \rVert _\Omega \lesssim h \lVert \mathrm{d} v _h
  \rVert _\Omega $.
\end{proof}

\begin{corollary}
  \label{cor:rhonor_superconvergence}
  For $ k = n - 1 $, we have the improved estimate
  \begin{equation*}
    \lVERT \widehat{ P } _h \rho  ^{\mathrm{nor}} - \widehat{ \rho } _h ^{\mathrm{nor}} \rVERT _{ \partial  \mathcal{T} _h  } \lesssim h \Bigl( \bigl\lVert \mathrm{d} ( \sigma - \sigma _h ) \bigr\rVert _\Omega + \lVert p - p _h \rVert _\Omega  \Bigr) .
  \end{equation*}
  In particular, when
  $ f \in \mathring{ \mathfrak{B} } _{ n -1 } ^\ast $, we have
  $ \widehat{ \rho } _h ^{\mathrm{nor}} = \widehat{ P } _h \rho
  ^{\mathrm{nor}} $ exactly.
\end{corollary}

\begin{proof}
  Since
  $ \bigl\lVert P _h \mathrm{d} ( u - u _h ) \bigr\rVert _{
    \mathcal{T} _h } = \bigl\lVert P _{ \mathfrak{B} _h } \mathrm{d} (
  u - u _h ) \bigr\rVert _\Omega $ when $ k = n -1 $, the improved
  estimate is immediate from \cref{thm:nor_estimates} and
  \cref{lem:Bh}. In particular, $\sigma$ and $p$ vanish when
  $ f \in \mathring{ \mathfrak{B} } _{ n -1 } ^\ast $, so in that case
  the left-hand side is identically zero.
\end{proof}

Assuming sufficient elliptic regularity, this gives the superconvergent rates
\begin{equation*}
  \lVERT \widehat{ P } _h \rho  ^{\mathrm{nor}} - \widehat{ \rho } _h ^{\mathrm{nor}} \rVERT _{ \partial  \mathcal{T} _h  } \lesssim \begin{cases}
    0 , & \text{if } f \in \mathring{ \mathfrak{B}  } _{ n -1 } ^\ast ,\\
    h ^{ r + 2 } \lVert f \rVert _{ r + 1 , \Omega } , & \text{otherwise.}
  \end{cases}
\end{equation*}

\section{Postprocessing}
\label{sec:postprocessing}

In this section, we introduce a local postprocessing procedure, which
generalizes that of \citet{Stenberg1991} from $ k = n $ to arbitrary
$k$. We develop new error estimates for the postprocessed solution
when $ k < n $; in particular, postprocessing gives a superconvergent
approximation $ \rho _h ^\ast $ to $ \mathrm{d} u $ for $ k = n -1 $,
and $ \delta \rho _h ^\ast $ is an improved approximation to
$ \delta \mathrm{d} u $ for all $k$. Finally, we discuss how this
analysis corresponds to that of \citet{Stenberg1991} in the case
$ k = n $, giving superconvergence of $ u _h ^\ast $ to $u$.

\subsection{The postprocessing procedure}

To motivate the proposed procedure, recall that the exact local solver
\eqref{eqn:local_solver} corresponds to solving
$ L u + \overline{p} = f - p $ such that
$ P _{ \overline{ \mathfrak{H} } } u = \overline{u} $, with tangential
boundary conditions given by $ \widehat{ \sigma } ^{\mathrm{tan}} $
and $ \widehat{ u } ^{\mathrm{tan}} $. Instead of writing this as a
variational problem on the $ \mathring{ H } \Lambda (K) $ complex, we
can equivalently write it on the $ H ^\ast \Lambda (K) $ complex as
\begin{subequations}
  \label{eqn:pp_exact}
\begin{alignat}{2}
  ( \rho, \eta ) _K - ( u , \delta \eta ) _K &= \langle \widehat{ u } ^{\mathrm{tan}} , \eta ^{\mathrm{nor}} \rangle _{ \partial K } , \quad &\forall \eta &\in H ^\ast \Lambda ^{ k + 1 } (K) ,\\
  ( \delta \rho , v ) _K + ( \delta u , \delta v ) _K + ( \overline{p}, v ) _K &= ( f - p,  v ) _K - \langle \widehat{ \sigma } ^{\mathrm{tan}} , v ^{\mathrm{nor}} \rangle _{ \partial K } , \quad &\forall v &\in H ^\ast \Lambda ^k (K) ,\\
  ( u, \overline{q} ) _K &= ( \overline{ u }, \overline{q} ) _K , \quad &\forall \overline{q} &\in \mathring{ \mathfrak{H}  } ^k (K) ,
\end{alignat}
\end{subequations}
where the tangential boundary conditions are now natural rather than
essential. As before, we have $ \sigma = \delta u $ and
$ \rho = \mathrm{d} u $

The postprocessing procedure is based on approximating
\eqref{eqn:pp_exact} on a finite-dimensional subcomplex
$ W _h ^\ast (K) \subset H ^\ast \Lambda (K) $, meaning
$ \delta W _h ^{\ast k+1} (K) \subset W _h ^{\ast k} (K) $. Since
$ \star H ^\ast \Lambda ^k (K) = H \Lambda ^{ n - k } (K) $, an
equivalent condition is that
$ \star W _h ^\ast (K) \subset H \Lambda (K) $ is a subcomplex.
Moreover,
$ \pi _h \colon H \Lambda (K) \rightarrow \star W _h ^\ast (K) $ is a
bounded commuting projection if and only if
$ \star ^{-1} \pi _h \star \colon H ^\ast \Lambda (K) \rightarrow W _h
^\ast (K) $ is. For a simplicial mesh, we may therefore take
\begin{equation*}
  \star W _h ^{ \ast k + 1 } (K) = \mathcal{P} _{ r ^\ast + 1 } ^\pm \Lambda ^{ n - k -1 } (K) , \qquad  \star W _h ^{ \ast k } (K) =  \begin{Bmatrix}
    \renewcommand\arraystretch{2} \mathcal{P} _{r ^\ast} \Lambda ^{n-k} ( K ) \text{ (if
      $ r ^\ast \geq 1 $)} \\[0.5ex]
    \text{or}\\[0.5ex]
    \mathcal{P} _{ r ^\ast + 1 } ^- \Lambda ^{n-k} ( K )
  \end{Bmatrix}.
\end{equation*}
This is just the Hodge dual of the stable pairs
\eqref{eqn:stable_pairs} with $k$ replaced by $ n - k $ and $r$ by
$ r ^\ast $, so all of the results of \citet*{ArFaWi2010} apply
immediately to the dual problem. We write the discrete Hodge
decomposition for this complex as
\begin{equation*}
  W _h ^{ \ast k } (K) = \mathfrak{B}  _h ^{ \ast k } (K) \oplus \mathfrak{H} _h ^{\ast k} (K) \oplus \mathfrak{Z}  _h ^{ \ast k \perp } (K) .
\end{equation*} 
When $K$ is contractible (e.g., a simplex), we have
$ \mathfrak{H} _h ^{\ast k} (K) = \mathring{ \mathfrak{H} } ^k (K)
$, which is $ \cong \mathbb{R} $ for $ k = n $ and trivial otherwise.

We are now ready to define the postprocessing procedure on
$ K \in \mathcal{T} _h $: Find
$ \rho _h ^\ast \in W _h ^{ \ast k + 1 } (K) $,
$ u _h ^\ast \in W _h ^{ \ast k } (K) $,
$ \overline{p} _h ^\ast \in \mathfrak{H} _h ^{ \ast k } (K) $ such
that
\begin{subequations}
  \label{eqn:pp}
\begin{alignat}{2}
  ( \rho _h ^\ast , \eta _h  ) _K - ( u _h ^\ast , \delta \eta _h ) _K &= \langle \widehat{ u } _h ^{\mathrm{tan}} , \eta _h ^{\mathrm{nor}} \rangle _{ \partial K } , \quad &\forall \eta _h &\in W _h ^{ \ast k + 1 } (K) ,\\
  ( \delta \rho _h ^\ast , v _h  ) _K + ( \delta u _h ^\ast  , \delta v _h  ) _K + ( \overline{p} _h ^\ast , v _h ) _K &= ( f - p _h ,  v _h ) _K - \langle \widehat{ \sigma } _h ^{\mathrm{tan}} , v _h ^{\mathrm{nor}} \rangle _{ \partial K } , \quad &\forall v _h &\in W _h ^{ \ast k } (K) ,\\
  ( u _h ^\ast , \overline{q} _h ) _K &= ( \overline{ u } _h , \overline{q}_h  ) _K ,\quad &\forall \overline{q} _h &\in \mathfrak{H} _h ^{ \ast k } (K) \label{eqn:pp_q}.
\end{alignat}
\end{subequations}

\begin{remark}
  The right-hand side only depends on the global variables $ p _h$,
  $\overline{ u } _h $, $\widehat{ \sigma } _h ^{\mathrm{tan}} $,
  $\widehat{ u } _h ^{\mathrm{tan}} $. Therefore, after we solve the
  statically condensed problem \eqref{eqn:sc}, this procedure can be
  used as an alternative to the local solvers \eqref{eqn:feec_local}
  for recovering approximations to the local variables on
  $ K \in \mathcal{T} _h $.

  We can also apply postprocessing if FEEC is implemented using
  \eqref{eqn:feec}, without hybridization, since
  $ \overline{ u } _h = P _{ \overline{ \mathfrak{H} } _h } u _h $,
  $ \widehat{ \sigma } _h ^{\mathrm{tan}} = \sigma _h ^{\mathrm{tan}}
  $, and $ \widehat{ u } _h ^{\mathrm{tan}} = u _h ^{\mathrm{tan}}
  $. In the simplicial case, since
  $ \mathfrak{H} _h ^{ \ast k } (K) = \mathring{ \mathfrak{H} } _h ^k
  (K) $, we can simply replace $ \overline{ u } _h $ by $ u _h $ on
  the right-hand side of \eqref{eqn:pp_q} without projecting.

  Note that, while the original solution variables are tangentially
  continuous between elements, the postprocessed solution variables
  generally do not have any tangential or normal continuity, i.e.,
  they are neither $ H \Lambda (\Omega) $- nor
  $ H ^\ast \Lambda (\Omega) $-conforming.
\end{remark}

\begin{example}[Stenberg postprocessing]
  When $ k = n $ and $ \mathcal{T} _h $ is a simplicial mesh, the
  space $ W _h ^{ \ast n + 1 } (K) $ is trivial,
  $ W _h ^{\ast n } (K) \cong \mathcal{P} _{ r ^\ast } (K) $, and
  $ \mathfrak{H} _h ^{ \ast n } (K) \cong \mathbb{R} $. Therefore,
  \eqref{eqn:pp} becomes
  \begin{alignat*}{2}
    ( \operatorname{grad} u _h ^\ast  , \operatorname{grad}  v _h  ) _K + ( \overline{p} _h ^\ast , v _h ) _K &= ( f ,  v _h ) _K - \langle \widehat{ \sigma } _h ^{\mathrm{tan}} , v _h \normal \rangle _{ \partial K } , \quad &\forall v _h &\in  \mathcal{P} _{ r ^\ast } (K) ,\\
    ( u _h ^\ast , \overline{q} _h ) _K &= ( \overline{ u } _h ,
    \overline{q}_h ) _K , \quad &\forall \overline{q} _h &\in \mathbb{R} ,
  \end{alignat*}
  which coincides with \citet{Stenberg1991} postprocessing for the RT
  and BDM methods. Stenberg also considered a second form of
  postprocessing with
  $ \overline{ p } _h ^\ast , \overline{ q } _h \in \mathcal{P} _r (K)
  $, but we do not consider that here.
\end{example}

\subsection{Error estimates for $ k < n $}
\label{sec:postprocessing_estimates}

We now analyze this postprocessing procedure when, as before,
$ \{ \mathcal{T} _h \} $ is a shape-regular family of simplicial
meshes of $ \Omega $. We wish to determine the accuracy of the
solution to the postprocessing problem \eqref{eqn:pp}, compared to
that obtained using the local solvers \eqref{eqn:feec_local}.

The $ k = n $ case has already been analyzed by \citet{Stenberg1991},
so we restrict our attention to $ k < n $. Since the local harmonic
spaces are trivial, the exact solver \eqref{eqn:pp_exact} simplifies
to
\begin{subequations}
  \label{eqn:pp_exact_k<n}
\begin{alignat}{2}
  ( \rho, \eta ) _K - ( u , \delta \eta ) _K &= \langle \widehat{ u } ^{\mathrm{tan}} , \eta ^{\mathrm{nor}} \rangle _{ \partial K } , \quad &\forall \eta &\in H ^\ast \Lambda ^{ k + 1 } (K) ,\\
  ( \delta \rho , v ) _K + ( \delta u , \delta v ) _K &= ( f - p,  v ) _K - \langle \widehat{ \sigma } ^{\mathrm{tan}} , v ^{\mathrm{nor}} \rangle _{ \partial K } , \quad &\forall v &\in H ^\ast \Lambda ^k (K) ,
\end{alignat}
\end{subequations}
and the postprocessing problem \eqref{eqn:pp} simplifies to
\begin{subequations}
  \label{eqn:pp_k<n}
\begin{alignat}{2}
  ( \rho _h ^\ast , \eta _h  ) _K - ( u _h ^\ast , \delta \eta _h ) _K &= \langle \widehat{ u } _h ^{\mathrm{tan}} , \eta _h ^{\mathrm{nor}} \rangle _{ \partial K } , \quad &\forall \eta _h &\in W _h ^{ \ast k + 1 } (K) , \label{eqn:pp_k<n_eta}\\
  ( \delta \rho _h ^\ast , v _h  ) _K + ( \delta u _h ^\ast  , \delta v _h  ) _K &= ( f - p _h ,  v _h ) _K - \langle \widehat{ \sigma } _h ^{\mathrm{tan}} , v _h ^{\mathrm{nor}} \rangle _{ \partial K } , \quad &\forall v _h &\in W _h ^{ \ast k } (K) \label{eqn:pp_k<n_v}.
\end{alignat}
\end{subequations}
To aid in the analysis, we introduce the intermediate approximation
$ \widetilde{ \rho } _h \in W _h ^{ \ast k + 1 } (K) $,
$ \widetilde{ u } _h \in W _h ^{ \ast k } (K) $ such that
\begin{subequations}
  \label{eqn:pp_tilde}
\begin{alignat}{2}
  ( \widetilde{ \rho } _h , \eta _h  ) _K - ( \widetilde{ u } _h , \delta \eta _h ) _K &= \langle \widehat{ u } ^{\mathrm{tan}} , \eta _h ^{\mathrm{nor}} \rangle _{ \partial K } , \quad &\forall \eta _h &\in W _h ^{ \ast k + 1 } (K) , \label{eqn:pp_tilde_eta}\\
  ( \delta \widetilde{ \rho } _h , v _h  ) _K + ( \delta \widetilde{ u } _h  , \delta v _h  ) _K &= ( f - p ,  v _h ) _K - \langle \widehat{ \sigma } ^{\mathrm{tan}} , v _h ^{\mathrm{nor}} \rangle _{ \partial K } , \quad &\forall v _h &\in W _h ^{ \ast k } (K) , \label{eqn:pp_tilde_v}
\end{alignat}
\end{subequations}
where the global variables on the right-hand side are the same as
those in the exact solution \eqref{eqn:pp_exact_k<n}.  Note that
\eqref{eqn:pp_tilde} is just the FEEC approximation of
\eqref{eqn:pp_exact_k<n} on the subcomplex
$ W _h ^\ast (K) \subset H ^\ast \Lambda (K) $, so the results of
\citet*{ArFaWi2010} immediately give us estimates for
$ \rho - \widetilde{ \rho } _h $ and $ u - \widetilde{ u } _h $.  It
therefore remains to analyze the difference between \eqref{eqn:pp_k<n}
and \eqref{eqn:pp_tilde}.

As in \citep{ArFaWi2010}, we assume that the exact solution satisfies
an elliptic regularity estimate of the form
\begin{equation*}
  \lVert u \rVert _{ t + 2, \Omega } + \lVert p \rVert _{ t + 2 ,\Omega } + \lVert \mathrm{d} u \rVert _{ t + 1 , \Omega } + \lVert \sigma \rVert _{ t + 1 , \Omega } + \lVert \mathrm{d} \sigma \rVert _{ t , \Omega } \lesssim \lVert f \rVert _{t , \Omega } ,
\end{equation*}
for $ 0 \leq t \leq t _{\max} $, where
$ \lVert \cdot \rVert _{ t, \Omega } $ denotes the $ H ^t $ norm on
$\Omega$. We will frequently invoke \citep[Theorem 3.11]{ArFaWi2010},
which gives $ L ^2 $ error estimates for the FEEC solution in terms of
the best approximation allowed by the regularity of the exact solution
and the polynomial degree of the finite element spaces. These
estimates will be applied both to the original FEEC approximation
\eqref{eqn:feec} on $ V _h $ and to the intermediate approximation
\eqref{eqn:pp_tilde} on $ W _h ^\ast (K) $.

We want the postprocessed solution to be at least as good as the
standard FEEC solution obtained from the local solvers
\eqref{eqn:feec_local}. The following assumptions ensure that
$ r ^\ast $ is large enough for the $ W _h ^\ast (K) $ complex to
approximate the exact solution as well as $ W _h (K) $ does. If
$ f \perp \mathfrak{B} ^k $, then $ \sigma = 0 $, so it is enough for
$ W _h ^{ \ast k } (K) $ to contain the same total space of
polynomials as $ W _h ^k (K) $, i.e., $ r ^\ast \geq r $. Otherwise,
in order to approximate $\sigma \neq 0$, we also need the stronger condition
that $ W _h ^{\ast k-1} (K) $ contains the same total space of
polynomials as $ W _h ^{ k -1 } (K) $.

\begin{assumption}
  \label{assumption}
  Assume that we are in one of the following three cases:
  \begin{enumerate}
  \item $ f \perp \mathfrak{B} ^k $ and $ r ^\ast \geq r $.

  \item
    $ W _h ^{ k -1 } (K) = \mathcal{P} _{ r + 1 } \Lambda ^{ k -1 }
    (K) $ and $ \star W _h ^{\ast k} (K) =
    \begin{cases}
      \mathcal{P} _{ r ^\ast } \Lambda ^{ n
        - k } (K), & r ^\ast \geq r + 2, \\
      \mathcal{P} _{ r ^\ast + 1 } ^- \Lambda ^{ n - k } (K), & r
      ^\ast \geq r + 1.
    \end{cases}$

  \item $ W _h ^{ k -1 } (K) = \mathcal{P} _{ r + 1 } ^- \Lambda ^{ k -1 }
    (K) $ and $ \star W _h ^{\ast k} (K) =
    \begin{cases}
      \mathcal{P} _{ r ^\ast } \Lambda ^{ n
        - k } (K), & r ^\ast \geq r + 1, \\
      \mathcal{P} _{ r ^\ast + 1 } ^- \Lambda ^{ n - k } (K), & r
      ^\ast \geq r .
    \end{cases}$
  \end{enumerate} 
\end{assumption}

Our first result shows that $ \delta \rho _h ^\ast $ gives an improved
approximation of $ \delta \rho = \delta \mathrm{d} u $, compared to
$ \delta \mathrm{d} u _h $. In particular, when
$ f = \delta \rho \in \mathring{ \mathfrak{B} } _k ^\ast $, we can
obtain an arbitrarily good approximation by taking the postprocessing
degree $ r ^\ast $ large enough.

\begin{theorem}
  \label{thm:pp_deltarho}
  For each $ K \in \mathcal{T} _h $ and $ 0 \leq s \leq t _{\max} $, we have
  \begin{align*}
    \bigl\lVert \delta ( \rho - \widetilde{ \rho } _h ) \bigr\rVert _K
    &\lesssim h _K ^s \lVert f \rVert _{ s , K } , \quad \text{if }  s \leq r ^\ast + 1 ,\\
    \bigl\lVert \delta ( \widetilde{ \rho } _h - \rho _h ^\ast ) \bigr\rVert _K &\leq \bigl\lVert \mathrm{d} ( \sigma - \sigma _h ) \bigr\rVert _K + \lVert p - p _h \rVert _K .
  \end{align*}
  Consequently, if \cref{assumption} holds, then
  \begin{equation*}
    \bigl\lVert \delta ( \rho - \rho _h ^\ast ) \bigr\rVert _{ \mathcal{T} _h } \lesssim h ^s \lVert f \rVert _{s , \Omega } , \quad \text{if }
    \begin{cases}
      s \leq r ^\ast + 1 ,& f \in \mathring{ \mathfrak{B}  } _k ^\ast  ,\\
      s \leq r + 1, & \text{otherwise}.
    \end{cases}
  \end{equation*}
\end{theorem}

\begin{proof}
  The first estimate is immediate from \citep[Theorem
  3.11]{ArFaWi2010} applied to the problem \eqref{eqn:pp_tilde}. Next,
  subtracting \eqref{eqn:pp_k<n_v} from \eqref{eqn:pp_tilde_v} with
  $ v _h \in \mathfrak{B} _h ^{ \ast k } (K) $ gives
  \begin{align*}
    \bigl( \delta ( \widetilde{ \rho } _h - \rho _h ^\ast ), v _h \bigr) _K
    &= ( p _h - p ,  v _h ) _K + \langle \widehat{ \sigma } _h ^{\mathrm{tan}} - \widehat{ \sigma } ^{\mathrm{tan}} , v _h ^{\mathrm{nor}} \rangle _{ \partial K } \\
    &= ( p _h - p ,  v _h ) _K + \bigl( \mathrm{d} (\sigma  _h  - \sigma ), v _h \bigr) _K \\
    &\leq \Bigl( \bigl\lVert \mathrm{d} ( \sigma - \sigma _h ) \bigr\rVert _K + \lVert p - p _h \rVert _K \Bigr) \lVert v _h \rVert _K ,
  \end{align*}
  and taking
  $ v _h = \delta ( \widetilde{ \rho } _h - \rho _h ^\ast ) $ implies
  the second estimate. Finally, summing over $ K \in \mathcal{T} _h $
  and applying \citep[Theorem 3.11]{ArFaWi2010} once more gives
  \begin{equation*}
    \bigl\lVert \delta ( \widetilde{ \rho } _h - \rho _h ^\ast ) \bigr\rVert _{ \mathcal{T} _h } \lesssim
    \begin{cases}
      0 , & \text{if } f \in \mathring{ \mathfrak{B} } _k ^\ast ,\\
      h ^s \lVert f \rVert _{ s, \Omega } , & \text{if } s \leq r + 1
      , \text{ otherwise},
    \end{cases}
  \end{equation*}
  so the last estimate follows by \cref{assumption} and the triangle
  inequality.
\end{proof}

The next result says that, generically, $ \delta u _h ^\ast $
approximates $ \sigma = \delta u $ as well as $ \sigma _h $ does, but
no better. In the case $ f \in \mathring{ \mathfrak{B} } _k ^\ast $,
when $ \sigma = \sigma _h = 0 $, we can make $ \delta u _h ^\ast $
arbitrarily small by taking $ r ^\ast $ large enough.

\begin{theorem}
  \label{thm:pp_deltau}
  For each $ K \in \mathcal{T} _h $ and $ 0 \leq s \leq t _{ \max } $,
  we have
  \begin{align*}
    \bigl\lVert \delta ( u - \widetilde{ u } _h ) \bigr\rVert _K &\lesssim h _K ^{ s + 1 } \lVert  f \rVert  _{ s, K } , \quad \text{if }
    \begin{cases}
      s \leq r ^\ast + 1 , & f \in \mathring{ \mathfrak{B}  } _k ^\ast ,\\
      s \leq r ^\ast ,& \star W _h ^{ \ast k } (K) = \mathcal{P} _{r ^\ast + 1 } ^- \Lambda ^{n-k} ( K ),\\
      s \leq r ^\ast - 1 , & \star W _h ^{ \ast k } (K) = \mathcal{P} _{r ^\ast} \Lambda ^{n-k} ( K ),
    \end{cases} \\
    \bigl\lVert \delta ( \widetilde{ u } _h - u _h ^\ast ) \bigr\rVert _K &\lesssim \lVert \sigma - \sigma _h \rVert _K + h _K \Bigl( \bigl\lVert  \mathrm{d} ( \sigma - \sigma _h ) \bigr\rVert _K +  \lVert p - p _h \rVert _K \Bigr) .
  \end{align*}
  Consequently, if \cref{assumption} holds, then
  \begin{equation*}
    \bigl\lVert \delta ( u - u _h ^\ast ) \bigr\rVert _{\mathcal{T} _h}  \lesssim h ^{ s + 1 } \lVert f \rVert _{ s, \Omega } ,\quad 
    \text{if }
    \begin{cases}
      s \leq r ^\ast + 1 , & f \in \mathring{ \mathfrak{B}  } _k ^\ast ,\\
      s \leq r + 1, & V _h ^{ k -1 } = \mathcal{P} _{ r + 1 } \Lambda ^{ k -1 } ( \mathcal{T} _h ), \\
      s \leq r, & V _h ^{ k -1 } = \mathcal{P} _{ r + 1 } ^-\Lambda ^{ k -1 } ( \mathcal{T} _h )  .
    \end{cases}
  \end{equation*}
\end{theorem}

\begin{proof}
  The first estimate is immediate from \citep[Theorem
  3.11]{ArFaWi2010}. Next, subtracting \eqref{eqn:pp_k<n_v} from
  \eqref{eqn:pp_tilde_v} with
  $ v _h \in \mathfrak{Z} _h ^{ \ast k \perp } (K) $ gives
  \begin{align*}
    \bigl( \delta ( \widetilde{ u } _h - u _h ^\ast ), \delta v _h \bigr) _K
    &= ( p _h - p ,  v _h ) _K + \langle \widehat{ \sigma } _h ^{\mathrm{tan}} - \widehat{ \sigma } ^{\mathrm{tan}} , v _h ^{\mathrm{nor}} \rangle _{ \partial K } \\
    &= ( p _h - p ,  v _h ) _K + \bigl( \mathrm{d} (\sigma _h - \sigma ) , v _h \bigr) _K - ( \sigma _h - \sigma , \delta v _h ) _K  \\
    &\lesssim \biggl[ \lVert \sigma - \sigma _h \rVert _K + h _K \Bigl( \bigl\lVert  \mathrm{d} ( \sigma - \sigma _h ) \bigr\rVert _K +  \lVert p - p _h \rVert _K \Bigr) \biggr] \lVert \delta v _h \rVert _K .
  \end{align*}
  In the last step, we have applied Cauchy--Schwarz and the Poincar\'e
  inequality with scaling, which says that
  $ \lVert v _h \rVert _K \lesssim h _K \lVert \delta v _h \rVert _K
  $. Taking $ v _h $ such that
  $ \delta v _h = \delta ( \widetilde{ u } _h - u _h ^\ast ) $ implies
  the second estimate. Finally, summing over $ K \in \mathcal{T} _h $
  and applying \citep[Theorem 3.11]{ArFaWi2010} gives
  \begin{equation*}
    \bigl\lVert \delta ( \widetilde{ u } _h  - u _h ^\ast ) \bigr\rVert _{\mathcal{T} _h} \lesssim
    \begin{cases}
      0, & \text{if } f \in \mathring{ \mathfrak{B}  } _k ^\ast ,\\
      h ^{ s + 1 } \lVert f \rVert _{ s, \Omega } , & \text{otherwise, if }
      \begin{cases}
        s \leq r + 1, & V _h ^{ k -1 } = \mathcal{P} _{ r + 1 } \Lambda ^{ k -1 } ( \mathcal{T} _h ), \\
        s \leq r, & V _h ^{ k -1 } = \mathcal{P} _{ r + 1 } ^-\Lambda
        ^{ k -1 } ( \mathcal{T} _h ) ,
      \end{cases}
    \end{cases}
  \end{equation*}
  so the last estimate follows by \cref{assumption} and the triangle
  inequality.
\end{proof}

Thus far, we have been able to avoid dealing with the error term
$ \widehat{ u } ^{\mathrm{tan}} - \widehat{ u } ^{\mathrm{tan}} _h $,
which dominates the postprocessing error, preventing improved
convergence of the $ \mathfrak{B} _h ^\ast (K) $ components. There is
one special exception, however: when $ k = n -1 $, the space
$ \mathfrak{B} _h ^{ \ast n } (K) $ is trivial, so there is no error
in this component of $ \rho _h ^\ast $. In this case, we will see that
$ \rho _h ^\ast $ is an improved estimate compared to
$ \mathrm{d} u _h $. Since
$ \mathfrak{H} _h ^{ \ast n } (K) \cong \mathbb{R} $ is nontrivial,
though, we need to control the $ \overline{ \mathfrak{H} } ^n $
component of the error, which we will do with the aid of the following
lemma.

\begin{lemma}
  \label{lem:n-1_trace}
  If $ k = n -1 $ and $ \eta _h \in \overline{ \mathfrak{H} } ^n $,
  then
  \begin{equation*}
    \langle \widehat{ u } ^{\mathrm{tan}} - \widehat{ u } _h
    ^{\mathrm{tan}} , \eta _h ^{\mathrm{nor}} \rangle _{ \partial \mathcal{T} _h }
    \lesssim h \Bigl( \bigl\lVert \mathrm{d} ( \sigma - \sigma _h )
    \bigr\rVert _\Omega + \lVert p - p _h \rVert _\Omega \Bigr) \lVert
    \eta _h \rVert _\Omega .
  \end{equation*}
  In particular, if
  $ f \in \mathring{ \mathfrak{B} } _{ n -1 } ^\ast $, then
  $ \int _{ \partial K } \operatorname{tr} ( u - u _h ) = 0 $ for all
  $ K \in \mathcal{T} _h $.
\end{lemma}

\begin{proof}
  Since $ \eta _h $ is piecewise constant,
  $ \langle \widehat{ u } ^{\mathrm{tan}} - \widehat{ u } _h
  ^{\mathrm{tan}} , \eta _h ^{\mathrm{nor}} \rangle _{ \partial
    \mathcal{T} _h } = \bigl( \mathrm{d} ( u - u _h ) , \eta _h \bigr)
  _{ \mathcal{T} _h } $. Piecewise constants are in
  $ V _h ^n = \mathfrak{B} _h ^n $, so the estimate follows by
  \cref{lem:Bh}. In particular, $\sigma$ and $p$ vanish when
  $ f \in \mathring{ \mathfrak{B} } _{ n -1 } ^\ast $, so in that case
  the left-hand side is identically zero.
\end{proof}

\begin{remark}
  This generalizes the well-known property that, when $ n = 1 $ and
  $ k = 0 $, the continuous Galerkin solution equals the exact
  solution at nodes.
\end{remark}

We now show that $ \rho _h ^\ast $ approximates
$ \rho = \mathrm{d} u $ as well as $ \mathrm{d} u _h $ does, but no
better when $ k < n -1 $. However, when $ k = n -1 $, we get an
improved estimate, and when
$ f \in \mathring{ \mathfrak{B} } ^\ast _{n-1} $, we can obtain an
arbitrarily good approximation by taking $ r ^\ast $ large enough.

\begin{theorem}
  For each $ K \in \mathcal{T} _h $ and $ 0 \leq s \leq t _{\max} $,
  \begin{align*}
    \lVert \rho - \widetilde{ \rho } _h \rVert _K &\lesssim h _K ^{ s + 1 } \lVert f \rVert _{ s, K } , \quad \text{if }
    \begin{cases}
      s \leq r ^\ast + 1 , & \star W _h ^{ \ast k + 1 } (K) = \mathcal{P} _{ r ^\ast + 1 } \Lambda ^{ n - k -1 } (K),\\
      s \leq r ^\ast , & \star W _h ^{ \ast k + 1 } (K) = \mathcal{P} _{ r ^\ast + 1 } ^- \Lambda ^{ n - k -1 } (K) ,
    \end{cases} \\
    \lVert \widetilde{ \rho } _h  - \rho _h ^\ast \rVert _K &\lesssim \bigl\lVert \mathrm{d} ( u - u _h ) \bigr\rVert _K + h _K \Bigl( \bigl\lVert \mathrm{d} ( \sigma - \sigma _h ) \bigr\rVert _K + \lVert p - p _h \rVert _K \Bigr) .
  \end{align*}
  Consequently, if \cref{assumption} holds, then
  \begin{equation*}
    \lVert \rho - \rho _h ^\ast \rVert _{ \mathcal{T} _h } \lesssim h ^{ s + 1 } \lVert f \rVert _{ s, \Omega } ,\quad 
    \text{if }
    \begin{cases}
      s \leq r + 1 , & f \perp \mathring{ \mathfrak{B}  } _k ^\ast ,\\
      s \leq r , &  V _h ^k = \mathcal{P} _{ r + 1 } ^- \Lambda ^k ( \mathcal{T} _h )\\
      s \leq r -1 , & V _h ^k = \mathcal{P} _r \Lambda ^k ( \mathcal{T} _h ) .
    \end{cases}
  \end{equation*}
  In the case $ k = n -1 $, this estimate may be improved to
  \begin{equation*}
    \lVert \rho - \rho _h ^\ast \rVert _{ \mathcal{T} _h } \lesssim h ^{ s + 1 } \lVert f \rVert _{ s, \Omega } ,\quad 
    \text{if }
    \begin{cases}
      s \leq r ^\ast + 1 , & f \in \mathring{ \mathfrak{B}  } _{n-1} ^\ast ,\ \star W _h ^{ \ast k + 1 } (K) = \mathcal{P} _{ r ^\ast + 1 } \Lambda ^{ n - k -1 } (K),\\
      s \leq r ^\ast , & f \in \mathring{ \mathfrak{B}  } _{n-1} ^\ast ,\ \star W _h ^{ \ast k + 1 } (K) = \mathcal{P} _{ r ^\ast + 1 } ^- \Lambda ^{ n - k -1 } (K) ,
\\
      s \leq r + 1 , &  \text{otherwise}.
    \end{cases}
  \end{equation*}
\end{theorem}

\begin{proof}
  The first estimate is immediate from \citep[Theorem
  3.11]{ArFaWi2010}.  Next, subtracting \eqref{eqn:pp_k<n_eta} from
  \eqref{eqn:pp_tilde_eta} with
  $ \eta _h \in \mathfrak{Z} _h ^{ \ast k + 1 } (K) $ gives
  \begin{equation*}
    ( \widetilde{ \rho } _h - \rho _h ^\ast , \eta _h ) _K = \langle \widehat{ u }  ^{\mathrm{tan}} - \widehat{ u } _h ^{\mathrm{tan}} , \eta _h ^{\mathrm{nor}} \rangle _{ \partial K } = \bigl( \mathrm{d} ( u - u _h ) , \eta _h \bigr) _K \leq \bigl\lVert \mathrm{d} ( u - u _h ) \bigr\rVert _K \lVert \eta _h \rVert _K ,
  \end{equation*}
  which implies
  \begin{equation*}
    \bigl\lVert P _{ \mathfrak{Z}  _h ^\ast (K) } ( \widetilde{ \rho } _h - \rho _h ^\ast ) \bigr\rVert _K \leq \bigl\lVert \mathrm{d} ( u - u _h ) \bigr\rVert _K .
  \end{equation*}
  Furthermore, by the Poincar\'e inequality and
  \cref{thm:pp_deltarho},
  \begin{equation*}
    \bigl\lVert P _{ \mathfrak{Z}  _h ^{\ast \perp} (K) } ( \widetilde{ \rho } _h - \rho _h ^\ast ) \bigr\rVert _K \lesssim h _K \bigl\lVert \delta ( \widetilde{ \rho } _h - \rho _h ^\ast ) \bigr\rVert _K \leq h _K \Bigl( \bigl\lVert \mathrm{d} ( \sigma - \sigma _h ) \bigr\rVert _K + \lVert p - p _h \rVert _K \Bigr),
  \end{equation*}
  so the second estimate follows by the Hodge decomposition and
  triangle inequality. Summing over $ K \in \mathcal{T} _h $ and
  applying \citep[Theorem 3.11]{ArFaWi2010} gives
  \begin{equation*}
    \lVert \widetilde{ \rho } _h - \rho _h ^\ast \rVert _{ \mathcal{T} _h } \lesssim h ^{ s + 1 } \lVert f \rVert _{ s, \Omega } ,\quad \text{if }
    \begin{cases}
      s \leq r + 1 , & f \perp \mathring{ \mathfrak{B}  } _k ^\ast ,\\
      s \leq r , &  V _h ^k = \mathcal{P} _{ r + 1 } ^- \Lambda ^k ( \mathcal{T} _h ), \\
      s \leq r -1 , & V _h ^k = \mathcal{P} _r \Lambda ^k ( \mathcal{T} _h ) ,
    \end{cases}
  \end{equation*}
  so the third estimate follows by \cref{assumption} and the triangle
  inequality.

  Finally, consider the special case $ k = n -1 $. Taking
  $ \eta _h \in \overline{ \mathfrak{H} } ^n $ and applying
  \cref{lem:n-1_trace} gives
  \begin{equation*}
    ( \widetilde{ \rho } _h - \rho _h ^\ast , \eta _h ) _{\mathcal{T} _h } = \langle \widehat{ u }  ^{\mathrm{tan}} - \widehat{ u } _h ^{\mathrm{tan}} , \eta _h ^{\mathrm{nor}} \rangle _{ \partial \mathcal{T} _h } \lesssim h \Bigl( \bigl\lVert \mathrm{d} ( \sigma - \sigma _h )
    \bigr\rVert _\Omega + \lVert p - p _h \rVert _\Omega \Bigr) \lVert
    \eta _h \rVert _\Omega ,
  \end{equation*}
  and therefore,
  \begin{equation*}
    \bigl\lVert P _{ \overline{ \mathfrak{H}  } } ( \widetilde{ \rho } _h - \rho _h ^\ast ) \bigr\rVert _{ \mathcal{T} _h } \lesssim h \Bigl( \bigl\lVert \mathrm{d} ( \sigma - \sigma _h ) \bigr\rVert _\Omega + \lVert p - p _h \rVert _\Omega \Bigr) .
  \end{equation*}
  Note that this eliminates the
  $ \bigl\lVert \mathrm{d} ( u - u _h ) \bigr\rVert _\Omega $ term
  that appears in the $ k < n -1 $ case. Hence,
  \begin{equation*}
    \lVert \widetilde{ \rho } _h - \rho _h ^\ast \rVert _{ \mathcal{T} _h } \lesssim h \Bigl( \bigl\lVert \mathrm{d} ( \sigma - \sigma _h ) \bigr\rVert _\Omega + \lVert p - p _h \rVert _\Omega \Bigr) \lesssim
    \begin{cases}
      0, & \text{if } f \in \mathring{ \mathfrak{B}  } _{n-1} ^\ast ,\\
      h ^{ s + 1 } \lVert f \rVert _{ s, \Omega } , & \text{if } s \leq r + 1 , \text{ otherwise},
    \end{cases}
  \end{equation*} 
  and the improved estimate follows.
\end{proof}

Finally, we show that $ u _h ^\ast $ approximates $u$ as well as
$ u _h $ does, but no better.

\begin{theorem}
  For each $ K \in \mathcal{T} _h $ and $ 0 \leq s \leq t _{ \max } $,
  \begin{align*} 
    \lVert u - \widetilde{ u } _h \rVert _K &\lesssim
    \begin{cases}
      h _K \lVert f \rVert _K , & \text{if } \star W _h ^{\ast k} = \mathcal{P} _1 ^- \Lambda ^{ n - k } (K) ,\\
      h _K ^{ s + 2 } \lVert f \rVert _{ s, K }, &\text{if } s \leq r ^\ast - 1 , \text{ otherwise},
    \end{cases}\\
    \lVert \widetilde{ u } _h - u _h ^\ast \rVert _K &\lesssim \lVert u - u _h \rVert _K + h _K \Bigl(  \bigl\lVert \mathrm{d} ( u - u _h ) \bigr\rVert _K + \lVert \sigma - \sigma _h \rVert _K \Bigr) + h _K ^2 \Bigl( \bigl\lVert \mathrm{d} ( \sigma - \sigma _h ) \bigr\rVert _K + \lVert p - p _h \rVert _K \Bigr) .
  \end{align*}
  Consequently, if \cref{assumption} holds, then
  \begin{equation*}
    \lVert u - u _h ^\ast \rVert _{ \mathcal{T} _h } \lesssim \begin{cases}
      h \lVert f \rVert _\Omega , & \text{if } V _h ^k = \mathcal{P} _1 ^- \Lambda ^k (\mathcal{T} _h ) ,\\
      h ^{ s + 2 } \lVert f \rVert _{ s , \Omega }, & \text{if } s \leq r - 1 , \text{ otherwise}.
    \end{cases}
  \end{equation*} 
\end{theorem}

\begin{proof}
  The first estimate is immediate from \citep[Theorem
  3.11]{ArFaWi2010}. Next, subtracting \eqref{eqn:pp_k<n_eta} from
  \eqref{eqn:pp_tilde_eta} with
  $ \eta _h \in \mathfrak{Z} _h ^{ * k + 1 \perp } (K) $ gives
  \begin{align*}
    ( \widetilde{ u } _h - u _h ^\ast , \delta \eta _h ) _K
    &= ( \widetilde{ \rho } _h - \rho _h ^\ast , \eta _h ) _K - \langle \widehat{ u } ^{\mathrm{tan}} - \widehat{ u } _h ^{\mathrm{tan}} , \eta _h ^{\mathrm{nor}} \rangle _{ \partial K } \\
    &= \bigl( P _{ \mathfrak{Z}  _h ^{ * \perp } (K) } ( \widetilde{ \rho } _h - \rho _h ^\ast ) , \eta _h \bigr) _K - \bigl( \mathrm{d} ( u - u _h ) , \eta _h \bigr) _K + ( u - u _h , \delta \eta _h ) _K \\
    &\lesssim \Bigl( \lVert u - u _h \rVert _K + h _K \bigl\lVert \mathrm{d} ( u - u _h ) \bigr\rVert _K + h _K ^2 \bigl\lVert \delta ( \widetilde{ \rho } _h - \rho _h ^\ast ) \bigr\rVert _K \Bigr) \lVert \delta \eta _h \rVert _K ,
  \end{align*} 
  by Cauchy--Schwarz and the Poincar\'e inequality. With
  \cref{thm:pp_deltarho}, this implies
  \begin{equation*}
    \bigl\lVert P _{ \mathfrak{B}  _h ^\ast (K) } ( \widetilde{ u } _h - u _h ^\ast ) \bigr\rVert _K \lesssim \lVert u - u _h \rVert _K + h _K \bigl\lVert \mathrm{d} ( u - u _h ) \bigr\rVert _K + h _K ^2 \Bigl( \bigl\lVert \mathrm{d} ( \sigma - \sigma _h ) \bigr\rVert _K + \lVert p - p _h \rVert _K  \Bigr).
  \end{equation*}
  Furthermore, by the Poincar\'e inequality and \cref{thm:pp_deltau},
  \begin{equation*}
    \bigl\lVert P _{ \mathfrak{Z}  _h ^{ \ast \perp } (K) } ( \widetilde{ u } _h - u _h ^\ast ) \bigr\rVert _K \lesssim h _K \lVert \sigma - \sigma _h \rVert _K + h _K ^2 \Bigl( \bigl\lVert  \mathrm{d} ( \sigma - \sigma _h ) \bigr\rVert _K +  \lVert p - p _h \rVert _K \Bigr),
  \end{equation*}
  so the second estimate follows by the Hodge decomposition and
  triangle inequality.  Finally, summing over $ K \in \mathcal{T} _h $
  and applying \citep[Theorem 3.11]{ArFaWi2010} gives
  \begin{equation*}
    \lVert \widetilde{ u } _h - u _h ^\ast \rVert _{ \mathcal{T} _h } 
    \lesssim \begin{cases}
      h \lVert f \rVert _\Omega , & \text{if } V _h ^k = \mathcal{P} _1 ^- \Lambda ^k (\mathcal{T} _h ) ,\\
      h ^{ s + 2 } \lVert f \rVert _{ s , \Omega }, & \text{if } s \leq r - 1 , \text{ otherwise},
    \end{cases}
  \end{equation*}
  so the last estimate follows by \cref{assumption} and the triangle
  inequality.
\end{proof}

\subsection{Remarks on the case $ k = n $}

Although the case $ k = n $ has already been analyzed by
\citet{Stenberg1991}, we now briefly describe this analysis from the
FEEC viewpoint, relating it to the techniques developed in this
section. In this case, the postprocessing procedure \eqref{eqn:pp}
becomes
\begin{alignat*}{2}
  ( \delta u _h ^\ast  , \delta v _h  ) _K + ( \overline{p} _h ^\ast , v _h ) _K &= ( f ,  v _h ) _K - \langle \widehat{ \sigma } _h ^{\mathrm{tan}} , v _h ^{\mathrm{nor}} \rangle _{ \partial K } , \quad &\forall v _h &\in W _h ^{ \ast n } (K) ,\\
  ( u _h ^\ast , \overline{q} _h ) _K &= ( u _h , \overline{q}_h  ) _K ,
  \quad &\forall \overline{q} _h &\in \mathfrak{H} _h ^{ \ast n } (K) ,
\end{alignat*}
and the intermediate approximation is given by
\begin{alignat*}{2}
  ( \delta \widetilde{ u } _h , \delta v _h  ) _K + ( \widetilde{ p } _h , v _h ) _K &= ( f ,  v _h ) _K - \langle \widehat{ \sigma } ^{\mathrm{tan}} , v _h ^{\mathrm{nor}} \rangle _{ \partial K } , \quad &\forall v _h &\in W _h ^{ \ast n } (K) ,\\
  ( \widetilde{ u } _h , \overline{q} _h ) _K &= ( u ,
  \overline{q}_h ) _K , \quad &\forall \overline{q} _h &\in \mathfrak{H} _h
  ^{ \ast n } (K) .
\end{alignat*}
The argument in \cref{thm:pp_deltau} still works, so applying the
Poincar\'e inequality gives
\begin{equation*}
  \bigl\lVert P _{ \mathfrak{Z}  _h ^{ \ast \perp } (K) } ( \widetilde{ u } _h - u _h ^\ast ) \bigr\rVert _K \lesssim h _K \lVert \sigma - \sigma _h \rVert _K + h _K ^2 \bigl\lVert  \mathrm{d} ( \sigma - \sigma _h ) \bigr\rVert _K .
\end{equation*}
Furthermore, since $ \overline{ \mathfrak{H} } ^n $ consists of
piecewise constants, which are in $ V _h ^n = \mathfrak{B} _h ^n $, we
have
\begin{equation*}
  \bigl\lVert P _{\overline{ \mathfrak{H}  }} (\widetilde{ u } _h - u _h ^\ast) \rVert _K = \bigl\lVert P _{\overline{ \mathfrak{H}  }} ( u - u _h ) \bigr\rVert _K \leq \bigl\lVert P _{ \mathfrak{B}  _h } ( u - u _h ) \bigr\rVert _K .
\end{equation*}
Summing over $ K \in \mathcal{T} _h $ and applying \citep[Lemma
3.13]{ArFaWi2010} implies
\begin{equation*}
  \lVert \widetilde{ u } _h - u _h ^\ast \rVert _{ \mathcal{T} _h } \lesssim \begin{cases}
    h ^{s+1} \lVert f \rVert _{ s, \Omega } , & \text{if }  s \leq 1 ,\ V _h ^n = \mathcal{P} _1 ^- \Lambda ^n ( \mathcal{T} _h )  ,\\
    h ^{s + 2} \lVert f \rVert _{ s , \Omega } ,& \text{otherwise, if }
      \begin{cases}
        s \leq r + 1, & V _h ^{ n -1 } = \mathcal{P} _{ r + 1 } \Lambda ^{ n -1 } ( \mathcal{T} _h ), \\
        s \leq r, & V _h ^{ n -1 } = \mathcal{P} _{ r + 1 } ^-\Lambda
        ^{ n -1 } ( \mathcal{T} _h ) ,
      \end{cases}
  \end{cases}
\end{equation*}
so by \cref{assumption} and the triangle inequality, this same
estimate holds for
$ \lVert u - u _h ^\ast \rVert _{ \mathcal{T} _h } $. This is
precisely the improved estimate in \citet[Theorem 2.2]{Stenberg1991},
by essentially the same proof.

\section{Illustration of the methods in $ n = 3 $ dimensions}
\label{sec:examples}

We now give a concrete illustration of the hybridization and
postprocessing schemes in $ n = 3 $ dimensions, using scalar and
vector proxy fields and the familiar operations of vector
calculus. Let $ \mathcal{T} _h $ be a simplicial triangulation of a
bounded, polyhedral domain $ \Omega \subset \mathbb{R}^3 $. For
simplicity, we also assume that $\Omega$ is contractible, so that
$ \mathfrak{H} ^0 \cong \mathbb{R} $ and $ \mathfrak{H} ^k $ is
trivial for $ k = 1, 2, 3 $.

Let $ V _h $ be a stable subcomplex of
\begin{equation*}
  \begin{tikzcd}
    0 \ar[r] & H ^1 (\Omega) \ar[r, "\operatorname{grad}"] & H
    ( \operatorname{curl}; \Omega ) \ar[r, "\operatorname{curl}"] & H
    ( \operatorname{div}; \Omega ) \ar[r, "\operatorname{div}"] & L ^2
    (\Omega) \ar[r] & 0 ,
  \end{tikzcd} 
\end{equation*}
containing continuous Lagrange elements, N\'ed\'elec edge and face
elements, and discontinuous Lagrange elements. Let $ W _h $ be the
corresponding ``broken'' complex, with
$ W _h ^k (K) = V _h ^k \rvert _K $ for $ K \in \mathcal{T} _h
$. Using the scalar and vector proxies for tangential traces in
\cref{tab:traces}, we have
\begin{align*}
   \widehat{ V } _h ^{ 0 , \mathrm{tan} } &= \bigl\{ v _h \rvert _{ \partial \mathcal{T} _h } : v _h \in V _h ^0 \bigr\}, &  \widehat{ W } _h ^{0, \mathrm{nor} } = \widehat{ W } _h ^{ 0 , \mathrm{tan} } &= \bigl\{ v _h \rvert _{ \partial \mathcal{T} _h } : v _h \in W _h ^0 \bigr\} ,\\
\widehat{ V } _h ^{1, \mathrm{tan}} &= \bigl\{ v _h \rvert _{ \partial \mathcal{T} _h } - ( v _h \cdot \normal ) \normal  : v _h \in V _h ^1 \bigr\}, &  \widehat{ W } _h ^{1, \mathrm{nor}} = \widehat{ W } _h ^{1, \mathrm{tan}} &= \bigl\{ v _h \rvert _{ \partial \mathcal{T} _h } - ( v _h \cdot \normal ) \normal  : v _h \in W _h ^1 \bigr\}, \\
\widehat{ V } _h ^{2, \mathrm{tan}} &= \bigl\{ ( v _h \cdot \normal ) \normal  : v _h \in V _h ^2 \bigr\},  &  \widehat{ W } _h ^{2, \mathrm{nor}} = \widehat{ W } _h ^{2, \mathrm{tan}} &= \bigl\{ ( v _h \cdot \normal ) \normal  : v _h \in W _h ^2 \bigr\} ,
\end{align*}
whose degrees of freedom are just those of $ V _h ^k $ and $ W _h ^k $
living on $ \partial \mathcal{T} _h $.

For postprocessing on $ K \in \mathcal{T} _h $, let $ W _h ^\ast (K) $
be a stable subcomplex of
\begin{equation*}
  \begin{tikzcd}
    0 & \ar[l] L ^2  (\Omega) 
    & \ar[l, "-\operatorname{div}"'] H ( \operatorname{div}; \Omega  ) 
    & \ar[l, "\operatorname{curl}"'] H ( \operatorname{curl} ;  \Omega ) 
    & \ar[l, "-\operatorname{grad}"'] H ^1  (\Omega) 
    & \ar[l] 0 ,
  \end{tikzcd} 
\end{equation*}
whose normal traces have scalar and vector proxies given in
\cref{tab:traces}.

\subsection{The case $ k = 0 $}

The hybrid method is
\begin{alignat*}{2}
    ( \operatorname{grad} u _h , \operatorname{grad} v _h ) _{ \mathcal{T} _h } + ( p _h , v _h ) _{ \mathcal{T} _h } - \langle \widehat{ \rho } _h ^{\mathrm{nor}} , v _h \rangle _{ \partial \mathcal{T} _h } &= ( f, v _h ) _{ \mathcal{T} _h } , \quad&\forall v _h &\in W _h ^0 ,\\
  \langle \widehat{ u } _h ^{\mathrm{tan}} - u _h , \widehat{ \eta } _h ^{\mathrm{nor}} \rangle _{ \partial \mathcal{T} _h } &= 0, \quad &\forall \widehat{ \eta } _h ^{\mathrm{nor}} &\in \widehat{ W } _h ^{0, \mathrm{nor}}, \\
  (u _h , q _h ) _{ \mathcal{T} _h } &= 0, \quad &\forall q _h &\in \mathbb{R} ,\\
  \langle \widehat{ \rho } _h ^{\mathrm{nor}} , \widehat{ v } _h ^{\mathrm{tan}} \rangle _{ \partial \mathcal{T} _h } &= 0 , \quad&\forall \widehat{ v } _h ^{\mathrm{tan}} &\in \widehat{ V } _h ^{0, \mathrm{tan}} ,
\end{alignat*}
which is the hybridized continuous Galerkin method of
\citet*{CoGoWa2007} for the Neumann problem. The postprocessing scheme
on $ K \in \mathcal{T} _h $ is
\begin{alignat*}{2}
  ( \rho _h ^\ast , \eta _h  ) _K + ( u _h ^\ast , \operatorname{div} \eta _h ) _K &= \langle \widehat{ u } _h ^{\mathrm{tan}} , \eta _h \cdot \normal \rangle _{ \partial K } , \quad &\forall \eta _h &\in W _h ^{ \ast 1 } (K) , \\
  -( \operatorname{div} \rho _h ^\ast , v _h  ) _K &= ( f - p _h ,  v _h ) _K , \quad &\forall v _h &\in W _h ^{ \ast 0 } (K) .
\end{alignat*}

\subsection{The case $ k = 1 $}

The hybrid method is
\begin{alignat*}{2}
  ( \sigma _h , \tau _h ) _{ \mathcal{T} _h } - ( u _h , \operatorname{grad} \tau _h ) _{ \mathcal{T} _h } + \langle \widehat{ u } _h ^{\mathrm{nor}} , \tau _h \rangle _{ \partial \mathcal{T} _h } &= 0, \quad &\forall \tau _h &\in W _h ^0 ,\\
  (\operatorname{grad} \sigma _h , v _h ) _{ \mathcal{T} _h } + ( \operatorname{curl} u _h , \operatorname{curl} v _h ) _{ \mathcal{T} _h } - \langle \widehat{ \rho } _h ^{\mathrm{nor}} , v _h \rangle _{ \partial \mathcal{T} _h } &= ( f, v _h ) _{ \mathcal{T} _h } , \quad&\forall v _h &\in W _h ^1 , \\
  \langle \widehat{ \sigma } _h ^{\mathrm{tan}} - \sigma _h, \widehat{ v } _h ^{\mathrm{nor}} \rangle _{ \partial \mathcal{T} _h } &= 0, \quad &\forall \widehat{ v } _h ^{\mathrm{nor}} &\in \widehat{ W } _h ^{0, \mathrm{nor}} , \\
  \langle \widehat{ u } _h ^{\mathrm{tan}} - u _h, \widehat{ \eta } _h ^{\mathrm{nor}} \rangle _{ \partial \mathcal{T} _h } &= 0, \quad &\forall \widehat{ \eta } _h ^{\mathrm{nor}} &\in \widehat{ W } _h ^{1, \mathrm{nor}} , \\
  \langle \widehat{ u } _h ^{\mathrm{nor}} , \widehat{ \tau } _h ^{\mathrm{tan}} \rangle _{ \partial \mathcal{T} _h } &= 0 , \quad&\forall \widehat{ \tau } _h ^{\mathrm{tan}} &\in \widehat{ V } _h ^{0, \mathrm{tan}} , \\
  \langle \widehat{ \rho } _h ^{\mathrm{nor}}, \widehat{ v } _h ^{\mathrm{tan}} \rangle _{ \partial \mathcal{T} _h } &= 0 , \quad&\forall \widehat{ v } _h ^{\mathrm{tan}} &\in \widehat{ V } _h ^{1, \mathrm{tan}} ,
\end{alignat*}
and the postprocessing scheme on $ K \in \mathcal{T} _h $ is
\begin{alignat*}{2}
  ( \rho _h ^\ast , \eta _h  ) _K - ( u _h ^\ast , \operatorname{curl} \eta _h ) _K &= \langle \widehat{ u } _h ^{\mathrm{tan}} , \eta _h \times \normal \rangle _{ \partial K } , &\forall \eta _h &\in W _h ^{ \ast 2 } (K) , \\
  ( \operatorname{curl} \rho _h ^\ast , v _h  ) _K + ( \operatorname{div} u _h ^\ast  , \operatorname{div} v _h  ) _K &= ( f ,  v _h ) _K - \langle \widehat{ \sigma } _h ^{\mathrm{tan}} , v _h \cdot \normal \rangle _{ \partial K } , \quad &\forall v _h &\in W _h ^{ \ast 1 } (K) .
\end{alignat*}

\subsection{The case $ k = 2 $}

The hybrid method is
\begin{alignat*}{2}
    ( \sigma _h , \tau _h ) _{ \mathcal{T} _h } - ( u _h , \operatorname{curl} \tau _h ) _{ \mathcal{T} _h } + \langle \widehat{ u } _h ^{\mathrm{nor}} , \tau _h \rangle _{ \partial \mathcal{T} _h } &= 0, \quad &\forall \tau _h &\in W _h ^1 ,\\
    (\operatorname{curl} \sigma _h , v _h ) _{ \mathcal{T} _h } + ( \operatorname{div} u _h , \operatorname{div} v _h ) _{ \mathcal{T} _h } - \langle \widehat{ \rho } _h ^{\mathrm{nor}} , v _h \rangle _{ \partial \mathcal{T} _h } &= ( f, v _h ) _{ \mathcal{T} _h } , \quad&\forall v _h &\in W _h ^2 , \\
  \langle \widehat{ \sigma } _h ^{\mathrm{tan}} - \sigma _h, \widehat{ v } _h ^{\mathrm{nor}} \rangle _{ \partial \mathcal{T} _h } &= 0, \quad &\forall \widehat{ v } _h ^{\mathrm{nor}} &\in \widehat{ W } _h ^{1, \mathrm{nor}}, \\
  \langle \widehat{ u } _h ^{\mathrm{tan}} - u _h , \widehat{ \eta } _h ^{\mathrm{nor}} \rangle _{ \partial \mathcal{T} _h } &= 0, \quad &\forall \widehat{ \eta } _h ^{\mathrm{nor}} &\in \widehat{ W } _h ^{2, \mathrm{nor}} , \\
  \langle \widehat{ u } _h ^{\mathrm{nor}} , \widehat{ \tau } _h ^{\mathrm{tan}} \rangle _{ \partial \mathcal{T} _h } &= 0 , \quad&\forall \widehat{ \tau } _h ^{\mathrm{tan}} &\in \widehat{ V } _h ^{1, \mathrm{tan}} , \\
  \langle \widehat{ \rho } _h ^{\mathrm{nor}} , \widehat{ v } _h ^{\mathrm{tan}} \rangle _{ \partial \mathcal{T} _h } &= 0 , \quad&\forall \widehat{ v } _h ^{\mathrm{tan}} &\in \widehat{ V } _h ^{2, \mathrm{tan}} ,
\end{alignat*}
and the postprocessing scheme on $ K \in \mathcal{T} _h $ is
\begin{alignat*}{2}
  ( \rho _h ^\ast , \eta _h  ) _K + ( u _h ^\ast , \operatorname{grad} \eta _h ) _K &= \langle \widehat{ u } _h ^{\mathrm{tan}} , \eta _h \normal \rangle _{ \partial K } , &\forall \eta _h &\in W _h ^{ \ast 3 } (K) , \\
  -( \operatorname{grad} \rho _h ^\ast , v _h  ) _K + ( \operatorname{curl} u _h ^\ast  , \operatorname{curl} v _h  ) _K &= ( f ,  v _h ) _K - \langle \widehat{ \sigma } _h ^{\mathrm{tan}} , v _h \times \normal \rangle _{ \partial K } , \quad &\forall v _h &\in W _h ^{ \ast 2 } (K) .
\end{alignat*}

\subsection{The case $ k = 3 $}

The hybrid method is
\begin{alignat*}{2}
    ( \sigma _h , \tau _h ) _{ \mathcal{T} _h } - ( u _h , \operatorname{div} \tau _h ) _{ \mathcal{T} _h } + \langle \widehat{ u } _h ^{\mathrm{nor}} , \tau _h \rangle _{ \partial \mathcal{T} _h } &= 0, \quad &\forall \tau _h &\in W _h ^2 ,\\
    (\operatorname{div} \sigma _h , v _h ) _{ \mathcal{T} _h } + ( \overline{p} _h , v _h ) _{ \mathcal{T} _h } &= ( f, v _h ) _{ \mathcal{T} _h } , \quad&\forall v _h &\in W _h ^3 , \\
    ( \overline{ u } _h - u _h , \overline{q} _h ) _{ \mathcal{T} _h } &= 0 , \quad &\forall \overline{q} _h &\in \mathbb{R}  ^{ \mathcal{T} _h } ,\\
  \langle \widehat{ \sigma } _h ^{\mathrm{tan}} - \sigma _h, \widehat{ v } _h ^{\mathrm{nor}} \rangle _{ \partial \mathcal{T} _h } &= 0, \quad &\forall \widehat{ v } _h ^{\mathrm{nor}} &\in \widehat{ W } _h ^{2, \mathrm{nor}}, \\
  ( \overline{p} _h , \overline{v} _h ) _{ \mathcal{T} _h } &= 0 , \quad &\forall \overline{v} _h &\in \mathbb{R}  ^{ \mathcal{T} _h } ,\\
  \langle \widehat{ u } _h ^{\mathrm{nor}} , \widehat{ \tau } _h ^{\mathrm{tan}} \rangle _{ \partial \mathcal{T} _h } &= 0 , \quad&\forall \widehat{ \tau } _h ^{\mathrm{tan}} &\in \widehat{ V } _h ^{2, \mathrm{tan}} ,
\end{alignat*}
which is the alternative hybridization of the RT and BDM methods in
\citet[Section 5]{Cockburn2016} using local Neumann solvers; its
solution coincides with the classic hybridized RT and BDM methods of
\citet*{ArBr1985,BrDoMa1985} using local Dirichlet solvers. The
postprocessing scheme on $ K \in \mathcal{T} _h $ is exactly that of
\citet{Stenberg1991},
\begin{alignat*}{2}
  ( \operatorname{grad} u _h ^\ast  , \operatorname{grad}  v _h  ) _K + ( \overline{p} _h ^\ast , v _h ) _K &= ( f ,  v _h ) _K - \langle \widehat{ \sigma } _h ^{\mathrm{tan}} , v _h \normal \rangle _{ \partial K } , \quad &\forall v _h &\in W _h ^{ \ast 3 } (K) ,\\
  ( u _h ^\ast , \overline{q} _h ) _K &= ( \overline{ u } _h ,
  \overline{q}_h ) _K , \quad &\forall \overline{q} _h &\in \mathbb{R} .
\end{alignat*}

\section{Numerical experiments}
\label{sec:numerical}

In this section, we present numerical experiments in $ n = 3 $
dimensions that illustrate and confirm the foregoing theory.  We omit
the cases $ k = 0 $ and $ k = n $, since these correspond to known
methods for the scalar Poisson equation whose properties are already
well understood. The remaining cases correspond to hybridization and
postprocessing methods for the vector Poisson equation.

For the sake of brevity, we present only numerical experiments using
$ \mathcal{P} _{ r + 1 } ^- \Lambda $ elements with
$ \star \mathcal{P} _{ r ^\ast + 1 } ^- \Lambda $ postprocessing,
where $ r ^\ast $ is chosen optimally according to \cref{assumption},
and where $f$ has nonvanishing components in both $ \mathfrak{B} ^k $
and $ \mathring{ \mathfrak{B} } ^\ast _k $. Errors and rates are shown
only for the normal traces and postprocessed solution components,
since the convergence behavior of the remaining variables follows from
previous work.  We have conducted many additional numerical
experiments, which all conform with the theoretical results.

All computations have been carried out using the Firedrake finite
element library \citep{RaHaMiLaLuMcBeMaKe2017} (version
0.13.0+3719.g8e730839), and a Firedrake component called Slate
\citep{GiMiHaCo2020} was used to implement the local solvers for
static condensation and postprocessing.

\begin{table}
  \centering
  \scriptsize \begin{tabular}{*{14}{c}}
\toprule
$r$ & $N$ &
\multicolumn{2}{c}{$ \lVERT \widehat{ P } _h u ^{\mathrm{nor}} - \widehat{u} _h ^{\mathrm{nor}} \rVERT _{ \partial \mathcal{T} _h } $} &
\multicolumn{2}{c}{$ \lVert u - u _h ^\ast \rVert _{ \mathcal{T} _h } $} &
\multicolumn{2}{c}{$ \bigl\lVert \delta (u - u _h ^\ast ) \bigr\rVert _{ \partial \mathcal{T} _h } $} &
\multicolumn{2}{c}{$ \lVERT \widehat{ P } _h \rho ^{\mathrm{nor}} - \widehat{\rho} _h ^{\mathrm{nor}} \rVERT _{ \partial \mathcal{T} _h } $} &
\multicolumn{2}{c}{$ \lVert \rho - \rho _h ^\ast \rVert _{ \mathcal{T} _h }  $} &
\multicolumn{2}{c}{$ \bigl\lVert \delta (\rho - \rho _h ^\ast) \bigr\rVert _{ \mathcal{T} _h }  $} \\
\midrule 0  & 1 & 2.06e-01 & --- & 5.03e-01 & --- & 5.87e-01 & --- & 7.87e-01 & --- & 1.38e+00 & --- & 6.24e+00 & --- \\
 & 2 & 4.67e-01 & -1.2 & 5.27e-01 & -0.1 & 6.32e-01 & -0.1 & 1.56e+00 & -1.0 & 1.27e+00 & 0.1 & 4.70e+00 & 0.4 \\
 & 4 & 3.13e-01 & 0.6 & 2.91e-01 & 0.9 & 1.95e-01 & 1.7 & 9.38e-01 & 0.7 & 6.83e-01 & 0.9 & 2.47e+00 & 0.9 \\
 & 8 & 1.80e-01 & 0.8 & 1.51e-01 & 0.9 & 5.35e-02 & 1.9 & 5.02e-01 & 0.9 & 3.49e-01 & 1.0 & 1.28e+00 & 1.0 \\
 & 16 & 9.42e-02 & 0.9 & 7.66e-02 & 1.0 & 1.38e-02 & 2.0 & 2.57e-01 & 1.0 & 1.76e-01 & 1.0 & 6.46e-01 & 1.0 \\
\midrule 1  & 1 & 1.99e-01 & --- & 3.33e-01 & --- & 3.69e-01 & --- & 1.65e+00 & --- & 1.02e+00 & --- & 3.62e+00 & --- \\
 & 2 & 1.76e-01 & 0.2 & 8.45e-02 & 2.0 & 4.97e-02 & 2.9 & 8.09e-01 & 1.0 & 2.81e-01 & 1.9 & 7.46e-01 & 2.3 \\
 & 4 & 5.82e-02 & 1.6 & 2.56e-02 & 1.7 & 7.93e-03 & 2.6 & 2.44e-01 & 1.7 & 7.70e-02 & 1.9 & 2.16e-01 & 1.8 \\
 & 8 & 1.60e-02 & 1.9 & 6.84e-03 & 1.9 & 1.06e-03 & 2.9 & 6.47e-02 & 1.9 & 1.98e-02 & 2.0 & 5.71e-02 & 1.9 \\
 & 16 & 4.14e-03 & 1.9 & 1.75e-03 & 2.0 & 1.36e-04 & 3.0 & 1.66e-02 & 2.0 & 5.01e-03 & 2.0 & 1.46e-02 & 2.0 \\
\midrule 2  & 1 & 1.09e-01 & --- & 5.68e-02 & --- & 2.01e-02 & --- & 5.14e-01 & --- & 2.10e-01 & --- & 5.65e-01 & --- \\
 & 2 & 4.61e-02 & 1.2 & 1.19e-02 & 2.3 & 4.46e-03 & 2.2 & 2.32e-01 & 1.1 & 5.06e-02 & 2.1 & 1.09e-01 & 2.4 \\
 & 4 & 7.16e-03 & 2.7 & 1.52e-03 & 3.0 & 2.84e-04 & 4.0 & 3.52e-02 & 2.7 & 6.67e-03 & 2.9 & 1.20e-02 & 3.2 \\
 & 8 & 9.68e-04 & 2.9 & 1.92e-04 & 3.0 & 1.77e-05 & 4.0 & 4.71e-03 & 2.9 & 8.41e-04 & 3.0 & 1.44e-03 & 3.1 \\
 & 16 & 1.25e-04 & 3.0 & 2.42e-05 & 3.0 & 1.11e-06 & 4.0 & 6.05e-04 & 3.0 & 1.05e-04 & 3.0 & 1.78e-04 & 3.0 \\
\bottomrule
\end{tabular}

  \bigskip
  \caption{Errors and rates for a manufactured solution with
    $ n = 3 $, $ k = 1 $, using hybridization with
    $ \mathcal{P} _{r+1} ^- \Lambda ^0 \cong \mathtt{CG} _{ r + 1 } $
    and
    $ \mathcal{P} _{ r + 1 } ^- \Lambda ^1 \cong \mathtt{N1E} _{r+1} $
    elements and local postprocessing with broken
    $ \star \mathcal{P} _{ r + 2 } ^- \Lambda ^1 \cong \mathtt{N1E} _{
      r + 2 } $ and
    $ \star \mathcal{P} _{ r + 2 } ^- \Lambda ^2 \cong \mathtt{N1F} _{
      r + 2 } $ elements. Since $ k < n -1 $, we get improved
    convergence of $ \delta \rho _h ^\ast $ but not
    $ \widehat{ \rho } _h ^{\mathrm{nor}} $ or $ \rho _h ^\ast $.\label{tab:n=3_k=1}}
\end{table}

\begin{table}
  \centering
  \scriptsize \begin{tabular}{*{14}{c}}
\toprule
$r$ & $N$ &
\multicolumn{2}{c}{$ \lVERT \widehat{ P } _h u ^{\mathrm{nor}} - \widehat{u} _h ^{\mathrm{nor}} \rVERT _{ \partial \mathcal{T} _h } $} &
\multicolumn{2}{c}{$ \lVert u - u _h ^\ast \rVert _{ \mathcal{T} _h } $} &
\multicolumn{2}{c}{$ \bigl\lVert \delta (u - u _h ^\ast ) \bigr\rVert _{ \partial \mathcal{T} _h } $} &
\multicolumn{2}{c}{$ \lVERT \widehat{ P } _h \rho ^{\mathrm{nor}} - \widehat{\rho} _h ^{\mathrm{nor}} \rVERT _{ \partial \mathcal{T} _h } $} &
\multicolumn{2}{c}{$ \lVert \rho - \rho _h ^\ast \rVert _{ \mathcal{T} _h }  $} &
\multicolumn{2}{c}{$ \bigl\lVert \delta (\rho - \rho _h ^\ast) \bigr\rVert _{ \mathcal{T} _h }  $} \\
\midrule 0  & 1 & 7.19e-01 & --- & 6.97e-01 & --- & 2.14e+00 & --- & 1.28e+00 & --- & 1.97e+00 & --- & 1.31e+01 & --- \\
 & 2 & 5.98e-01 & 0.3 & 5.07e-01 & 0.5 & 1.75e+00 & 0.3 & 1.17e+00 & 0.1 & 7.57e-01 & 1.4 & 9.79e+00 & 0.4 \\
 & 4 & 3.30e-01 & 0.9 & 2.62e-01 & 0.9 & 9.51e-01 & 0.9 & 3.98e-01 & 1.6 & 2.19e-01 & 1.8 & 5.32e+00 & 0.9 \\
 & 8 & 1.74e-01 & 0.9 & 1.33e-01 & 1.0 & 4.90e-01 & 1.0 & 1.15e-01 & 1.8 & 5.78e-02 & 1.9 & 2.72e+00 & 1.0 \\
 & 16 & 8.84e-02 & 1.0 & 6.66e-02 & 1.0 & 2.47e-01 & 1.0 & 3.02e-02 & 1.9 & 1.47e-02 & 2.0 & 1.37e+00 & 1.0 \\
\midrule 1  & 1 & 6.54e-01 & --- & 5.55e-01 & --- & 1.72e+00 & --- & 2.97e+00 & --- & 9.23e-01 & --- & 9.40e+00 & --- \\
 & 2 & 2.83e-01 & 1.2 & 1.59e-01 & 1.8 & 4.69e-01 & 1.9 & 3.97e-01 & 2.9 & 2.21e-01 & 2.1 & 3.66e+00 & 1.4 \\
 & 4 & 8.64e-02 & 1.7 & 4.26e-02 & 1.9 & 1.24e-01 & 1.9 & 5.31e-02 & 2.9 & 3.07e-02 & 2.8 & 1.01e+00 & 1.9 \\
 & 8 & 2.32e-02 & 1.9 & 1.09e-02 & 2.0 & 3.15e-02 & 2.0 & 6.79e-03 & 3.0 & 3.94e-03 & 3.0 & 2.59e-01 & 2.0 \\
 & 16 & 5.97e-03 & 2.0 & 2.74e-03 & 2.0 & 7.93e-03 & 2.0 & 8.60e-04 & 3.0 & 4.97e-04 & 3.0 & 6.51e-02 & 2.0 \\
\midrule 2  & 1 & 2.44e-01 & --- & 2.25e-01 & --- & 6.60e-01 & --- & 3.66e-01 & --- & 6.63e-01 & --- & 6.18e+00 & --- \\
 & 2 & 8.47e-02 & 1.5 & 4.16e-02 & 2.4 & 1.02e-01 & 2.7 & 8.26e-02 & 2.1 & 5.40e-02 & 3.6 & 1.08e+00 & 2.5 \\
 & 4 & 1.29e-02 & 2.7 & 5.71e-03 & 2.9 & 1.33e-02 & 2.9 & 5.51e-03 & 3.9 & 3.75e-03 & 3.8 & 1.50e-01 & 2.8 \\
 & 8 & 1.73e-03 & 2.9 & 7.31e-04 & 3.0 & 1.68e-03 & 3.0 & 3.54e-04 & 4.0 & 2.41e-04 & 4.0 & 1.92e-02 & 3.0 \\
 & 16 & 2.22e-04 & 3.0 & 9.20e-05 & 3.0 & 2.11e-04 & 3.0 & 2.24e-05 & 4.0 & 1.51e-05 & 4.0 & 2.42e-03 & 3.0 \\
\bottomrule
\end{tabular}

  \bigskip
  \caption{Errors and rates for a manufactured solution with
    $ n = 3 $, $ k = 2 $, using hybridization with
    $ \mathcal{P} _{r+1} ^- \Lambda ^1 \cong \mathtt{N1E} _{ r + 1 } $
    and
    $ \mathcal{P} _{ r + 1 } ^- \Lambda ^2 \cong \mathtt{N1F} _{r+1} $
    elements and local postprocessing with broken
    $ \star \mathcal{P} _{ r + 1 } ^- \Lambda ^0 \cong \mathtt{CG} _{
      r + 1 } $ and
    $ \star \mathcal{P} _{ r + 1 } ^- \Lambda ^1 \cong \mathtt{N1E} _{
      r + 1 } $ elements. Since $ k = n -1 $, we get
    superconvergence of $ \widehat{ \rho } _h ^{\mathrm{nor}} $ and
    $ \rho _h ^\ast $ as compared with
    $ \rho _h = \mathrm{d} u _h $.\label{tab:n=3_k=2}}
\end{table}

\subsection{Test problems}
On the unit cube $ \Omega = [ 0, 1 ] ^3 $, a structured tetrahedral
mesh $ \mathcal{T} _h $ is formed by partitioning $\Omega$ into
$ N \times N \times N $ cubes, each of which is divided into six
tetrahedra. As in \cref{sec:examples}, we identify
$ H \Lambda (\Omega) $ and $ H ^\ast \Lambda (\Omega) $ with the
complexes of scalar and vector proxy fields. We use the ``method of
manufactured solutions'' by choosing a smooth $u$ satisfying the
boundary conditions, taking $ f = - \Delta u $, and applying the
numerical method to this $f$. For $ k = 1 $, we choose
\begin{equation*}
  u(x,y,z) =
  \begin{bmatrix}
    \sin ( \pi x ) \\
    \sin ( \pi y ) \\
    \sin ( \pi z )
  \end{bmatrix} +
  \begin{bmatrix}
    \hphantom{-} \sin ( \pi x ) \cos ( \pi y ) \\
    - \cos ( \pi x ) \sin ( \pi y ) \\
    0
  \end{bmatrix},
\end{equation*} 
where the first term is in $ \mathfrak{B} ^1 $ and the second is in
$ \mathring{ \mathfrak{B} } ^\ast _1 $. For $ k = 2 $, we choose
\begin{equation*}
  u(x, y, z ) =
  \begin{bmatrix}
    \sin ( \pi y ) \sin ( \pi z ) \\
    \sin ( \pi x ) \sin ( \pi z ) \\
    \sin ( \pi x ) \sin ( \pi y ) 
  \end{bmatrix} +
  \begin{bmatrix}
    \cos ( \pi x ) \sin ( \pi y ) \sin ( \pi z ) \\
    \sin ( \pi x ) \cos ( \pi y ) \sin ( \pi z ) \\
    \sin ( \pi x ) \sin ( \pi y ) \cos ( \pi z ) 
  \end{bmatrix},
\end{equation*}
where the first term is in $ \mathfrak{B} ^2 $ and the second is in
$ \mathring{ \mathfrak{B} } ^\ast _2 $.

\subsection{Results}

\cref{tab:n=3_k=1} shows the errors and rates for the $ k = 1 $
problem, using $ \mathcal{P} _{r+1} ^- \Lambda $ elements and
$ \star \mathcal{P} _{ r + 2 } ^- \Lambda $ postprocessing. (Since
$ \mathcal{P} _{ r + 1 } ^- \Lambda ^0 \cong \mathcal{P} _{ r + 1 }
\Lambda ^0 $, the minimum degree satisfying \cref{assumption} is
$ r ^\ast = r + 1 $.)  \cref{tab:n=3_k=2} shows the errors and rates
for the $ k = 2 $ problem, using $ \mathcal{P} _{r+1} ^- \Lambda $
elements and $ \star \mathcal{P} _{ r + 1 } ^- \Lambda $
postprocessing. For clarity, the captions describe the elements both
in FEEC notation and in terms of their classical scalar and vector
proxies. Adopting the Unified Form Language (UFL)
\citep{AlLoOlRoWe2014} notation used by Firedrake, we denote Lagrange
finite elements by \texttt{CG}, N\'ed\'elec
$ H ( \operatorname{curl} ) $ edge elements of the first kind by
\texttt{N1E}, and N\'ed\'elec $ H ( \operatorname{div} ) $ face
elements of the first kind by \texttt{N1F}.

These results match the error estimates in
\cref{sec:error_estimates,sec:postprocessing_estimates}. Specifically,
when $ k = 1 < n - 1 $, we do not get superconvergence of
$ \widehat{ \rho } _h ^{\mathrm{nor}} $ or $ \rho _h ^\ast $: both
converge with the same rate $ \mathcal{O} ( h ^{ r + 1 } ) $ as
$ \rho _h = \mathrm{d} u _h $. However, $ \delta \rho _h ^\ast $
converges with improved rate $ \mathcal{O} ( h ^{ r + 1 } ) $,
compared with $ \mathcal{O} ( h ^r ) $ for $ \delta \rho _h $. On the
other hand, when $ k = 2 = n - 1 $, we see that
$ \widehat{ \rho } _h ^{\mathrm{nor}} $ and $ \rho _h ^\ast $ both
superconverge with rate $ \mathcal{O} ( h ^{ r + 2 } ) $.

\section{A view toward HDG methods for finite element exterior
  calculus}
\label{sec:hdg}

In this last section, we briefly present an even more general approach
to domain decomposition and hybrid methods for the Hodge--Laplace
problem. This includes hybridization of the conforming FEEC methods we
have discussed so far, as well as nonconforming and HDG methods. In
the cases $ k = 0 $ and $ k = n $, we recover the unified
hybridization framework of \citet*{CoGoLa2009} for the scalar Poisson
equation. When $ n = 3 $, the cases $ k = 1 $ and $ k = 2 $ include
some recently proposed HDG methods for the vector Poisson equation and
Maxwell's equations. Although we lay out the framework here, we
postpone a detailed discussion and analysis of these methods for
future work.

\subsection{Variational principle}

To motivate the variational principle for these more general methods,
we begin with a new formulation of the exact local solvers for the
Hodge--Laplace problem. Given $ \widehat{ \sigma } ^{\mathrm{tan}} $,
$ \widehat{ u } ^{\mathrm{tan}} $ on $ \partial K $,
$ \overline{ u } \in \mathring{ \mathfrak{H} } ^k (K) $, and
$ p \in \mathfrak{H} ^k $, observe that the exact solution satisfies
\begin{alignat*}{2}
  ( \sigma, \tau ) _K - ( u , \mathrm{d} \tau ) _K + \langle u
  ^{\mathrm{nor}} , \tau ^{\mathrm{tan}} \rangle _{ \partial K } &= 0,
  \quad &\forall \tau &\in H \Lambda ^{k-1} (K) \cap H ^\ast \Lambda
  ^{k-1} (K) ,\\
  ( \sigma, \delta v ) _K + ( \rho , \mathrm{d} v ) _K + ( \overline{p} , v ) _K - \langle \rho ^{\mathrm{nor}} , v ^{\mathrm{tan}} \rangle _{ \partial K } &= \rlap{$( f - p , v ) _K - \langle \widehat{ \sigma } ^{\mathrm{tan}} , v ^{\mathrm{nor}} \rangle _{ \partial K }$,} \\
  &&\forall v &\in H \Lambda ^k (K) \cap H ^\ast \Lambda ^k (K) ,\\
  (\rho, \eta ) _K - ( u, \delta \eta ) _K &= \langle \widehat{ u }
  ^{\mathrm{tan}} , \eta ^{\mathrm{nor}} \rangle _{ \partial K } ,
  \quad &\forall \eta &\in H \Lambda ^{ k + 1 } (K) \cap H ^\ast \Lambda ^{ k + 1 } (K) ,\\
  ( u , \overline{q} ) _K &= ( \overline{ u } , \overline{q} ) _K , \quad &\forall \overline{q} &\in \mathring{ \mathfrak{H}  } ^k (K) .
\end{alignat*}
Here, both $ \mathrm{d} $ and $ \delta $ are taken weakly, as they are
only applied to test functions.

Now, suppose we choose finite element spaces
$ W _h ^k (K) \subset H \Lambda ^k (K) \cap H ^\ast \Lambda ^k (K) $
for each $ K \in \mathcal{T} _h $, giving the broken space
$ W _h ^k \coloneqq \prod _{ K \in \mathcal{T} _h } W _h ^k (K) $, and
likewise for $ W _h ^{ k \pm 1 } $. Suppose we also choose unbroken
spaces
$ \widehat{ V } _h ^{k -1 , \mathrm{tan}} \subset \widehat{ V } ^{ k
  -1 , \mathrm{tan} } $ and
$ \widehat{ V } _h ^{k, \mathrm{tan}} \subset \widehat{ V } ^{ k,
  \mathrm{tan} } $, which do not necessarily correspond to tangential
traces of $ W _h ^{ k -1 } $ and $ W _h ^k $. Then we consider the
variational problem: Find
\begin{align*} 
  \tag{local variables} \sigma _h &\in W _h ^{ k -1 } , & u _h &\in W _h  ^k ,  & \rho _h & \in W _h ^{ k + 1 } , & \overline{p} _h &\in \overline{ \mathfrak{H}  } _h ^k , \\
  \tag{global variables} p _h &\in \mathfrak{H} _h ^k ,& \overline{ u } _h &\in \overline{ \mathfrak{H}  } _h ^k , & \widehat{ \sigma } _h ^{\mathrm{tan}} &\in \widehat{ V } _h ^{k-1, \mathrm{tan}} , & \widehat{ u } _h ^{\mathrm{tan}} &\in \widehat{ V } _h ^{k, \mathrm{tan}} ,
\end{align*} 
satisfying
\begin{subequations}
  \label{eqn:feec-hdg}
  \begin{alignat}{2}
    (\sigma _h , \tau _h ) _{ \mathcal{T} _h } - ( u _h , \mathrm{d} \tau _h ) _{ \mathcal{T} _h } + \langle \widehat{ u } _h ^{\mathrm{nor}} , \tau _h ^{\mathrm{tan}} \rangle _{ \partial \mathcal{T} _h } &= 0, \quad &\forall \tau _h &\in W _h ^{ k -1 } ,\\
    \begin{multlined}[b]
      ( \sigma _h , \delta v _h ) _{ \mathcal{T} _h } + ( \rho _h , \mathrm{d} v _h ) _{ \mathcal{T} _h } + ( \overline{p} _h + p _h , v _h ) _{ \mathcal{T} _h } \\
      + \langle \widehat{ \sigma } _h ^{\mathrm{tan}} , v _h ^{\mathrm{nor}} \rangle _{ \partial \mathcal{T} _h } - \langle \widehat{ \rho } _h ^{\mathrm{nor}} , v _h ^{\mathrm{tan}} \rangle _{ \partial \mathcal{T} _h }
    \end{multlined}
    &= ( f, v _h ) _{ \mathcal{T} _h } ,\quad &\forall v _h &\in W _h ^k , \label{eqn:feec-hdg_v}\\
    ( \rho _h , \eta _h ) _{ \mathcal{T} _h } - ( u _h , \delta \eta _h ) _{ \mathcal{T} _h } - \langle \widehat{ u } _h ^{\mathrm{tan}} , \eta _h ^{\mathrm{nor}} \rangle _{ \partial \mathcal{T} _h } &= 0, \quad &\forall \eta _h &\in W _h ^{ k + 1 } , \label{eqn:feec-hdg_eta}\\
    ( \overline{ u } _h - u _h , \overline{q} _h ) _{ \mathcal{T} _h } &= 0 , \quad &\forall \overline{q} _h &\in \overline{ \mathfrak{H}  } _h ^k ,\\
    ( u _h , q _h ) _{ \mathcal{T} _h } &= 0, \quad &\forall q _h &\in \mathfrak{H}  _h ^k ,\\
    (\overline{p} _h , \overline{ v } _h ) _{ \mathcal{T} _h } &= 0, \quad &\forall \overline{ v } _h &\in \overline{ \mathfrak{H}  } _h ^k ,\\
    \langle \widehat{ u } _h ^{\mathrm{nor}} , \widehat{ \tau } _h ^{\mathrm{tan}} \rangle _{ \partial \mathcal{T} _h } &= 0 , \quad &\forall \widehat{ \tau } _h ^{\mathrm{tan}} &\in \widehat{ V } _h ^{k-1, \mathrm{tan}} ,\\
    \langle \widehat{ \rho } _h ^{\mathrm{nor}} , \widehat{ v } _h ^{\mathrm{tan}} \rangle _{ \partial \mathcal{T} _h } &= 0 , \quad &\forall \widehat{ v } _h ^{\mathrm{tan}} &\in \widehat{ V } _h ^{k, \mathrm{tan}} .
  \end{alignat}
\end{subequations}
To complete the specification of the problem, one must define the
approximate normal traces $ \widehat{ u } _h ^{\mathrm{nor}} $ and
$ \widehat{ \rho } _h ^{\mathrm{nor}} $, which play the same role as
the ``numerical flux'' does in \citep{CoGoLa2009}. The discrete
harmonic spaces $ \overline{ \mathfrak{H} } _h ^k $ and
$ \mathfrak{H} _h ^k $ are then defined so that the local and global
solvers have unique solutions.

\begin{remark}
For the scalar Poisson
equation, we recover the unified hybridization framework of
\citep{CoGoLa2009}. If $ k = 0 $, then in terms of scalar and vector
proxies, \eqref{eqn:feec-hdg} simplifies to
\begin{alignat*}{2}
  ( \rho _h , \operatorname{grad} v _h ) _{ \mathcal{T} _h } + ( p _h , v _h ) _{ \mathcal{T} _h } - \langle \widehat{ \rho } _h ^{\mathrm{nor}} , v _h \rangle _{ \partial \mathcal{T} _h } &= ( f, v _h ) _{ \mathcal{T} _h } ,\quad &\forall v _h &\in W _h ^0 ,\\
  ( \rho _h , \eta _h ) _{ \mathcal{T} _h } + ( u _h , \operatorname{div} \eta _h ) _{ \mathcal{T} _h } - \langle \widehat{ u } _h ^{\mathrm{tan}} , \eta _h \cdot \normal \rangle _{ \partial \mathcal{T} _h } &= 0, \quad &\forall \eta _h &\in W _h ^1 ,\\
  ( u _h , q _h ) _{ \mathcal{T} _h } &= 0, \quad &\forall q _h &\in \mathfrak{H}  _h ^0 ,\\
  \langle \widehat{ \rho } _h ^{\mathrm{nor}} , \widehat{ v } _h ^{\mathrm{tan}} \rangle _{ \partial \mathcal{T} _h } &= 0 , \quad &\forall \widehat{ v } _h ^{\mathrm{tan}} &\in \widehat{ V } _h ^{0, \mathrm{tan}} ,
\end{alignat*}
which gives the methods of \citep{CoGoLa2009} for the Neumann problem,
using local Dirichlet solvers.  Alternatively, if $ k = n $, and each
$ K \in \mathcal{T} _h $ is connected (e.g., simplicial), then
\eqref{eqn:feec-hdg} becomes
\begin{alignat*}{2}
  (\sigma _h , \tau _h ) _{ \mathcal{T} _h } - ( u _h , \operatorname{div} \tau _h ) _{ \mathcal{T} _h } + \langle \widehat{ u } _h ^{\mathrm{nor}} , \tau _h \rangle _{ \partial \mathcal{T} _h } &= 0, \quad &\forall \tau _h &\in W _h ^{ n -1 } ,\\
  -( \sigma _h , \operatorname{grad} v _h ) _{ \mathcal{T} _h } + ( \overline{p} _h , v _h ) _{ \mathcal{T} _h } + \langle \widehat{ \sigma } _h ^{\mathrm{tan}} , v _h \normal \rangle _{ \partial \mathcal{T} _h } &= ( f, v _h ) _{ \mathcal{T} _h } ,\quad &\forall v _h &\in W _h ^n ,\\
  ( \overline{ u } _h - u _h , \overline{q} _h ) _{ \mathcal{T} _h } &= 0 , \quad &\forall \overline{q} _h &\in \mathbb{R}  ^{ \mathcal{T} _h }  ,\\
  (\overline{p} _h , \overline{ v } _h ) _{ \mathcal{T} _h } &= 0, \quad &\forall \overline{ v } _h &\in \mathbb{R} ^{ \mathcal{T} _h }  ,\\
  \langle \widehat{ u } _h ^{\mathrm{nor}} , \widehat{ \tau } _h ^{\mathrm{tan}} \rangle _{ \partial \mathcal{T} _h } &= 0 , \quad &\forall \widehat{ \tau } _h ^{\mathrm{tan}} &\in \widehat{ V } _h ^{k-1, \mathrm{tan}} ,
\end{alignat*}
which is the alternative hybridization of \citet[Section
5]{Cockburn2016} using local Neumann solvers.
\end{remark}

\subsection{Examples of methods}

Different choices of the finite element spaces and approximate normal
traces in \eqref{eqn:feec-hdg} yield different families of methods. We
now discuss a few specific examples.

\subsubsection{The hybridized FEEC methods}

Suppose we choose the spaces $ W _h $ and $ \widehat{ V } _h $ as in
\cref{sec:hybrid}. We then define
$ \widehat{ u } _h ^{\mathrm{nor}} \in \widehat{ W } _h ^{k -1 ,
  \mathrm{nor}} $ and
$ \widehat{ \rho } _h ^{\mathrm{nor}} \in \widehat{ W } _h ^{k ,
  \mathrm{nor}} $ to be new unknown variables, which are determined by
augmenting \eqref{eqn:feec-hdg} by the equations
\begin{alignat*}{2}
  \langle \widehat{ \sigma } _h ^{\mathrm{tan}} - \sigma _h ^{\mathrm{tan}} , \widehat{ v } _h ^{\mathrm{nor}} \rangle _{ \partial \mathcal{T} _h } &= 0, \quad &\forall \widehat{ v } _h ^{\mathrm{nor}} &\in \widehat{ W } _h ^{k-1, \mathrm{nor}}, \tag{\ref{eqn:feec-h_vnor}}\\
  \langle \widehat{ u } _h ^{\mathrm{tan}} - u _h ^{\mathrm{tan}} , \widehat{ \eta } _h ^{\mathrm{nor}} \rangle _{ \partial \mathcal{T} _h } &= 0, \quad &\forall \widehat{ \eta } _h ^{\mathrm{nor}} &\in \widehat{ W } _h ^{k, \mathrm{nor}} \tag{\ref{eqn:feec-h_etanor}}.
\end{alignat*}
Using these, \eqref{eqn:feec-hdg_v} and \eqref{eqn:feec-hdg_eta}
become equivalent to \eqref{eqn:feec-h_v} and
$ \rho _h = \mathrm{d} u _h $, respectively. Hence, the variational
problem is equivalent to \eqref{eqn:feec-h}, so we recover the
hybridized FEEC methods of \cref{sec:hybrid}.

\subsubsection{Mixed and nonconforming hybrid methods}
\label{sec:mixed-nonconforming}

Suppose we take
$ \widehat{ u } _h ^{\mathrm{nor}} = u _h ^{\mathrm{nor}} $ and
$ \widehat{ \rho } _h ^{\mathrm{nor}} = \rho _h ^{\mathrm{nor}}
$. Then, using integration by parts, \eqref{eqn:feec-hdg} simplifies
to
\begin{alignat*}{2}
  ( \delta u _h , \delta v _h ) _{ \mathcal{T} _h } + ( \delta \rho _h , v _h ) _{ \mathcal{T} _h } + ( \overline{p} _h + p _h , v _h ) _{ \mathcal{T} _h } + \langle \widehat{ \sigma } _h ^{\mathrm{tan}} , v _h ^{\mathrm{nor}} \rangle _{ \partial \mathcal{T} _h } &= ( f, v _h ) _{ \mathcal{T} _h } ,\quad &\forall v _h &\in W _h ^k ,\\
  ( \rho _h , \eta _h ) _{ \mathcal{T} _h } - ( u _h , \delta \eta _h ) _{ \mathcal{T} _h } - \langle \widehat{ u } _h ^{\mathrm{tan}} , \eta _h ^{\mathrm{nor}} \rangle _{ \partial \mathcal{T} _h } &= 0, \quad &\forall \eta _h &\in W _h ^{ k + 1 } ,\\
  ( \overline{ u } _h - u _h , \overline{q} _h ) _{ \mathcal{T} _h } &= 0 , \quad &\forall \overline{q} _h &\in \overline{ \mathfrak{H}  } _h ^k ,\\
  ( u _h , q _h ) _{ \mathcal{T} _h } &= 0, \quad &\forall q _h &\in \mathfrak{H}  _h ^k ,\\
  (\overline{p} _h , \overline{ v } _h ) _{ \mathcal{T} _h } &= 0, \quad &\forall \overline{ v } _h &\in \overline{ \mathfrak{H}  } _h ^k ,\\
  \langle u _h ^{\mathrm{nor}} , \widehat{ \tau } _h ^{\mathrm{tan}} \rangle _{ \partial \mathcal{T} _h } &= 0 , \quad &\forall \widehat{ \tau } _h ^{\mathrm{tan}} &\in \widehat{ V } _h ^{k-1, \mathrm{tan}} ,\\
  \langle \rho _h ^{\mathrm{nor}} , \widehat{ v } _h ^{\mathrm{tan}} \rangle _{ \partial \mathcal{T} _h } &= 0 , \quad &\forall \widehat{ v } _h ^{\mathrm{tan}} &\in \widehat{ V } _h ^{k, \mathrm{tan}} ,
\end{alignat*}
and $ \sigma _h = \delta u _h $.
When $ k = 0 $, we obtain mixed hybrid methods for the Neumann problem
using local Dirichlet solvers, including the classic hybridized RT and
BDM methods \citep{ArBr1985,BrDoMa1985}. When $ k = n $, we obtain
primal hybrid methods for the Dirichlet problem using local Neumann
solvers, including the nonconforming hybrid method of
\citet{RaTh1977_hybrid}.

\subsubsection{Hybridizable discontinuous Galerkin methods} Suppose we
take
\begin{equation*}
  \widehat{ u } _h ^{\mathrm{nor}} = u _h ^{\mathrm{nor}} - \lambda ( \widehat{ \sigma } _h ^{\mathrm{tan}} - \sigma _h ^{\mathrm{tan}} ) , \qquad \widehat{ \rho } _h ^{\mathrm{nor}} = \rho _h ^{\mathrm{nor}} + \mu ( \widehat{ u } _h ^{\mathrm{tan}} - u _h ^{\mathrm{tan}} ) ,
\end{equation*}
where $\lambda$ and $\mu$ are penalty functions on
$ \partial \mathcal{T} _h $. \cref{sec:mixed-nonconforming}
corresponds to the case $ \lambda = \mu = 0 $, while the hybridized
FEEC methods of \cref{sec:hybrid} can be seen as the limiting case
$ \lambda, \mu \rightarrow \infty $.

When $ k = 0 $, \eqref{eqn:feec-hdg} becomes the hybrid local
discontinuous Galerkin (LDG-H) method of \citep{CoGoLa2009}, while
$ k = n $ gives the alternative implementation of \citep[Section
5]{Cockburn2016} using local Neumann solvers. For the vector Poisson
equation when $ n = 2 $ or $ n = 3 $, \eqref{eqn:feec-hdg} corresponds
to the recent HDG methods of \citet*{NgPeCo2011,ChQiShSo2017}, which
have been applied to Maxwell's equations. Since the initial appearance
of the current manuscript as a preprint, \citet*{HoLiXu2021} have
analyzed several methods of this type for general $k$ and $n$ within
the extended Galerkin (XG) framework.

Finally, a different family of HDG methods may be constructed by
taking
\begin{equation*}
  \widehat{ u } _h ^{\mathrm{nor}} = u _h ^{\mathrm{nor}} - \lambda \bigl(  \widehat{ \sigma } _h ^{\mathrm{tan}} - (\delta u _h) ^{\mathrm{tan}} \bigr)  , \qquad \widehat{ \rho } _h ^{\mathrm{nor}} = \mathrm{d} u _h ^{\mathrm{nor}} + \mu ( \widehat{ u } _h ^{\mathrm{tan}} - u _h ^{\mathrm{tan}} ) ,
\end{equation*}
which generalizes the hybrid interior penalty (IP-H) method of
\citep{CoGoLa2009}.

\section*{Acknowledgments}

Gerard Awanou, Johnny Guzman, and Ari Stern would like to thank the
Isaac Newton Institute for Mathematical Sciences for support and
hospitality during the program ``Geometry, compatibility and structure
preservation in computational differential equations,'' when work on
this paper was undertaken. This program was supported by EPSRC grant
number EP/R014604/1.

Gerard Awanou was supported by NSF grant DMS-1720276, Johnny Guzman by
NSF grants DMS-1620100 and DMS-1913083, and Ari Stern by NSF grant
DMS-1913272.

Finally, Ari Stern wishes to thank Mary Barker for assisting with
preliminary numerical experiments and Colin Cotter, Thomas Gibson, and
David Ham for help with Firedrake.

\end{document}